\documentclass{amsart}

\usepackage{color}
\usepackage{graphicx}
\usepackage{hyperref}
\usepackage{amssymb}
\usepackage{amsfonts}
\usepackage{euscript}
\usepackage[all]{xy}
\usepackage{bbm}

\numberwithin{equation}{section}
\numberwithin{figure}{section}
\renewcommand{\subsection}[1]{\hspace{-\parindent}\refstepcounter{subsection}{\bf (\arabic{section}\alph{subsection}) #1.}\addcontentsline{toc}{subsection}{\bf #1.}}

\newenvironment{nouppercase}{%
  \renewcommand{\uppercasenonmath}[1]{}}{}
\newcommand{\mybox}[1]{\parbox[t]{37.5em}{#1}}
\allowdisplaybreaks

\theoremstyle{plain}
\newtheorem{thm}{Theorem}[section]
\newtheorem{theorem}[thm]{Theorem}

\newtheorem{cor}[thm]{Corollary}
\newtheorem{corollary}[thm]{Corollary}

\newtheorem{prop}[thm]{Proposition}

\newtheorem{remark}[thm]{Remark}

\newtheorem{proposition}[thm]{Proposition}

\newtheorem{lemma}[thm]{Lemma}

\newtheorem{setup}[thm]{Setup}

\newtheorem*{claim*}{Claim} 
\newtheorem*{lemma*}{Lemma}
\newtheorem*{theorem*}{Theorem}
\newtheorem*{conjecture*}{Conjecture}

\newcommand{\bC}{{\mathbb C}}

\newcommand{\bR}{{\mathbb R}}

\newcommand{\bZ}{{\mathbb Z}}

\newcommand{\scrA}{\EuScript A}
\newcommand{\scrB}{\EuScript B}

\newcommand{\scrD}{\EuScript D}

\newcommand{\scrG}{\EuScript G}

\newcommand{\scrP}{\EuScript P}
\newcommand{\scrQ}{\EuScript Q}
\newcommand{\scrR}{\EuScript R}

\newcommand{\frakg}{\mathfrak{g}}

\newcommand{\half}{{\textstyle\frac{1}{2}}}

\newcommand{\iso}{\cong}
\newcommand{\htp}{\simeq}
\newcommand{\smooth}{C^\infty}

\renewcommand{\ge}{{\scriptscriptstyle\geq 0}}
\newcommand{\gr}{{\scriptscriptstyle >0}}

\renewcommand{\hom}{\mathit{hom}}

\newcommand{\re}{\mathrm{re}}
\newcommand{\im}{\mathrm{im}}

\title[LEFSCHETZ FIBRATIONS]{\Large\larger\rm Fukaya $A_\infty$-structures associated to\\ Lefschetz fibrations. V}
\author{Paul Seidel}

\begin{document}
\begin{nouppercase}
\maketitle
\end{nouppercase}

\begin{abstract}
We (re)consider how the Fukaya category of a Lefschetz fibration is related to that of the fibre. The distinguishing feature of the approach here is a more direct identification of the bimodule homomorphism involved.
\end{abstract}

\section{Introduction}
The subject of this paper is an idea originally proposed in \cite[Conjecture 4]{seidel06}, concerning the relation between the Fukaya $A_\infty$-category of a Lefschetz fibration and that of its fibre. That idea is central to the series of articles \cite{seidel12b, seidel14b, seidel15, seidel16, seidel17, seidel18, seidel18b} to which this one belongs. Indeed, a first answer was given in \cite{seidel12b}, and a much more comprehensive one in \cite{seidel18b} (the author apologizes for the anachronistic numbering; \cite{seidel18b} was completed before the present paper). However, in those constructions one salient feature of the situation, namely the special role of the dual diagonal $A_\infty$-bimodule, appears in a post hoc way (one first introduces a geometrically defined bimodule, and then later identifies it as quasi-isomorphic to the dual diagonal). Here, we take a modified route which builds in the dual diagonal bimodule from the outset. Independently of that, we also take advantage of the simpler definition of Fukaya categories of Lefschetz fibrations from \cite{seidel18}. Finally, our construction has larger scope than the previous ones, since we'll allow (a particular class of) Fukaya categories with curvature terms.

\subsection{Poincar{\'e} duality\label{subsec:poincare}}
Our thinking is informed by classical Poincar{\'e} duality for manifolds with boundary (the idea of using this as a model already occurs in \cite{seidel06, kontsevich-vlassopoulos13}). Namely, let $E^n$ be a compact oriented manifold with boundary $M = \partial E$. Let's consider the associated cochain complexes (leaving the precise meaning open; secretly, we are thinking of Morse theory). We have a short exact sequence
\begin{equation} \label{eq:classical-les}
0 \rightarrow C^*(E,M) \xrightarrow{\text{inclusion}} C^*(E) \xrightarrow{\text{restriction}} C^*(M) \rightarrow 0,
\end{equation}
together with a (Poincar{\'e} duality) homotopy equivalence
\begin{equation} \label{eq:poincare}
C^*(E,M) \stackrel{\htp}{\longrightarrow} C^*(E)^\vee[-n].
\end{equation}
Let's write $\scrA = C^*(E)$, $\scrB = C^*(M)$ for the complexes, and $\rho: \scrA \rightarrow \scrB$ for the restriction map. Inverting \eqref{eq:poincare} and composing with inclusion yields a chain map $\delta: \scrA^\vee[-n] \rightarrow \scrA$. Altogether, the outcome is that \eqref{eq:classical-les} now appears as an acyclic complex of the form
\begin{equation} \label{eq:total-chains}
\big\{ \scrA^\vee[-n] \stackrel{\delta}{\longrightarrow} \scrA \stackrel{\rho}{\longrightarrow} \scrB \big\}.
\end{equation}
The drawback of this approach is the indirect construction of $\delta$. Instead, it is preferable to obtain it directly from a geometrically defined diagonal cocycle
\begin{equation} \label{eq:chain-diagonal}
\delta \in (\scrA \otimes \scrA)^n.
\end{equation}
However, with such a modified definition, the two maps in \eqref{eq:total-chains} may no longer strictly compose to zero, thus requiring the introduction of an auxiliary nullhomotopy to build our acyclic complex.

\subsection{The result}
We now switch to the actual situation under consideration (in particular, the meaning of several notations will change at this point). Take a symplectic manifold $E^{2n}$ with a proper map
\begin{equation} \label{eq:pi}
p: E \longrightarrow \bR^2.
\end{equation}
Outside a compact subset, $p$ should be a locally trivial symplectic fibration. We require that $c_1(E) = 0$. We also require that the Poincar{\'e} dual of the symplectic form is represented by a symplectic submanifold $\Omega_E \subset E$, such that $p|\Omega_E$ has no critical points, and $\Omega_E$ is preserved by symplectic parallel transport for $p$. Let $M$ be the fibre of $p$ at a point close to infinity on $\bR^2$, with its symplectic submanifold $\Omega_M = \Omega_E \cap M$. Write $\scrB_q$ for the relative Fukaya category of $(M,\Omega_M)$, defined as in \cite{sheridan11b, perutz-sheridan20}. The objects of that category are certain closed Lagrangian submanifolds in $M \setminus \Omega_M$, but its structure maps involve discs that intersect $\Omega_M$. The resulting $\scrB_q$ is a curved $A_\infty$-category over $\bZ[[q]]$. For the Fukaya category of $E$, one wants to consider noncompact Lagrangian submanifolds which go to infinity in an appropriately constrained way, modelled on the classical theory of thimbles for Lefschetz fibrations. There is again a relative version, using Lagrangians in $E \setminus \Omega_E$, which we denote by $\scrA_q$. By a version of the construction from \cite[Section 6]{seidel18b} (itself borrowed from unpublished work of Abouzaid and the present author), those categories come with a restriction functor
\begin{equation}
\scrQ_q: \scrA_q \longrightarrow \scrB_q.
\end{equation}
In particular, pullback yields an $A_\infty$-bimodule $\scrQ_q^*\scrB_q$ over $\scrA_q$, which comes with a tautological map from the diagonal bimodule,
\begin{equation} \label{eq:restriction-trans}
\rho_q: \scrA_q \longrightarrow \scrQ_q^*\scrB_q.
\end{equation}
Fukaya categories, and in particular $\scrA_q$, come with additional structure, which we call the diagonal class. One can think of this as a natural tranformation from the Serre functor to the identity, or as a bimodule map
\begin{equation} \label{eq:serre-trans}
\delta_q: \scrA_q^\vee[-n] \longrightarrow \scrA_q,
\end{equation}
where $\scrA_q^\vee[-n]$ is the dual diagonal bimodule, with the grading shifted up by $n$. Such bimodule maps are the cocycles in a complex $\mathit{CC}^*(\scrA_q,2)$, whose cohomology is denoted by $\mathit{HH}^*(\scrA_q,2)$. (The notation expresses the fact that this is part of a sequence of groups which generalize classical Hochschild cohomology.) In this terminology, \eqref{eq:serre-trans} describes a class
\begin{equation}
[\delta_q] \in \mathit{HH}^n(\scrA_q,2),
\end{equation}
in which form it is independent of all choices. Our aim is to combine \eqref{eq:restriction-trans} and \eqref{eq:serre-trans}:

\begin{theorem} \label{th:main}
The composition of \eqref{eq:serre-trans} and \eqref{eq:restriction-trans} is nullhomotopic. Moreover, for a suitable choice of nullhomotopy, putting all parts together yields a filtered acyclic bimodule
\begin{equation} \label{eq:contractible}
\big\{\!
\xymatrix{ 
\ar@/_1pc/[rrrr]_-{\text{nullhomotopy}}
\scrA_q^\vee[-n] \ar[rr]^-{\delta_q} && \scrA_q \ar[rr]^-{\rho_q} && \scrQ_q^*\scrB_q
}\!\big\}.
\end{equation}
\end{theorem}

The most difficult part of proving the theorem is to construct the required nullhomotopy. Filtered acyclicity, on the other hand, is a statement about the Floer cohomology groups obtained by setting $q = 0$, which means working on the exact symplectic manifold $E \setminus \Omega_E$, and is correspondingly more elementary.
%

\begin{remark} \label{th:pre-cy}
Suppose for simplicity that we work with rational coefficients instead of integer ones. Dualization equips $\mathit{CC}^*(\scrA_q,2)$ with a $\bZ/2$-action. The $(-1)^n$-eigenspace is the ``higher Hochschild cochain group'' $C^{(2,n)}(\scrA_q)$ in the terminology of \cite[Section 3.2]{kontsevich-vlassopoulos13}. The cocycle \eqref{eq:serre-trans} lies in that eigenspace at least up to homotopy. Therefore, its mapping cone $\{\scrA_q^\vee[-n] \rightarrow \scrA_q\}$ is a bimodule which comes with a filtered quasi-isomorphism to its own (shifted) dual. The same is true for $\scrQ_q^*\scrB_q$, as a consequence of the weak Calabi-Yau structure (in the sense of \cite{tradler01}) of $\scrB_q$. We will not attempt to prove that, but it seems natural to suppose that the two structures correspond to each other under the homotopy equivalence $\{\scrA_q^\vee[-n] \rightarrow \scrA_q\} \htp \scrB_q$ induced by \eqref{eq:contractible}.

Going further, one conjectures that $\scrA_q$ has a geometrically defined ``pre-Calabi-Yau structure'' \cite{kontsevich-vlassopoulos13} (also called ``boundary $A_\infty$-algebra'' in \cite{seidel06}, or ``$V_\infty$ algebra'' in \cite{tradler-zeinalian07}; in spite of the proliferation of existing names, I would propose ``log Calabi-Yau'' as a better alternative). Such a structure induces a cyclic $A_\infty$-structure on $\scrA_q^\vee[1-n] \oplus \scrA_q$. The expectation would then be that this is equivalent to $\scrB_q$ as a cyclic $A_\infty$-category. However, as far as this author knows, the geometric moduli spaces underlying the construction of pre-Calabi-Yau structures are not well understood (the ``popsicle spaces'' from \cite{abouzaid-seidel07}, used in \cite{seidel18b}, are unsuitable here because they lack the requisite cyclic symmetry).
\end{remark}

\subsection{Structure of the paper}
Section \ref{sec:gluing} introduces, in its simplest form, the Floer-theoretic trick which underlies our main construction. Section \ref{sec:algebra} reviews the algebraic terminology and sign conventions. Section \ref{sec:fukaya} recalls the definition of the relative Fukaya category, including the diagonal class, for closed Calabi-Yau manifolds. Section \ref{sec:plane} contains some simple considerations about $\bar\partial$-operators, which are needed for the extension of our basic trick to more general Riemann surfaces. Sections \ref{sec:connections} and \ref{sec:setup} set up the geometry and Floer theory of the spaces \eqref{eq:pi}.

Up to that point, everything is essentially getting the preliminaries together. Section \ref{sec:acyclic} carries out a simpler version of the main argument, where the parameter $q$ is set to zero, and we consider Floer cochain complexes without any of their additional $A_\infty$-structures. Nevertheless, this already contains all the main geometric ideas. After that, Section \ref{sec:final} develops the argument in full.

{\em Acknowledgments.} I would like to thank: Tim Large, for a useful discussion concerning coupled Floer equations; and Nick Sheridan, for sharing the unpublished manuscript \cite{perutz-sheridan20}. This work was supported by: the Simons Foundation, through a Simons Investigator award and the Simons Collaboration for Homological Mirror Symmetry; and by the NSF, through award DMS-1904997.

\section{Parametrized continuation maps\label{sec:gluing}}
This section addresses a basic question in the underpinnings of Floer theory. Briefly, the question has to do with the behaviour of the parameter values under gluing in parametrized moduli spaces. We will consider parametrized continuation maps, since this is the most basic version.

\subsection{Background}
Suppose that we have a symplectic manifold $M$, and two Lagrangian submanifolds $L_0,L_1$. Take a time-dependent Hamiltonian $(H_t)$, $t \in [0,1]$, with vector field $(X_t)$; as well as a time-dependent compatible almost complex structure $(J_t)$. A chord between $L_0$ and $L_1$ is an
\begin{equation} \label{eq:chord}
\left\{
\begin{aligned}
& x: [0,1] \longrightarrow M, \\
& x(0) \in L_0, \, x(1) \in L_1, \\
& dx/dt = X_t.
\end{aligned}
\right.
\end{equation}
The linearization of \eqref{eq:chord}, or rather of its equivalent form $J_t(dx/dt - X_t) = 0$, along a chord is a formally selfadjoint operator
\begin{equation}
\begin{aligned}
& Q_x: \smooth([0,1], x^*TM, x^*TL_0, x^*TL_1) 
\stackrel{\mathrm{def}}{=} \\ & \qquad \qquad
\{ \Xi(t)\in TM_{x(t)}\,:\, \Xi(0) \in TL_0, \;\; \Xi(1) \in TL_1 \}
\longrightarrow \smooth([0,1], x^*TM).
\end{aligned}
\end{equation}
One says that $x$ is nondegenerate if $Q_x$ is injective, and hence an isomorphism in suitable Sobolev completions (say from $W^{1,2}$ to $L^2$). Throughout the following discussion, we assume that this holds for all chords. 

The Floer trajectories are
\begin{equation} \label{eq:floer}
\left\{
\begin{aligned}
& u: \bR \times [0,1] \longrightarrow M, \\
& u(s,0) \in L_0, \, u(s,1) \in L_1, \\
& \textstyle \lim_{s \rightarrow \pm\infty} u(s,t) = x_\pm(t), \\
& \partial_s u + J_t(\partial_t u - X_t) = 0,
\end{aligned}
\right.
\end{equation}
where $x_{\pm}$ are chords \eqref{eq:chord}. One usually excludes solutions $u(s,t) = x(t)$ from consideration. The linearization of \eqref{eq:floer} is an operator
\begin{equation} \label{eq:linearized-op}
\begin{aligned}
& D_u: \smooth(\bR \times [0,1], u^*TM, u^*TL_0,u^*TL_1) 
\stackrel{\mathrm{def}}{=} \\ & \qquad
\{ \Upsilon(s,t) \in TM_{u(s,t)}\,:\, \Upsilon(s,0) \in TL_0, \;\; \Upsilon(s,1) \in TL_1 \}
\longrightarrow \smooth(\bR \times [0,1], u^*TM),
\end{aligned}
\end{equation}
whose asymptotic behaviour over the ends of the strip is that
 \begin{equation} \label{eq:asymptotic-d}
D_u \longrightarrow \partial_s + Q_{x_\pm}.
\end{equation}
In suitable Sobolev spaces (one can choose $W^{1,2}$ and $L^2$ as before), $D_u$  becomes Fredholm. Regularity of a solution $u$ of \eqref{eq:floer} means surjectivity of the associated Fredholm operator.
 
Let's suppose that we have two choices $(H_{\pm,t}, J_{\pm,t})$. Take $(H_{s,t},J_{s,t})$, $(s,t) \in \bR \times [0,1]$, which interpolates between them, in the sense that
\begin{equation} \label{eq:cont-family}
(H_{s,t},J_{s,t}) \longrightarrow (H_{\pm,t},J_{\pm,t}) \;\; \text{ as } s \rightarrow \pm\infty.
\end{equation}
Here, convergence is understood to be exponentially fast in any $C^r$ topology. The continuation map equation is a version of \eqref{eq:floer} which breaks the $s$-translation invariance, meaning
\begin{equation} \label{eq:cont-map-equation}
\partial_s u + J_{s,t}(\partial_t u - X_{s,t}) = 0,
\end{equation}
with limits $x_\pm$ which are $H_\pm$-chords. The linearization of this equation has the same form as in \eqref{eq:linearized-op}. Finally, suppose that we have a family $(H_{r,s,t},J_{r,s,t})$ depending on an additional parameter $r$, and which satisfies \eqref{eq:cont-family} for all values of $r$. One can then look at the parametrized version of the continuation map equation, meaning pairs $(r,u)$ consisting of a parameter value and a solution of \eqref{eq:cont-map-equation} with $J$ and $X$ for that value of $r$. Linearizing that gives rise to an extended version of $D_u$, denoted by
\begin{equation}
D_{(r,u)}^{\mathit{para}}: \bR \oplus \smooth(\bR \times [0,1], u^*TM, u^*TL_0,u^*TL_1) 
\longrightarrow \smooth(\bR \times [0,1], u^*TM),
\end{equation}
where the $\bR$ component takes into account infinitesimal changes of the parameter $r$. In this context, one has to distinguish between regularity of $u$ in the ordinary sense (surjectivity of $D_u$, in the usual Sobolev spaces) and surjectivity of $(r,u)$ in the parametrized sense (surjectivity of $D_{(r,u)}^{\mathit{para}}$); only the latter, weaker, property is generically satisfied over a whole parametrized moduli space.

\subsection{Local linearity}
Let's briefly review the effect of time-dependent coordinate changes on our equations. Namely, suppose that alongside our manifold $M$ we have another one $\tilde{M}$ of the same dimension, and symplectic embeddings $f_t: \tilde{M} \rightarrow M$ ($t \in [0,1]$), such that $f_t^*(\partial_t f_t)$ is the Hamiltonian vector field of some function $\tilde{F}_t$. If we write a chord or a Floer trajectory on $M$ as
\begin{equation} \label{eq:transform}
x(t) = f_t(\tilde{x}(t)) \;\; \text{ or } \;\; u(s,t) = f_t(\tilde{u}(s,t)),
\end{equation}
then $\tilde{x}$ and $\tilde{u}$ satisfy the analogues of \eqref{eq:chord} and \eqref{eq:floer} for
\begin{equation} \label{eq:transformed-data}
\left\{
\begin{aligned}
& \tilde{L}_0 = f_0^{-1}(L_0), \; \tilde{L}_1 = f_1^{-1}(L_1), \\
& \tilde{H}_t = f_t^*H_t - \tilde{F}_t, \\
& \tilde{J}_t = f_t^*J_t.
\end{aligned}
\right.
\end{equation}
On the level of linearized operators, there is an induced isomorphism between $Q_x$ (or $D_u$) and $Q_{\tilde{x}}$ (respectively $D_{\tilde{u}}$). 

We find it convenient to introduce a simplifying technical assumption (which could presumably be dropped, but at the cost of requiring slightly beefier analytic methods; see e.g.\ \cite{robbin-salamon02}). For Floer's equation, this looks as follows. Let $x$ be a chord.
\begin{equation} \label{eq:locally-linear}
\mybox{
We say that $(H_t)$ and $(J_t)$ are locally linear near $x$ if the following holds. There is a time-dependent local Darboux chart,
\[
f_t: (\bR^{2n},0) \dashrightarrow (M,x(t))
\]
 (the dashed arrow indicates that $f_t(x)$ is defined only for $(t,x)$ in a neighbourhood of $[0,1] \times \{0\} \subset [0,1] \times \bR^{2n}$), such that for the transformed data \eqref{eq:transformed-data}, $(\tilde{L}_0,\tilde{L}_1)$ are linear Lagrangian subspaces, $\tilde{H}_t$ are quadratic forms plus constants, and $\tilde{J}_t$ are constant (over $\bR^{2n}$, while still potentially $t$-dependent).
}
\end{equation}
We will assume from now on that all chords have this property. When setting up continuation maps, the corresponding condition is:
\begin{equation} \label{eq:added-condition}
\mybox{
Take $(H_{s,t},J_{s,t})$ which define a continuation map equation, with limits $(H_{\pm,t},J_{\pm,t})$. Let $x_\pm$ be one of the chords that can appear in the limit (and which satisfy local linearity). We say that the continuation map equation is locally $s$-independent near that chord if for $\pm s \gg 0$, we have $(H_{s,t},J_{s,t}) = (H_{\pm,t}, J_{\pm, t})$ in a neighbourhood of $\{ (t,x_\pm(t)) \} \subset [0,1] \times M$.
}
\end{equation}

Suppose that we have a Floer trajectory or continuation map solution $u$. Local linearity, together with \eqref{eq:added-condition} for continuation maps, implies that on each of the regions $\pm s \gg 0$, we can apply a coordinate change \eqref{eq:transform} so that the transformed maps $\tilde{u}_{\pm}$ satisfy a linear (and $s$-independent) Cauchy-Riemann equation. As a consequence, we have an equality
\begin{equation} \label{eq:large-s}
D_{\tilde{u}_\pm} = \partial_s + Q_{\tilde{x}_\pm} \;\; \text{on regions $\pm s \gg 0$,}
\end{equation}
rather than the asymptotic statement \eqref{eq:asymptotic-d}. Suppose that we have a solution $\tilde{\Upsilon}$ of $D_{\tilde{u}_+} \tilde{\Upsilon} = 0$, defined on $[s_+,\infty) \times [0,1]$ for some $s_+ \gg 0$. One can write it as
\begin{equation} \label{eq:upsilon-plus}
\tilde{\Upsilon}(s,t) = \textstyle
\sum_{\lambda} \exp(-\lambda s) \Xi_{\lambda}(t),
\end{equation}
where each $\Xi_{\lambda}$ is a $\lambda$-eigenvector for $Q_{\tilde{x}_+}$ (which includes the possibility of its being zero). The solution decays as $s \rightarrow +\infty$ iff the only nonzero terms come from positive eigenvalues. Here's a simple application:

\begin{lemma}
Take a nonzero solution of $D_u \Upsilon = 0$, defined on $[s_+,\infty) \times [0,1]$, and which decays as $s \rightarrow +\infty$. Then there is an eigenvalue $\lambda > 0$ of $Q_{x_+}$, such that
\begin{equation} \label{eq:exp-decay}
\textstyle \lim_{s \rightarrow \pm\infty} \exp(\lambda s) \Upsilon\;\; \text{ is a nonzero $\lambda$-eigenvector.}
\end{equation}
\end{lemma}

\begin{proof}
By unique continuation, there are $(s,t)$ with arbitrarily large $s$ on which $\Upsilon(s,t) \neq 0$. Hence, we can make $s_+$ larger if necessary, so that local linearity applies, and obtain a nonzero \eqref{eq:upsilon-plus}.
For that, the corresponding property is obvious, and one then transforms back to $\Upsilon$.
\end{proof}

Take some \eqref{eq:upsilon-plus} which decays as $s \rightarrow +\infty$. The $L^2$-norm on a region $[s,\infty) \times [0,1]$ is (omitting a factor of $1/2$, which is irrelevant for future considerations)
\begin{equation} \label{eq:l2} \textstyle
\big\|\tilde\Upsilon \big\|^2_{L^2([s,\infty) \times [0,1])} = \sum_{\lambda>0} \lambda^{-1}\exp(-2\lambda s) \|\Xi_{\lambda}\|^2,
\end{equation}
where $\|\cdot\|$ is the $L^2$-norm on $[0,1]$. One can compare that with
\begin{equation} \label{eq:norm} \textstyle
\big\|\tilde\Upsilon|(\{s\} \times [0,1])\big\|^2 =
\sum_{\lambda > 0} \exp(-2\lambda s) \|\Xi_{\lambda}\|^2.
\end{equation}
Finally, the analogue of \eqref{eq:l2} for the $W^{1,2}$-norm (or rather, a norm equivalent to it) is
\begin{equation} \label{eq:w12} \textstyle
\big\|\tilde\Upsilon \big\|^2_{W^{1,2}([s,\infty) \times [0,1])} = \sum_{\lambda>0} \lambda \exp(-2\lambda s) \|\Xi_{\lambda}\|^2.
\end{equation}
Since the $\lambda$ are bounded below by a positive constant, these three norms are increasingly strict (each dominating the previous one). On the other hand, \eqref{eq:l2} dominates the $W^{1,2}$-norm on any $[s',\infty) \times [0,1]$ for $s' > s$; this is an elementary instance of ellipticity, and has analogues for higher Sobolev spaces $W^{l,2}$. An easy consequence of these observations is:

\begin{lemma}
Take a sequence $\Upsilon_k$ of solutions of $D_u \Upsilon_k = 0$, defined on some region $(s_+,\infty) \times [0,1]$, and each of which decays as $s \rightarrow +\infty$. Suppose that the $\Upsilon_k$ converge on compact subsets to some $\Upsilon$ (because of ellipticity, it does not matter which topology one chooses for convergence). Then $\Upsilon$ again decays as $s \rightarrow +\infty$; and moreover, the convergence holds in $W^{l,2}$-sense on $[s,\infty) \times [0,1]$, for $s>s_+$ and any $l$.
\end{lemma}

Of course, all of these remarks have counterparts for regions with $s \ll 0$, where the negative half of the spectrum of $Q_{x_-}$ is relevant for decaying solutions.

\subsection{Gluing}
We need to say a few words about gluing of solutions. Let's suppose that $u^2$ is a solution of a continuation map equation, and $u^1$ a Floer trajectory for $(H_+,J_+)$, with matching limits
\begin{equation} \label{eq:matching-limits}
\textstyle \lim_{s \rightarrow +\infty} u^2(s,t) = x(t) = \lim_{s \rightarrow -\infty} u^1(s,t).
\end{equation}
Both should be regular isolated points of the respective moduli spaces. In the case of $u^1$, this of course means Floer trajectories up to translation, so that $\mathit{ker}(D_{u^1}) = \bR \cdot \partial_s u^1$. Gluing then produces a family $(u_g)$ of continuation map solutions, parametrized by a gluing length $g \gg 0$. In the limit $g \rightarrow \infty$, we get back the original maps, thought of as forming a single broken solution. A weak form of that convergence statement is that $u_g \rightarrow u^2$ on compact subsets, and that $u_g(s+g,t) \rightarrow u^1$ in the same sense.

\begin{lemma} \label{th:convergence-to-derivative}
As $g \rightarrow \infty$, $\partial_g u_g$ converges to zero on compact subsets. In the same sense, $(\partial_g u_g)(s+g,t)$ converges to $-\partial_s u^1$.
\end{lemma}

\begin{proof}[Sketch of proof]
The statement is slightly imprecise, since the parametrization of the glued family by $g$ is not canonical, but rather depends on the setup used for gluing. We will explain this for one particular choice, which involves a stabilizing hypersurface (a geometrically straightforward way of breaking translation invariance, even if not the simplest analytically).
That is enough for our applications. 

Namely, pick some point $(s_*,t_*) \in \bR \times (0,1)$ where $\partial_s u^1$ is nonzero (this is always possible, as the subset of points where $\partial_s u^1 = 0$ is discrete). Then, find a local hypersurface $H \subset M$ through $p = u^1(s_*,t_*)$, which is transverse to the path $s \mapsto u^1(s,t_*)$ at $s = s_*$. We fix the choice of glued solution so that $u_g(s_*+g,t_*) \in H$. Because $u_g(\cdot+g,\cdot) \rightarrow u^1$, we have in particular
\begin{equation} \label{eq:d-convergence}
(\partial_s u_g)(s_*+g,t_*) \longrightarrow (\partial_s u^1)(s_*,t_*),
\end{equation}
which means that $(\partial_s u_g)(s_*+g,t_*) \neq 0$ for $g \gg 0$. Correspondingly, $\partial_g u_g$ can be characterized as the unique element of $\mathit{ker}(D_{u_g}) \iso \bR$ which satisfies
\begin{equation} \label{eq:characterize-dg}
(\partial_g u_g)(s_*+g,t_*) 
\in \mathit{TH}_p - (\partial_s u_g)(s_*+g,t_*) \subset \mathit{TM}_p.
\end{equation}

Let's quickly think through the gluing process. One starts with a preglued solution $u^2 \#_g u^1$, which satisfies
\begin{equation} \label{eq:preglue}
(u^2 \#_g \, u^1)(s,t) = \begin{cases} u^2(s,t) & s \leq g/2-1/2, \\
u^1(s-g,t) & s \geq g/2+1/2, \end{cases}
\end{equation}
and which on the interval not described in \eqref{eq:preglue} uses cutoff functions, in a suitable tubular neighbourhood of the chord $x$ from \eqref{eq:matching-limits}, to interpolate between the relevant values. One obtains $u_g$ from this by an inverse function theorem argument in the space of maps satisfying $u(s_*+g,t_*) \in H$. The preglued solution comes with a section of $(u^2 \#_g u^1)^*TM$, which we can analogously write as
\begin{equation} \label{eq:linear-preglue}
(0 \#_g\, (-\partial_s u^1))(s,t) = \begin{cases} 0 & s \leq g/2-1/2, \\
-(\partial_s u^1)(s-g,t) & s \geq g/2+1/2, \end{cases}
\end{equation}
again using a cutoff on the missing interval. The failure of \eqref{eq:linear-preglue} to lie in the kernel of an appropriate operator $D_{u^2 \#_g u^1}$ goes to zero as $g \rightarrow \infty$. Moreover, while $(0 \#_g (-\partial_s u^1))(s_*+g,t_*) = -\partial_s u^1(s_*,t_*)$ does not lie in the subspace from \eqref{eq:characterize-dg}, it lies in the parallel translate $\mathit{TH}_p - (\partial_s u^1)(s_*,t_*)$, which is close by \eqref{eq:d-convergence}. One can therefore apply a linearized version of standard inverse function theorem argument to \eqref{eq:linear-preglue}, in order to obtain an $\Upsilon_g \in \mathit{ker}(D_{u_g})$ such that $\Upsilon_g(s_*+g,t_*) \in \mathit{TH}_p - (\partial_s u_g)(s_*+g,t_*)$, and whose distance to \eqref{eq:linear-preglue} goes to zero as $g \rightarrow \infty$. By our previous characterization \eqref{eq:characterize-dg}, $\Upsilon_g = \partial_g u_g$. On any compact subset, \eqref{eq:linear-preglue} ultimately becomes zero, and therefore $\Upsilon_g$ converges to zero as $g \rightarrow \infty$, which proves our first claim. Similarly, the translated versions $(0 \#_g (-\partial_s u^1))(s+g,t)$ converge to $-\partial_s u^1$ on compact subsets, and therefore, the same is true for $\Upsilon_g$, which was our second claim.
\end{proof}

\subsection{A parametrized gluing situation}
Our main argument takes place in the following setup. We have a parametrized continuation map equation, given by $(H_{r,s,t},J_{r,s,t})$. 
\begin{align} \label{eq:matching-us}
&
\mybox{
Let $(r^2,u^2)$ be a solution of the parametrized equation. We require that this is an isolated regular point of the parametrized moduli space, meaning that $D_{(r^2,u^2)}^{\mathit{para}}$ is invertible.
}
\\ \label{eq:matching-us-2}
&
\mybox{
Let $u^1$ be a Floer trajectory for $(H_+,J_+)$, which is also an isolated regular point of its moduli space. Moreover, the two pieces should have matching limits \eqref{eq:matching-limits}.
}
\end{align}
Parametrized gluing theory says that there is a family $(r_g, u_g)$ of solutions to the parametrized continuation map equation, depending on a large gluing length $g \gg 0$, which in the limit $g \rightarrow \infty$ converges to the two given pieces. This family comes with tangent vectors
\begin{equation} \label{eq:tangent-vector}
(R_g, \Upsilon_g) = \partial_g (r_g, u_g) \in \mathit{ker}(D_{(r_g,u_g)}^{\mathit{para}}).
\end{equation}
The parametrized analogue of Lemma \ref{th:convergence-to-derivative}, proved in the same way, tells us that:
\begin{equation} \label{eq:plus-derivative}
\mybox{
as $g \rightarrow \infty$, $(\partial_g u_g)(s+g,t) \longrightarrow -\partial_s u^1$ on compact subsets.
}
\end{equation}

Choose $s^1$, $s^2$ so that, for large $g$, the restriction of $u_g$ to $[s^2,s^1 + g] \times [0,1]$ remains close to $x$; the same should apply to $u^1$ on $(-\infty,s^1] \times [0,1]$. We apply a change of variables as in \eqref{eq:upsilon-plus} to \eqref{eq:tangent-vector}, and write
\begin{equation} \label{eq:transformed-ul}
\tilde\Upsilon_g = \widetilde{\partial_g u_g} = \textstyle \sum_{\lambda} \exp(-\lambda s) \Xi_{g,\lambda}, 
\end{equation}
where the $\lambda$ are eigenvalues of $Q_x$ (of either sign), and the expression is valid on $[s^2,s^1 + g] \times [0,1]$. In parallel, we get
\begin{equation} \label{eq:transformed-uplus}
\tilde\Upsilon^1 = \widetilde{-\partial_s u^1} = \textstyle \sum_{\lambda} \exp(-\lambda s)
\Xi_{\lambda}^1
\end{equation}
on $(-\infty,s^1] \times [0,1]$, and where this time, the positive $\lambda$ must have vanishing coefficients $\Xi_{\lambda}^1$. As a consequence of \eqref{eq:plus-derivative} for $s = s^1$, we have that as $g \rightarrow \infty$,
\begin{equation} \label{eq:convergence-2}
\textstyle \sum_{\lambda} \exp(-2\lambda s^1) \|\exp(-\lambda g)\Xi_{g,\lambda} - \Xi_{\lambda}^1 \|^2 \rightarrow 0.
\end{equation}
In particular,
\begin{equation} \label{eq:convergence-implies-bound}
\textstyle \sum_{\lambda} \exp(2 \lambda (-s^1 - g))\, \|\Xi_{g,\lambda}\|^2 \; \text{ is bounded as $g \rightarrow \infty$.}
\end{equation}
Let $\lambda^1 < 0$ be the eigenvalue which governs the $s \rightarrow -\infty$ behaviour of $\partial_s u^1$, so 
\begin{equation} \label{eq:lambda-plus} \textstyle
\lim_{s \rightarrow -\infty} \exp(\lambda^1 s) \partial_s u^1\;\; \text{is a nonzero $\lambda^1$-eigenvector.}
\end{equation} 
Equivalently, this is the highest eigenvalue which has a nonzero coefficient in \eqref{eq:transformed-uplus}. 

\begin{prop} \label{th:rescaled-limit}
Take \eqref{eq:tangent-vector} for some sequence $g_k$ going to $\infty$, and assume that $R_{g_k} = (dr_g/dg)_{g_k} \neq 0$ for all $k$. Then, after rescaling by suitable positive constants $c_k$, a subsequence of the rescaled versions $(c_k R_{g_k}, c_k \Upsilon_{g_k})$ will converge on compact subsets to a nonzero solution $(R,\Upsilon)$ of
\begin{equation}
D_{(r^2,u^2)}^{\mathit{para}}(R,\Upsilon) = 0,
\end{equation}
with the following properties:
\begin{align}
& \label{eq:left-decay}
\mybox{$\Upsilon$ decays as $s \rightarrow -\infty$.}
\\
& \label{eq:matching}
\mybox{$\lim_{s \rightarrow +\infty} \exp(\lambda^1 s) \Upsilon$ is a nonpositive multiple of the eigenvector from \eqref{eq:lambda-plus}.}
\end{align}
\end{prop}

\begin{proof}
We change notation from now on, using the subscript $k$ to refer to the data associated to $g = g_k$. Suppose that the rescaling factors $c_k$ are chosen so that
\begin{equation} \label{eq:c-bound}
c_k^2\big(R_k^2 + \|\Upsilon_k\|^2_{W^{1,2}((-\infty,s^2] \times [0,1])} \big) \;\text{ is bounded as $k \rightarrow \infty$.}
\end{equation}
This implies that
\begin{equation} \label{eq:c-bound-2}
\textstyle c_k^2 \sum_{\lambda} \exp(-2\lambda s^2) \|\Xi_{k,\lambda}\|^2 \; \text{ is bounded as $k \rightarrow \infty$.} 
\end{equation}
By \eqref{eq:convergence-2} we have $\exp(-\lambda^1 g_k)\Xi_{k,\lambda^1} \rightarrow \Xi_{\lambda^1}^1 \neq 0$. Combining this with \eqref{eq:c-bound-2} yields that
\begin{equation} \label{eq:scaling-bound}
c_k \lesssim  \exp(\lambda^1 (s^2 - g_k)) \sim \exp(\lambda^1(-s^1 - g_k))
\end{equation}
where $\lesssim$ means less than a constant ($k$-independent) multiple of the other side, and similarly for $\sim$. Fix some $s \geq s^2$. As long as $k$ is large, we have 
\begin{equation} \textstyle
s^1 + g_k \geq \frac{s}{1 - \lambda^1/\lambda} 
\; \; \Leftrightarrow \;\; (\lambda^1 - \lambda) (s^1 + g_k) \geq -\lambda s,
\;\; \text{ for all eigenvalues $\lambda < \lambda^1$.}
\end{equation}
Therefore,
\begin{equation} \label{eq:ana}
\begin{aligned}
& c_k^2 \textstyle \sum_{\lambda < \lambda^1} \exp(-2\lambda s)\, \|\Xi_{k,\lambda}\|^2
\\ & \qquad
\lesssim 
\textstyle \sum_{\lambda < \lambda^1} \exp(2\lambda^1 (-s^1 - g_k) - 2\lambda s)\, \|\Xi_{k,\lambda}\|^2
\\ & \qquad
\leq \textstyle \sum_{\lambda < \lambda^1} \exp(2 \lambda  (-s^1 - g_k) + 2(\lambda^1 - \lambda) (-s^1 - g_k) - 2\lambda s)\,
\|\Xi_{k,\lambda}\|^2
\\ 
 & \qquad
\leq \textstyle \sum_{\lambda < \lambda^1} \exp(2\lambda (-s^1 - g_k)) \,\|\Xi_{k,\lambda}\|^2.
\end{aligned}
\end{equation}
The last line is bounded because of \eqref{eq:convergence-implies-bound}. For positive eigenvalues we have the elementary inequality
\begin{equation} \label{eq:ana-2}
\textstyle
c_k^2 \sum_{\lambda > 0} \exp(-2\lambda s) \|\Xi_{k,\lambda}\|^2
\leq c_k^2 \sum_{\lambda > 0} \exp(-2\lambda s^2) \|\Xi_{k,\lambda}\|^2.
\end{equation}
The remaining eigenvalues $\lambda \in [\lambda^1,0)$ are finite in number, and so ($s$ being fixed)
\begin{equation} \label{eq:ana-3}
\textstyle
c_k^2 \sum_{\lambda^1 \leq \lambda < 0}
\exp(-2\lambda s) \|\Xi_{k,\lambda}\|^2 \lesssim c_k^2 
\sum_{\lambda^1 \leq \lambda < 0}
\exp(-2\lambda s^2) \|\Xi_{k,\lambda}\|^2.
\end{equation}
The right hand sides of \eqref{eq:ana-2}, \eqref{eq:ana-3} are bounded by \eqref{eq:c-bound-2}. By combining that with \eqref{eq:ana}--\eqref{eq:ana-3}, one sees that
\begin{equation} \textstyle
c_k^2 \sum_{\lambda} \exp(-2\lambda s)\|\Xi_{k,\lambda}\|^2 \;\; \text{is bounded as $k \rightarrow \infty$.}
\end{equation}
The bound depends on $s \geq s^2$, but inspection of the argument shows that it can be made uniform on any bounded interval. This, together with the original assumption \eqref{eq:c-bound}, shows that the rescaled solutions $c_k\Upsilon_k$ are bounded in $L^2((-\infty,s] \times [0,1]$) for every $s$ (by a bound that depends on $s$, but is independent of $k$). By ellipticity, one gets the same kind of bound in any $W^{l,2}$-norm. As a consequence, after passing to a subsequence, we can achieve that the rescaled solutions converge on any subset of the form $(-\infty,s] \times [0,1]$, with the limit being some $(R,\Upsilon)$ defined on all of $\bR \times [0,1]$. Let's choose the scaling factors so that the norm in \eqref{eq:c-bound} is bounded below by a positive number (for instance, one could take it equal to $1$; here, we are using the fact that $R_k \neq 0$). Then, the limit of our subsequence is necessarily nonzero.

Because of the convergence, $\Upsilon$ satisfies \eqref{eq:left-decay}. For $\lambda < \lambda^1$ we have, by \eqref{eq:scaling-bound} and \eqref{eq:convergence-implies-bound},
\begin{equation}
\begin{aligned}
& \|c_k \Xi_{k,\lambda}\| \lesssim \exp(-\lambda^1 g_k) \|\Xi_{k,\lambda}\| 
= \exp((\lambda - \lambda^1) g_k) \| \exp(-\lambda g_k) \Xi_{k,\lambda} \| \\
& \qquad \lesssim \exp((\lambda - \lambda^1) g_k) \longrightarrow 0\;\; \text{ as $k \rightarrow \infty$,}
\end{aligned}
\end{equation}
which means that the $s \geq s^2$ part of $\Upsilon$ must have zero coefficients for those eigenvalues, establishing the growth rate bound from \eqref{eq:matching}. For $\lambda = \lambda^1$, using \eqref{eq:convergence-2} instead, one gets
\begin{equation}
\| c_k \Xi_{k,\lambda^1} - c_k \exp(\lambda^1 g_k) \Xi^1_{\lambda^1} \| \lesssim \|\exp(-\lambda^1 g_k) \Xi_{k,\lambda^1} - \Xi^1_{\lambda^1}\| \longrightarrow 0.
\end{equation}
This leaves two possibilities: either $c_k \exp(\lambda^1 g)$ goes to zero, in which case $\lim_{s \rightarrow +\infty} \exp(\lambda^1 s) \Upsilon$ vanishes; or it converges to a positive number, in which case $\lim_{s \rightarrow +\infty} \exp(\lambda^1 s) \Upsilon$ is a positive multiple of $\lim_{s \rightarrow -\infty} \exp(\lambda^1 s) (-\partial_s u^1)$. This proves the rest of \eqref{eq:matching}.
\end{proof}

\begin{corollary} \label{th:rescaled-limit-2}
Assume that:
\begin{equation} \label{eq:lowest-eigenvalue}
\mybox{
In \eqref{eq:lambda-plus}, $\lambda^1$ is the highest (closest to zero) negative eigenvalue of $Q_x$.
}
\end{equation}
Then, the multiple in \eqref{eq:matching} is necessarily nonzero (hence negative).
\end{corollary}

\begin{proof}
If the multiple were zero, $(R,\Upsilon)$ would be a nonzero element in the kernel of $D^{\mathit{para}}_{(r,u)}$ which decays as $s \rightarrow \pm\infty$. But the assumption that $(r,u)$ is an isolated regular point of the parametrized moduli space forbids the existence of such elements.
\end{proof}

We now turn to the question which is actually of interest to us: in the glued family $(r_g,u_g)$, how do the parameter values $r_g$ behave? Are they bigger or smaller than $r^2$? One can't expect an answer in complete generality, but we can see the following:

\begin{corollary} \label{th:r-moves}
Suppose that \eqref{eq:lowest-eigenvalue}  holds. Take the unique solution of $D_{(r,u)}^{\mathit{para}}(R, \Upsilon) = 0$ which decays as $s \rightarrow -\infty$, and such that 
\begin{equation} \textstyle \label{eq:matching-condition}
\lim_{s \rightarrow +\infty} \exp(\lambda^1 s) \Upsilon = -\lim_{s \rightarrow -\infty}\, \exp(\lambda^1 s) \partial_s u^1. 
\end{equation}
Assume that $R \neq 0$. Then, for large $g$ we have
\begin{equation} \label{eq:r-moves-how}
\begin{cases} r_g \leq r^2 &  \text{if $R > 0$,} \\
r_g \geq r^2 & \text{if $R < 0$.}
\end{cases}
\end{equation}
\end{corollary}

\begin{proof}
We need a preliminary index theory consideration (such arguments will figure in greater generality in Section \ref{subsec:weights}). Take a $\mu$ which is smaller than $\lambda^1$, but larger than the next lower eigenvalue of $Q_x$. We introduce weighted Sobolev spaces $W^{l,2;-\mu}([0,1] \times \bR)$ which agree with their standard counterparts for $s \ll 0$, but which are defined by asking for $\exp(\mu s)\Upsilon$ to be of class $W^{l,2}$ for $s \gg 0$. The linearized operator
\begin{equation} \label{eq:alpha-d}
\begin{aligned} &
D_{(r^2,u^2)}^{\mathit{para},\mu}: \bR \oplus W^{1,2;-\mu}(\bR \times [0,1], (u^2)^*TM, (u^2)^*TL_0,(u^2)^*TL_1) \\ & \qquad \qquad
\longrightarrow L^{2;-\mu}(\bR \times [0,1], (u^2)^*TM)
\end{aligned}
\end{equation}
is Fredholm, and its index is equal to the multiplicity of $\lambda^1$ as an eigenvalue. If $E_{\lambda^1}$ is the corresponding eigenspace, the asymptotic behaviour of $\exp(\lambda^1 s)\Upsilon$ gives a map
\begin{equation} \label{eq:lambda-plus-map}
\mathit{ker}(D_{(r^2,u^2)}^{\mathit{para},\mu}) \longrightarrow E_{\lambda^1}.
\end{equation}
Because of our assumption that $(r^2,u^2)$ is isolated and regular, that map must be injective. This, in combination with the index statement, implies that \eqref{eq:alpha-d} is onto, and that \eqref{eq:lambda-plus-map} is an isomorphism. That justifies the uniqueness statement for $(R,\Upsilon)$.

Suppose that there are arbitrarily large $g$ such that $r_g > r^2$. Then, there must also be arbitrarily large $g$ such that $\partial_g r_g < 0$. By applying Proposition \ref{th:rescaled-limit} to a sequence of such $g$, one gets another $(R,\Upsilon)$, which because of its properties must agree with the one described above (up to multiplication with a positive constant, which is irrelevant). By construction of that element, it must have $R \leq 0$. Contrapositively, if $R > 0$, then $r_g \leq r^2$ for large $g$, as claimed. The other case is parallel.
\end{proof}

\subsection{A modified situation\label{subsec:modified}}
We also want to consider a related problem, in which two different continuation map equations are being glued together. 

Let's say that we have three Floer data $(H_{-,t},J_{-,t})$, $(H_t, J_t)$, $(H_{+,t},J_{+,t})$; as well as two continuation map data $(H^2_{s,t}, J^2_{s,t})$ and $(H^1_{s,t}, J^1_{s,t})$, with
\begin{equation}
\left\{
\begin{aligned}
& \textstyle \lim_{s \rightarrow -\infty} (H^2_{s,t}, J^2_{s,t}) = (H_{-,t}, J_{-,t}), \\
& \textstyle \lim_{s \rightarrow +\infty} (H^2_{s,t}, J^2_{s,t}) 
= \lim_{s \rightarrow -\infty} (H^1_{s,t}, J^1_{s,t}) = (H_t,J_t), \\
& \textstyle \lim_{s \rightarrow +\infty} (H^1_{s,t}, J^1_{s,t}) = (H_{+,t}, J_{+,t}).
\end{aligned}
\right.
\end{equation}
Along with these, we want to have glued data, which define a family of continuation map equations depending on a length parameter $g \gg 0$. These data can be viewed in two different conventions, denoted by $(H^\nu_{g,s,t}, J^\nu_{g,s,t})$ for $\nu = 1,2$, and related by
\begin{equation} \label{eq:translate-12}
(H^1_{g,s,t}, J^1_{g,s,t}) = (H^2_{g,s+g,t}, J^2_{g,s+g,t}).
\end{equation}
The relation between these and the original continuation maps is:
\begin{align} \label{eq:glue-the-data-1}
& \mybox{
On every compact subset of $\bR \times [0,1]$, the $(H^\nu_{g,s,t}, J^\nu_{g,s,t})$ converge to $(H^\nu_{s,t}, J^\nu_{s,t})$ as $g \rightarrow \infty$. More precisely, in each $C^r$-norm, the convergence is exponentially fast in $g$.
} \\
& \mybox{
As $s \rightarrow-\infty$, we have exponential convergence $(H^2_{g,s,t},J^2_{g,s,t}) \rightarrow (H_{-,t},J_{-,t})$ in a way which holds uniformly in $g$. Convergence $(H^1_{g,s,t},J^1_{g,s,t}) \rightarrow (H_{+,t},J_{+,t})$ as $s \rightarrow +\infty$ holds in the same sense.
} \\ \label{eq:glue-the-data-3}
& \mybox{
On any subset of the form $[s^2, s^1 + g]$, the difference between $(H_{g,s,t}^2,J_{g,s,t}^2)$ and $(H_t,J_t)$, in each $C^r$-norm, is bounded by a constant (independent of $g$) times $\exp(-\alpha^2 s) + \exp(\alpha^1 (s-g))$, for some $\alpha^1, \alpha^2 > 0$.}
\end{align}
Suppose that $(u^2,u^1)$ are regular isolated solutions of our continuation map equations, with a common limit \eqref{eq:matching-limits}. Gluing produces, for each sufficiently large $g$, solutions $u_g^\nu$, $\nu = 1,2$, of the continuation map associated to $(H_{g,s,t}^\nu, J_{g,s,t}^\nu)$; these are two forms of the same solution, related as in \eqref{eq:translate-12}, but are separated for notational convenience. The counterpart of Lemma \ref{th:convergence-to-derivative} says that:

\begin{lemma} \label{th:convergence-to-derivative-2}
As $g \rightarrow \infty$, both $\partial_g u_g^\nu$ converge to zero on compact subsets. 
\end{lemma}

\begin{proof}[Sketch of proof]
This situation is simpler than the previous one, since the parameter is fixed explicitly by the $g$-dependence of the glued equation. We can introduce an extension of the linearized operator which takes into account varying $g$. Since the geometric origin of the parameter is a bit different than before, we prefer to keep separate notation, and write $\scrD_{(g,u_g^2)}$ for this extended operator (rather than $D^{\mathit{para}}$). By definition,
\begin{equation} \label{eq:1-g}
(1,\partial_g u_g^2) \in \mathit{ker}(\scrD_{(g,u_g^2)}) \iso \bR.
\end{equation}

Now consider a preglued solution $u^2 \#_g u^1$ as in \eqref{eq:preglue}, and a section $0 \#_g (-\partial_s u^1)$ of $(u^2 \#_g u^1)^*TM$, as in \eqref{eq:linear-preglue}. As $g \rightarrow \infty$, $\scrD_{(g,u^2 \#_g u^1)}(1, 0 \#_g (-\partial_s u^1) ) \longrightarrow 0$ in any Sobolev space. Using the linearized version of gluing, we can therefore find an $(R_g,\Upsilon_g) \in \mathit{ker}(\scrD_{(g,u^2_g)})$ which, as $g \rightarrow \infty$, becomes closer and closer to $(1,0 \#_g (-\partial_s u^1))$. Comparing that with \eqref{eq:1-g}, it follows that $(R_g,\Upsilon_g) = \gamma_g (1,\partial_g u_g^2)$ with $\gamma_g \rightarrow 1$ as $g \rightarrow \infty$. Since $0 \#_g (-\partial_s u^1)$ goes to zero on compact subsets as $s \rightarrow \infty$, so does $\Upsilon_g$, and hence so does $\partial_g u_g^2$. One can play the same game in the $(s+g,t)$ coordinate system to obtain the corresponding result for $u_g^1$.
\end{proof}

Finally, we want to tweak the setup so that local linearity applies. This means that for all Floer equations involved, \eqref{eq:locally-linear} holds; and similarly, \eqref{eq:added-condition} holds for all continuation maps. This is not quite everything, there are more precise condition on the glued data:
\begin{align} \label{eq:local-ex}
&
\mybox{
Let $x_-$ be a chord for $(H_-,J_-)$. There is a $\sigma^2$ such that $(H_{g,s,t}^2,J_{g,s,t}^2) = (H_{-,t},J_{-,t})$ for all $s \leq \sigma^2$ and all $g \gg 0$, in a neighbourhood of $\{(t,x_-(t))\} \subset [0,1] \times  M$. 
} \\ &
\mybox{Similarly, let $x_+$ be a chord for $(H_+,J_+)$. There is a $\sigma^1$ such that $(H^1_{g,s,t},J^1_{g,s,t}) = (H^1_{+,t}, J^1_{+,t})$ for all $s \geq \sigma^1$ and all $g \gg 0$, in a neighbourhood of $\{(t,x_+(t))\} \subset [0,1] \times M$.
} \\ & \label{eq:local-ex-3}
\mybox{Finally, let $x$ be a chord for $(H,J)$. Then there are $s^1$, $s^2$ such that $(H_{g,s,t}^2, J^2_{g,s,t}) = (H_t,J_t)$ for all $s \in [s^2,s^1 + g]$ and all $g \gg 0$, in a neighbourhood of $\{(t,x(t))\} \subset [0,1] \times M$.}
\end{align}
These conditions ensure that on suitable subsets $(-\infty,\sigma^2] \times [0,1]$, $[s^2, s^1 + g] \times [0,1]$, $[\sigma^1 + g,\infty) \times [0,1]$, the glued solutions $u_g^2$ for large $g$ satisfy a Floer equation, which can be linearized by a suitable change of coordinates. 

As usual, we're actually interested in a parametrized situation. This means that we have a family $(H^2_{r,s,t},J^2_{r,s,t})$ with an added parameter $r$; but we still have a single $(H^1_{s,t},J^1_{s,t})$, not depending on that parameter. Correspondingly, the glued data are of the form $(H^\nu_{g,r,s,t}, J^\nu_{g,r,s,t})$. The previously imposed conditions should hold for each value of the parameter. The specific gluing setup is quite similar to \eqref{eq:matching-us}, \eqref{eq:matching-us-2}, and \eqref{eq:lowest-eigenvalue}, but we repeat it for convenience:
\begin{align} 
&
\mybox{
Let $(r^2,u^2)$ be a solution of the parametrized continuation map equation for $(H^2_{r,s,t},J^2_{r,s,t})$, with limit $x$ as $s \rightarrow +\infty$, which is an isolated regular point of the parametrized moduli space.
}
\\ &
\mybox{
Let $u^1$ be a solution of the continuation map equation for $(H^1_{s,t},J^1_{s,t})$, with the same limit $x$ as $s \rightarrow -\infty$, and which is an isolated regular point of its moduli space. Moreover, let $\lambda^1$ be the highest negative eigenvalue of the selfadjoint operator associated to $x$. Then, we require that $\lim_{s \rightarrow -\infty} \exp(\lambda^1 s) \partial_s u^1$ is a nonzero $\lambda^1$-eigenveector.
}
\end{align}
There are glued solutions in the parametrized moduli space, and the exact counterpart of Corollary \ref{th:r-moves} holds for them.

\section{Homological algebra\label{sec:algebra}}
This section summarizes the necessary (quite basic) algebraic concepts. 

\subsection{$A_\infty$-structures}
Let $\scrA$ be a free graded abelian group. An $A_\infty$-algebra structure on it is given by maps
\begin{equation} \label{eq:mu}
\mu_{\scrA}^d: \scrA^{\otimes d} \longrightarrow \scrA[2-d], \;\; d \geq 1, 
\end{equation}
which satisfy the $A_\infty$-associativity conditions
\begin{equation} \label{eq:associativity}
\sum_{ij} (-1)^\dag \mu_{\scrA}^{d-j+1}(a_d,\dots,a_{i+j},\mu_{\scrA}^j(a_{i+j-1},\dots,a_i),a_{i-1},\dots,a_1) = 0,
\end{equation}
with $\dag = \|a_1\| + \cdots + \|a_{i-1}\|$ (our notation is that $\|a\| = |a|-1$ is the reduced degree). One can consider $\scrA$ as a chain complex with differential $da = (-1)^{|a|} \mu^1_{\scrA}(a)$. Then, 
\begin{equation}
(a_2,a_1) \longmapsto (-1)^{|a_1|} \mu^2_{\scrA}(a_2,a_1)
\end{equation}
is a chain map $\scrA^{\otimes 2} \rightarrow \scrA$, and induces an associative multiplication on $H^*(\scrA)$. All our $A_\infty$-structures will be required to be cohomologically unital, which means that $H^*(\scrA)$ is a unital algebra. 

\begin{remark}
Let $e \in \scrA^0$ be a representative for the cohomological unit. Then,
\begin{equation} \label{eq:left-mu}
a \longmapsto \mu^2_{\scrA}(a,e)
\end{equation}
is a chain map which is homotopic to its square. It is also a chain homotopy equivalence, because it's a quasi-isomorphism between complexes of free abelian groups. It follows that \eqref{eq:left-mu} is chain homotopic to the identity; and of course, the same is true for multiplication with $e$ on the other side. (This is for readers wondering why we have not required the existence of such homotopies as our unitality condition; it follows from cohomological unitality.)
\end{remark}

\subsection{$A_\infty$-bimodules}
An $\scrA$-bimodule is a graded free abelian group $\scrP$ together with operations
\begin{equation}
\mu^{l,1,k}_{\scrP}: \scrA^{\otimes l} \otimes \scrP \otimes \scrA^{\otimes k} \longrightarrow \scrP[1-k-l], 
\;\; k,l \geq 0,
\end{equation}
satisfying a version of \eqref{eq:associativity}. The operations $\mu_{\scrP}^{1,1,0}$ and $\mu_{\scrP}^{0,1,1}$ (with suitable signs) make $H^*(\scrP)$ into a bimodule over $H^*(\scrA)$, and we again impose a unitality condition on the cohomological level. Bimodules over $\scrA$ form a dg category. A morphism $\phi: \scrP \rightarrow \scrQ$ in that category is a collection of maps
\begin{equation} \label{eq:bimodule-map}
\phi^{l,1,k}: \scrA^{\otimes l} \otimes \scrP \otimes \scrA^{\otimes k} \longrightarrow \scrQ[-k-l], \;\;
k,l \geq 0.
\end{equation}
Consider an $A_\infty$-bimodule homomorphism, by which we mean a closed ($d\phi = 0$) degree $0$ morphism in our dg category. On the cohomology level, this induces a bimodule map $[\phi^{0,1,0}]: H^*(\scrP) \rightarrow H^*(\scrQ)$, and one defines the notion of quasi-isomorphism of bimodules using that.

\begin{lemma} \label{th:bimodule-quasi-iso}
Any quasi-isomorphism of $\scrA$-bimodules has a homotopy inverse. To spell this out: if $\phi: \scrP \rightarrow \scrQ$ is such a quasi-isomorphism, there is an $A_\infty$-bimodule homomorphism $\psi: \scrQ \rightarrow \scrP$ such that $\psi\phi = \mathit{id}_{\scrP} + d\alpha$, $\phi\psi = \mathit{id}_{\scrQ} + d\beta$ in the dg category of $\scrA$-bimodules.
\end{lemma}


From now on, assume that $\scrA$ is strictly proper (of finite rank over $\bZ$). The bimodules relevant for our purpose are the diagonal $\scrA$ and its linear dual $\scrA^\vee = \mathit{Hom}(\scrA,\bZ)$. Write
\begin{equation} \label{eq:quadratic-hochschild}
\mathit{CC}^*(\scrA,2) = \prod_{k,l \geq 0} \mathit{Hom}(\scrA^{\otimes l} \otimes \scrA^\vee \otimes \scrA^{\otimes k}, \scrA)[-k-l]
\end{equation}
for the chain complex of morphisms $\scrA^\vee \rightarrow \scrA$ in the bimodule category, and $\mathit{HH}^*(\scrA,2)$ for its cohomology (the notation comes from thinking of this as a generalization of Hochschild cohomology). If the components of $\psi \in \mathit{CC}^*(\scrA,2)$ are written as in \eqref{eq:bimodule-map}, the differential is
\begin{equation} \label{eq:quadratic-hochschild-differential}
\begin{aligned}
& (d\psi)^{l,1,k}(a_{k+l+1},\dots,a_{k+2},a^\vee_{k+1},a_k,\dots,a_1) = \\ 
& \sum_{ij} (-1)^{\dag_{(i)}} \mu^{k+l-j+i+2}_{\scrA}(a_{k+l+1},\dots,\psi^{j-k-1,1,k-i}(a_j,\dots,a_{k+2},a^\vee_{k+1},a_k,\dots,a_{i+1}),\dots,a_1) \\
& + \sum_{ij} (-1)^{\dag_{(ii)}} \psi^{l,1,k-j+1}(a_{k+l+1},\dots,a_{k+1}^\vee,\dots,a_{i+j+1},
\mu_{\scrA}^j(a_{i+j},\dots,a_{i+1}),a_i,\dots,a_1) \\
& + \sum_{ij} (-1)^{\dag_{(iii)}}\psi^{l-j+1,1,k}(a_{k+l+1},\dots,\mu_{\scrA}^j(a_{i+j},\dots,a_{i+1}),\dots,a^\vee_{k+1},a_k,\dots,a_1) \\
& + \sum_{ij} (-1)^{\dag_{(iv)}} \psi^{k+l-i-j+1,1,i}(a_{k+l+1},\dots,a_{i+j+1},
\\[-1.25em] & \qquad \qquad \qquad \qquad \qquad
\langle a_{k+1}^\vee, \mu^j_{\scrA}(a_k,\dots,a_{i+1},\bullet,a_{i+j},\dots,a_{k+2}) \rangle,a_i,\dots,a_1). 
\end{aligned}
\end{equation}
where $\bullet$ is an unspecified entry in $\scrA$, which creates an element of $\scrA^\vee$; the notation $\langle \cdot, \cdot \rangle$ is the canonical pairing between $\scrA^\vee$ and $\scrA$; and the signs are
\begin{equation} \label{eq:dag-signs}
\begin{aligned}
& \dag_{(i)} = |\psi|(\|a_1\|+\cdots+\|a_i\|), \\
& \dag_{(ii)} = |\psi|+1+\|a_1\|+\cdots+\|a_i\|, \\
& \dag_{(iii)} = |\psi|+1+\|a_1\|+\cdots+\|a_{k+1}^\vee\|+\cdots+\|a_i\|, \\
& \dag _{(iv)} = |\psi| +1+\|a_1\|+\cdots+\|a_{k+1}^\vee\|+\cdots+\|a_{i+j}\|.
\end{aligned}
\end{equation}

\begin{remark}
To be precise, the sign conventions in \eqref{eq:quadratic-hochschild-differential} do not quite describe an $A_\infty$-bimodule map $\scrA^\vee \rightarrow \scrA$ of degree $|\psi|$, but rather its shifted version $\scrA^\vee[1] \rightarrow \scrA[1]$. The structure of the shifted diagonal bimodule is particularly simple (see e.g.\ \cite[Section 2]{seidel14b}):
\begin{equation}
\mu_{\scrA[1]}^{l,1,k} = \mu_{\scrA}^{k+1+l}.
\end{equation}
For the dual, one has
\begin{equation}
\begin{aligned}
& \mu_{\scrA^\vee[1]}^{l,1,k}(a_{k+l+1},\dots,a_{k+2},a_{k+1}^\vee,a_k,\dots,a_1) \\
& \qquad
= (-1)^{|\bullet|}
\langle a_{k+1}^\vee, \mu_{\scrA}^{k+l+1}(a_k,\dots,a_1,\bullet,a_{k+l+1},\dots,a_{k+2}) \rangle,
\end{aligned}
\end{equation}
where of course $|\bullet|$ can be expressed in terms of the degrees of the $a$'s and of $a_{k+1}^\vee$ (which is why in part (iv) of \eqref{eq:dag-signs}, it does not appear explicitly).
\end{remark}


\subsection{Curved deformations}
A curved $A_\infty$-structure on a free graded abelian group $\scrA$ consists of operations
\begin{equation} \label{eq:mu-curved}
\mu_{\scrA_q}^d: \scrA^{\otimes d} \longrightarrow (\scrA[[q]])[2-d], \;\; d \geq 0,
\end{equation}
such that $\mu_{\scrA_q}^0$ has no constant term (vanishes when we specialize to $q = 0$). One extends \eqref{eq:mu-curved} $q$-linearly to operations on $\scrA_q = \scrA[[q]]$, which explains the notation. These operations must satisfy equations as in \eqref{eq:associativity}, but which now contain $\mu^0$ terms. For the $q = 0$ reduction, we impose the same cohomological unitality condition as before.

An $\scrA_q$-bimodule consists of a free graded abelian group $\scrP$ together with operations 
\begin{equation}
\mu^{l,1,k}_{\scrP_q}: \scrA^{\otimes l} \otimes \scrP \otimes \scrA^{\otimes k} \longrightarrow (\scrP[[q]])[1-k-l], \;\; k,l \geq 0.
\end{equation}
As before, one extends these to $\scrP_q = \scrP[[q]]$, and imposes suitable associativity conditions. We require cohomological unitality to hold after specializing to $q = 0$. Such bimodules form a dg category over $\bZ[[q]]$, with the analogue of \eqref{eq:bimodule-map} being
\begin{equation}
\phi^{l,1,k}_q: \scrA^{\otimes l} \otimes \scrP \otimes \scrA^{\otimes k} \longrightarrow (\scrQ[[q]])[-k-l].
\end{equation}
Take a bimodule homomorphism (a closed morphism in the dg category of $\scrA_q$-bimodules). We say that it is a filtered quasi-isomorphism if its $q = 0$ reduction is a quasi-isomorphism. The analogue of Lemma \ref{th:bimodule-quasi-iso} is:

\begin{lemma} \label{th:bimodule-quasi-iso-2}
Any filtered quasi-isomorphism of $\scrA_q$-bimodules has a homotopy inverse. To spell this out: if $\phi_q: \scrP_q \rightarrow \scrQ_q$ is a filtered quasi-isomorphism, there is an $A_\infty$-bimodule homomorphism $\psi_q: \scrQ_q \rightarrow \scrP_q$ such that $\psi_q\phi_q = \mathit{id}_{\scrP_q} + d\alpha_q$, $\phi_q\psi_q = \mathit{id}_{\scrQ} + d\beta_q$ in the dg category of $\scrA_q$-bimodules.
\end{lemma}

The previously mentioned diagonal and dual diagonal bimodules also exist in this context, and will be denoted by $\scrA_q$ and $\scrA_q^\vee$. Correspondingly, we have an analogue of \eqref{eq:quadratic-hochschild}, which is 
\begin{equation} \label{eq:quadratic-hochschild-2}
\mathit{CC}^*(\scrA_q,2) = \mathit{CC}^*(\scrA,2)[[q]],
\end{equation}
with a differential which is obtained by using $\mu_{\scrA_q}$ instead of $\mu_{\scrA}$ in \eqref{eq:quadratic-hochschild-differential}.

\subsection{$A_\infty$-categories}
In applications, we work with $A_\infty$-categories, which have a set $\mathit{Ob}(\scrA)$ of objects, a graded free abelian group $\hom_{\scrA}(X_0,X_1)$ associated to any two objects, and operations
\begin{equation}
\mu_{\scrA}^d: \hom_{\scrA}(X_{d-1},X_d) \otimes \cdots \otimes \hom_{\scrA}(X_0,X_1)
\longrightarrow \hom_{\scrA}(X_0,X_d)[2-d].
\end{equation}
The definitions of $A_\infty$-bimodules and morphisms are adjusted accordingly. For \eqref{eq:quadratic-hochschild}, this means
\begin{equation}
\begin{aligned}
& \mathit{CC}^*(\scrA,2) = 
\\ &
\!\!\! \prod_{\substack{k,l \geq 0 \\ X_0,\dots,X_{k+l+1}}} \!\!\!
\mathit{Hom}\big(\hom_{\scrA}(X_{k+l},X_{k+l+1}) \otimes \cdots \otimes \hom_{\scrA}(X_{k+1},X_{k+2}) 
\otimes \hom_{\scrA}(X_{k+1},X_k)^\vee \\[-1.75em] & \qquad \qquad  \qquad \otimes 
\hom_{\scrA}(X_{k-1},X_k) \otimes \cdots \otimes \hom_{\scrA}(X_0,X_1), 
\hom_{\scrA}(X_0,X_{k+l+1}) \big)[-k-l].
\end{aligned}
\end{equation}
The same applies to the curved case.

\section{The relative Fukaya category\label{sec:fukaya}}
This section is a recap of the definition of the relative Fukaya category for compact symplectic Calabi-Yau manifolds, following \cite{perutz-sheridan20} in all essential points (note that \cite{perutz-sheridan20} is itself a modified version of \cite{sheridan11b}; for a related construction of the absolute Fukaya category, see \cite{charest-woodward15, charest-woodward17}). 

\subsection{Geometric data\label{subsec:data}}
Let $M^{2n}$ be a closed symplectic manifold, together with a symplectic ample divisor. By that, we mean the following:

\begin{setup} \label{th:ample}
We are given a symplectic submanifold $\Omega_M^{2n-2} \subset M$ which represents $[\omega_M]$, together with a $\theta_M \in \Omega^1(M \setminus \Omega_M)$, satisfying $d\theta_M = \omega_M|(M \setminus \Omega_M)$, and such that: if $S$ is a compact oriented surface with boundary, and $u: S \rightarrow M$ a map with $u(\partial S) \cap \Omega_M = \emptyset$, then
\begin{equation} \label{eq:stokes}
\textstyle \int_S u^*\omega_M = \int_{\partial S} u^*\theta_M + (u \cdot \Omega_M). 
\end{equation}
 All Lagrangian submanifolds $L \subset M$ we consider are disjoint from $\Omega_M$, and exact, $[\theta_M|L] = 0 \in H^1(L)$. We will only use compatible almost complex structures $J$ on $M$ which make $\Omega_M$ into an almost complex submanifold. Similarly, we use Hamiltonians $H \in \smooth(M,\bR)$ such that the associated vector field $X$ is tangent to $\Omega_M$.
\end{setup}

Suppose that we have two such submanifolds $(L_0,L_1)$. A choice of Floer datum consists of a time-dependent Hamiltonian $(H_t)$, $t \in [0,1]$, such that the chords \eqref{eq:chord} are nondegenerate (the properties of the Lagrangian submanifolds and Hamiltonian vector fields imply that all such chords remain in $M \setminus \Omega_M$). We also choose a family $J_t$ of almost complex structures. That allows us to write down the relevant Floer equation \eqref{eq:floer}. 

The more general class of surfaces under consideration will be as follows. Take a Riemann surface $\bar{S}$ isomorphic to the closed disc, and a finite nonempty set $\Sigma \subset \partial\bar{S}$, arbitrarily divided into a positive and negative part, $\Sigma = \Sigma_- \cup \Sigma_+$. We form
\begin{equation} \label{eq:s}
S = \bar{S} \setminus \Sigma.
\end{equation}
For each connected component $C \subset \partial S$ we choose a Lagrangian submanifold $L_C$. Given any $\zeta \in \Sigma$, we have a pair of Lagrangian submanifolds $(L_{\zeta,0},L_{\zeta,1})$, which are the $L_C$ associated to the boundary components adjacent to $\zeta$. More precisely, if $\zeta \in \Sigma_+$, then $L_{\zeta,0}$ precedes $\zeta$ in the boundary orientation, and $L_{\zeta,1}$ follows it; while the convention for $\zeta \in \Sigma_-$ is the opposite. We assume that a Floer datum $(H_\zeta,J_\zeta)$ has been chosen for each of these pairs of Lagrangian submanifolds.

\begin{setup} \label{th:ends-etc}
A set of strip-like ends for $S$ consists of proper holomorphic embeddings 
\begin{equation}  \label{eq:ends}
\epsilon_{\zeta}:
\left\{
\begin{aligned}
&
(-\infty,\sigma_\zeta] \times [0,1] \longrightarrow S, && \textstyle \lim_{s \rightarrow-\infty} \epsilon_{\zeta}(s,t) = \zeta\;\; \text{ for } \zeta \in \Sigma_-, \\
&
[\sigma_\zeta,\infty) \times [0,1] \longrightarrow S, && \textstyle \lim_{s \rightarrow +\infty} \epsilon_{\zeta}(s,t) = \zeta\;\; \text{ for } \zeta \in \Sigma_+.
\end{aligned}
\right.
\end{equation}
Take a family $(J_z)_{z \in S}$. We also need a one-form $K \in \Omega^1(S, \smooth(M,\bR))$, with the property that 
\begin{equation} \label{eq:z-tangent}
\mybox{
if $Z \in T_z(\partial S)$, $z \in C \subset \partial S$, then $K(Z)|L_C = 0$. 
}
\end{equation}
To clarify the context: the almost complex structures, as well as $K(Z)$ for any $Z$, must belong to the classes described in Setup \ref{th:ample}. Over the ends (as $s \rightarrow \pm\infty$), we impose asymptotic conditions
\begin{equation} \label{eq:jk-converge}
\left\{
\begin{aligned}
& J_{\epsilon_{\zeta}(s,t)} \longrightarrow J_{\zeta,t}, \\
& \epsilon_{\zeta}^*K \longrightarrow H_{\zeta,t}\, \mathit{dt}.
\end{aligned}
\right.
\end{equation}
The convergence is understood to be exponentially fast, in the same sense as in \eqref{eq:cont-family}.
\end{setup}

Let $Y$ be the vector-field-valued one-form associated to $K$, which means that for $Z \in TS$, $Y(Z)$ is the Hamiltonian vector field of $K(Z)$. We are now ready to write down the relevant Cauchy-Riemann equation: 
\begin{equation} \label{eq:cauchy-riemann}
\left\{
\begin{aligned} 
& u: S \longrightarrow M, \\
& u(C) \subset L_C, \\
& \textstyle \lim_{s \rightarrow \pm\infty} u(\epsilon_{\zeta}(s,t)) = x_\zeta(t), \\
& (du - Y)^{0,1} = 0,
\end{aligned}
\right.
\end{equation}
where the $x_\zeta$ are appropriate chords. In local holomorphic coordinates $z = s+it$ on $S$, the last line of \eqref{eq:cauchy-riemann} is
\begin{equation} \label{eq:cauchy-riemann-explicit}
\partial_t u - Y_{s,t}(\partial_t) = J_{s,t} \big(\partial_s u - Y_{s,t}(\partial_s) \big).
\end{equation}
For fixed data, an application of \eqref{eq:stokes} provides a bound on the energy, of the form
\begin{equation}
E(u) = \textstyle \half \int_S \| du - Y \|^2 \leq (u \cdot \Omega_M) + \text{\it constant}.
\end{equation}

In our applications, the surfaces will carry a finite set of interior marked points $\Xi \subset S \setminus \partial S$, and we impose the following intersection constraints:
\begin{equation} \label{eq:adjacency-conditions}
\left\{
\begin{aligned} 
& u \cdot \Omega_M = |\Xi|, \\
& u(\xi) \in \Omega_M \;\; \text{for all $\xi \in \Xi$.}
\end{aligned}
\right.
\end{equation}
By Gromov's trick, one can convert \eqref{eq:cauchy-riemann-explicit} into a straight pseudo-holomorphic map equation for the graph $(z,u(z)): S \longrightarrow S \times M$. Because of the assumptions on $J$ and $K$, the relevant almost complex structure on $S \times M$ has the property that $S \times \Omega_M$ is an almost complex submanifold. This allows us to apply positivity of intersections in the classical sense: $u^{-1}(\Omega_M)$ is a finite set; every point in it contributes positively to $u \cdot \Omega_M$; and the contribution is $1$ exactly when the intersection is transverse. Hence:

\begin{lemma} \label{th:positivity}
All solutions of \eqref{eq:cauchy-riemann}, \eqref{eq:adjacency-conditions} satisfy $u^{-1}(\Omega_M) = \Xi$; and moreover, at each point of $u^{-1}(\Omega_M)$ the intersection is transverse.
\end{lemma}

\subsection{Surfaces with one negative end\label{subsec:one-negative}}
We now introduce the more specific class of maps which enters into the definition of the relative Fukaya category. Namely, for some $d \geq 0$, take
\begin{equation} \label{eq:disc-surface}
\begin{aligned}
&
\bar{S} = D \subset \bC \text{ (closed unit disc)}, \;\; \Sigma _-= \{\zeta_0\}, \; \;\Sigma_+= \{\zeta_1,\dots,\zeta_d\}, 
\\ & \qquad
\text{where }
\left\{
\begin{aligned}
& \zeta_0 = 1, \\
& \zeta_1 = e^{i \rho_1},\; \dots,\; \zeta_d = e^{i \rho_d}, \quad 0< \rho_1 < \cdots < \rho_d <2\pi.
\end{aligned}
\right.
\end{aligned}
\end{equation}
The resulting $S = \bar{S} \setminus \Sigma$ is called a $(d+1)$-punctured disc. Since the punctures are numbered, the associated ends can be denoted by $\{\epsilon_0,\dots,\epsilon_d\}$. We can also number the boundary components as $\partial S = C_0 \cup C_1 \cup \cdots \cup C_d$, starting with that between $\zeta_0$ and $\zeta_1$. Hence, the labeling of components with Lagrangian submanifolds is a choice of $L_0,\dots,L_d$. Additionally, our surface comes with a set of $m \geq 0$ of interior marked points, which we order, $\Xi = \{\xi_1,\dots,\xi_m\}$.

We need to work with varying $S$. Let $\scrR^{d+1;m}$ be the moduli space of $(d+1)$-punctured discs with $m \geq 1-d/2$ interior marked points. This is a (real) manifold with
\begin{equation} \label{eq:stability}
\mathrm{dim}(\scrR^{d+1;m}) = d+2m -2.
\end{equation}
The group $\mathit{Sym}_m$ acts freely on it, by permuting the interior marked points. There is a natural compactification $\bar\scrR^{d+1;m}$, which is a manifold with corners, still carrying an action of $\mathit{Sym}_m$ (which is no longer free). The compactification can be constructed by embedding our space into (the real locus of) the Deligne-Mumford moduli space of genus $0$ curves with $d+1+2m$ marked points. A stratum of $\bar\scrR^{d+1;m}$ is described by the following combinatorial structure. We have a planar tree $T$ with $(d+1)$ semi-infinite edges, one of which is distinguished as the root. This additionally comes with a partition of $\{1,\dots,m\}$ into subsets $\Xi^v$ corresponding to the vertices, and we must have
\begin{equation} \label{eq:valence-inequality}
|v| + 2|\Xi^v| \geq 3.
\end{equation}
The points in such a stratum are geometrically described as follows: to each vertex $v$ corresponds a genus zero nodal surface $S^v$, which is a disc with $|v|$ boundary punctures (called the principal component) together with other, spherical, irreducible components attached to it in tree-like patterns (each tree is attached to an interior point of $S^v$, and no interior point appears more than once). Furthermore, each $S^v$ has smooth interior marked points labeled by $\Xi^v$. Finally, there is the usual no-automorphisms stability condition (if there are no spheres, this follows from \eqref{eq:valence-inequality}; in the general case, it restricts the structure of the tree of spheres). Figure \ref{fig:r12} shows the two-dimensional moduli space $\bar\scrR^{1;2}$, which is the first one to include a surface with a spherical irreducible component.
\begin{figure}
\begin{centering}
\includegraphics[scale=0.8]{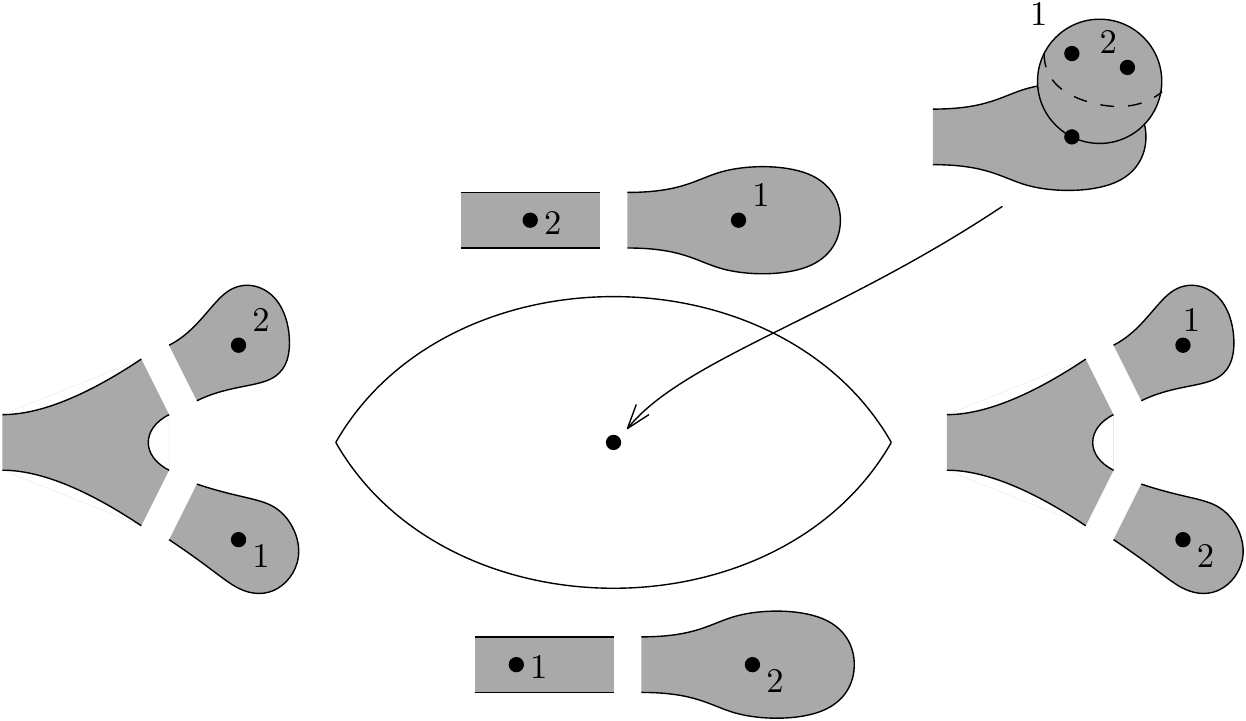}
\caption{The space $\bar\scrR^{1;2}$ is a bi-gon.\label{fig:r12}}
\end{centering}
\end{figure}

To construct the necessary moduli spaces of maps with varying domains, we first choose a Floer datum for every pair of Lagrangian submanifolds. Then we choose ends $\epsilon$, as well as data $(J,K)$ (asymptotic to the previously chosen Floer data on the ends), for each fibre $S$ of the universal family of surfaces over $\scrR^{d+1;m}$. We require that these are $\mathit{Sym}_m$-invariant. Moreover, there is a consistency condition, which 
takes the form of an extension to $\bar\scrR^{d+1;m}$, with the following properties.
If a component $S^v$ contains no spheres, the almost complex structures and the inhomogeneous term on $S^v$ should agree with those made when considering it as representing a point of $\scrR^{|v|;|\Xi^v|}$ (which makes sense thanks to the assumption of $\mathit{Sym}_{|\Xi^v|}$-invariance). More generally, on each disc part of an $S^v$, the auxiliary data only depend on the moduli of that part with its special points (interior marked points, or nodal points where the spheres are attached); while on the sphere parts the inhomogeneous term is zero, and the family of almost complex structures is constant, agreeing with the almost complex structure associated to the point on the disc where the tree of sphere is attached. This is not a full description of the consistency condition, one has to be more precise about what happens on the thin parts (long strips or cylinders) as a surface degenerates; but since that is a familiar issue, and as special case of it was already discussed in \eqref{eq:glue-the-data-1}--\eqref{eq:glue-the-data-3}, we will not review it here.

The outcome of our construction are spaces $\scrR^{d+1;m}(x_0,\dots,x_d)$, which parametrize the choice of a point of $\scrR^{d+1;m}$ together with a solution of \eqref{eq:cauchy-riemann} on the surface $S$ determined by that point. These again carry free $\mathit{Sym}_m$-actions. We extend the notation by taking $\scrR^{2;0}(x_0,x_1)$ to be the space of those Floer trajectories which avoid $\Omega_M$, divided by translation. The transversality issue is relatively straightforward (because of the freeness of the $\mathit{Sym}_m$-action), leading to:

\begin{lemma}
For generic choices of the auxiliary data (almost complex structures, Hamiltonians) satisfying the consistency condition outlined above, all spaces $\scrR^{d+1;m}(x_0,\dots,x_d)$ are regular. 
\end{lemma}

We denote by $\bar\scrR^{d+1;m}(x_0,\dots,x_d)$ the standard stable map compactification. The main issue is to analyze the structure of the codimension $<2$ strata in the compactification (for generic choices of auxiliary data). As usual, the underlying philosophy is that only moduli spaces of dimension $\leq 1$ are used in the construction of the Fukaya category, so codimension $\geq 2$ phenomena can be excluded from consideration. 

\subsection{Disc bubbles}
Let's suppose that we have a sequence $(u_k: S_k \rightarrow M)_{k = 1,2,\dots}$ in some space $\scrR^{d+1;m}(x_0,\dots,x_d)$. After passing to a subsequence, we have convergence in the stable map sense. The combinatorial structure of the limit is similar to that considered before, with a planar rooted tree $T$ and nodal surfaces $S^v$ which come with maps $u^v: S^v \rightarrow M$. However, the collection $\{S^v\}$ is more complicated than the limit of the sequence $(S_k)$ in $\bar\scrR^{d+1;m}$. First of all, there can be additional two-valent vertices $v$ with $\Xi^v = \emptyset$, which carry solutions of Floer's equation. Moreover, each $S^v$ can have additional irreducible components, which are collapsed when passing to the limit in $\bar\scrR^{d+1;m}$. A priori these irreducible components, which we call bubbles, can be either spheres or discs. The discs would appear in a tree-like pattern attached to a boundary nodal point of the main component of $S^v$ (this is completely different from the boundary-puncture issue; the nodal points are not being removed, and on each disc bubble component, $u^v$ is pseudo-holomorphic with no inhomogeneous term). In fact, we can exclude that for easy topological reasons:

\begin{lemma} \label{th:disc-bubbling}
No disc bubbles can occur in our limits.
\end{lemma}

\begin{proof}[Sketch of proof]
Take a compact part $B_k \subset S_k$, for $k \gg 0$, which (in somewhat imprecise words) surrounds the point where disc bubbling will occur; in Figure \ref{fig:disc-bubbling}, $B_k$ is indicated by the lighter shading. On the dividing curve between $B_k$ and the rest of $S_k$, the map $u_k$ will be nearly constant at a point of one of our Lagrangian submanifolds. In particular, the restriction of $u_k$ to that curve avoids $\Omega_M$. Because of the conditions in \eqref{eq:cauchy-riemann}, it is clear that
\begin{equation}
(u_k|B_k) \cdot \Omega_M = 0.
\end{equation}
That will then imply the corresponding property for the bubble tree that appears in the limit. However, this is impossible, since that tree needs to include at least one irreducible component on which the pseudo-holomorphic map is non-constant (that component has positive energy, hence by \eqref{eq:stokes} must have positive intersection number with $\Omega_M$).
\end{proof}
\begin{figure}
\begin{centering}
\includegraphics[scale=0.8]{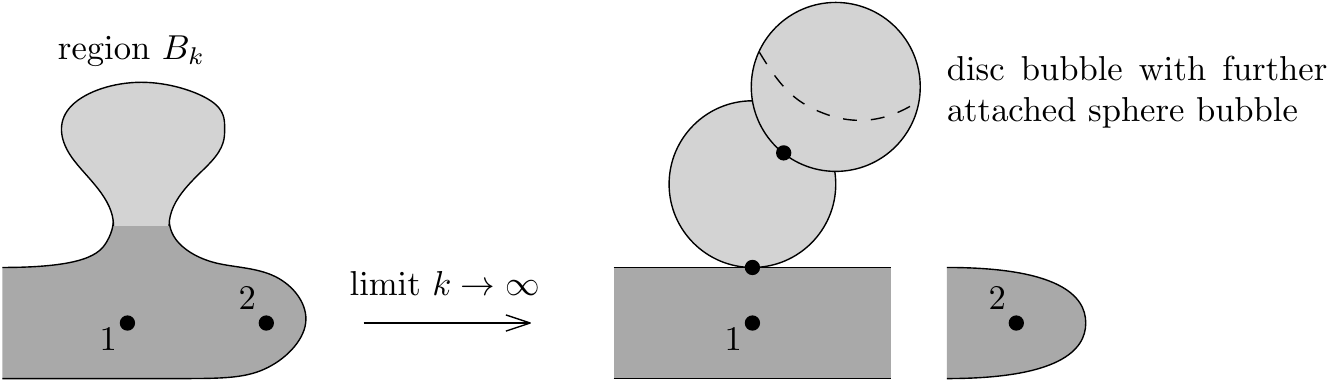}
\caption{\label{fig:disc-bubbling}An example (not the simplest one, since there are two bubble components) of the potential disc-bubbling behaviour, ruled out by Lemma \ref{th:disc-bubbling}.}
\end{centering}
\end{figure}

\subsection{Constant sphere bubbles}
Another fairly straightforward situation which can arise in the stable map limit is the following. Suppose that $S^v$ has a constant tree of spheres attached to the principal component (at some interior point). Constant means that the map $u^v$ is constant on all of the spheres in the tree. Because of the stable map condition, this implies that our tree of spheres must contain at least two of the interior marked points (the points labeled by $\Xi^v \subset \{1,\dots,m\}$); which also means that the constant map takes values at some point of $\Omega_M$. One can say a little more:

\begin{lemma} \label{th:constant-bubbles}
Suppose that $S^v$ has a constant tree of spheres attached to the principal component at some interior point, and that tree contains $l \geq 2$ points of $\Xi^v$. Then, the restriction of $u^v$ to the principal component has $l$-fold intersection multiplicity (in other words, order of tangency $l-1$) with $\Omega_M$ at the attaching point.
\end{lemma}

\begin{proof}[Sketch of proof]
The restriction of $u^v$ to the principal component can't be contained in $\Omega_M$, hence has finitely many intersection points with that submanifolds. Draw a small loop around the attaching point. That loop cuts off a small disc with the tree of spheres attached. Smallness means that other than the nodal point, the disc doesn't contain any intersection points of the principal component with $\Omega_M$. Look (again in somewhat vague terminology) at the corresponding region $B_k \subset S_k$, $k \gg 0$, as schematically drawn in Figure \ref{fig:constant-bubble}. By definition,
\begin{equation}
(u_k|B_k) \cdot \Omega_M = \# \{ \text{points $\xi_i$ contained in $B_k$} \}.
\end{equation}
In the limit, the intersection number of the small disc-plus-tree-of-spheres is the same, namely $l$ in the terminology of our statement. But since the spheres come with constant maps, and on the principal component only the nodal point contributes, this number is the multiplicity of intersection at that point.
\end{proof}
\begin{figure}
\begin{centering}
\includegraphics[scale=0.8]{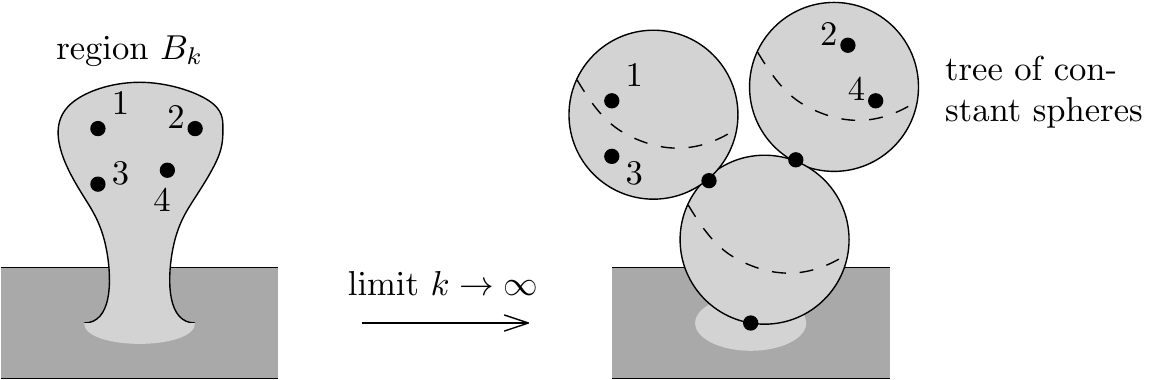}
\caption{\label{fig:constant-bubble}Constant spheres leading to higher tangencies, see Lemma \ref{th:constant-bubbles} (again, this is not the simplest possible example of such behaviour).}
\end{centering}
\end{figure}

Via transversality techniques from \cite{cieliebak-mohnke07} applied to the disc component of $S^v$, Lemma \ref{th:constant-bubbles} implies that constant sphere bubbles are a codimension $2$ phenomenon.

\subsection{Non-constant sphere bubbles}
At this point, we need to add a partial Calabi-Yau condition:

\begin{setup} \label{th:calabi-yau}
We assume that $c_1(M)|\pi_2(M)= 0$.
\end{setup}

Note that then, $c_1(\Omega_M) = - \omega_M|\Omega_M$ on $\pi_2(\Omega_M)$. As a classical consequence:

\begin{lemma}
For generic $J$ (in the class from Setup \ref{th:ample}), the following are true. (i) Consider non-constant $J$-holomorphic spheres in $M$ which are not entirely contained in $\Omega_M$. Then, the subset of $M$ consisting of points lying on such a sphere is of codimension $\geq 4$ (which means, it is contained in the image of a map from a manifold of dimension $\leq 2n-4$ to $M$). Moreover, the subset of $\Omega_M$ consisting of points lying on the same kind of spheres (spheres that are not contained in $\Omega_M$) is of codimension $\geq 4$ inside $\Omega_M$ (contained in the image of a manifold of dimension $\leq 2n-6$). (ii) Consider non-constant $J$-holomorphic spheres contained in $\Omega_M$. Then, the subset of $\Omega_M$ consisting of points which lie on such a sphere is of codimension $\geq 6$ inside $\Omega_M$ (meaning, contained in the image of a map from a manifold of dimension $\leq 2n-8$ to $\Omega_M$).
\end{lemma}

This is not strictly speaking enough for us; instead, we will use a mild generalization to families of almost complex structures. The outcome is that if we consider limits in which some $S^v$ contains a non-constant tree of spheres (meaning, a tree of spheres at least one of which is not constant), then that phenomenon is again of codimension $\geq 2$ in our moduli spaces. This completes our discussion of transversality for the compactification: for generic choices, the only codimension $\leq 1$ degenerations are those where the surface $S$ splits into two punctured discs (one of which can be a Floer strip), with no other irreducible components.

\subsection{Completing the definition\label{subsec:define-fukaya}}
At this point, we add further specifics:

\begin{setup} \label{th:branes}
We assume that $c_1(M) = 0$, and fix a trivialization of the canonical bundle $K_M$ (for some compatible almost complex structure). We only allow Lagrangian submanifolds which (in addition to the previous assumptions) are graded, and which come equipped with {\em Spin} structures. 
\end{setup}

Such Lagrangian submanifolds will be the objects of the relative Fukaya category, which is a $\bZ$-graded curved $A_\infty$-category, denoted here simply by $\scrB_q$. The morphism spaces are Floer complexes,
\begin{equation}
\hom_{\scrB_q}(L_0,L_1) = \mathit{CF}^*(L_0,L_1; H_{L_0,L_1})[[q]] = \bigoplus_x \bZ_x[[q]],
\end{equation}
where $\bZ_x \iso \bZ$ is the orientation space associated to a chord $x$. The $A_\infty$-structure is 
\begin{equation} \label{eq:mu-q}
\mu_{\scrB_q}^d(x_d,\dots,x_1) = \sum_{m \geq 0} \pm \# (\scrR^{d+1;m}(x_0,\dots,x_d)/\mathit{Sym}_m) \, q^m x_0.
\end{equation}
Here, $\pm \#$ is the usual signed count of points in zero-dimensional moduli spaces.

\subsection{The diagonal class\label{subsec:diagonal-class}}
So far, we have only used Riemann surfaces with one negative end. The next step beyond that would be as follows. Take, for some $k,l \geq 0$,
\begin{equation} \label{eq:z-surface}
\begin{aligned}
&
\bar{S} = (\bR \times [0,1]) \cup \{\pm\infty\}, \;\; \Sigma_- = \{\zeta_0,\zeta_{k+1}\}, \;\;
\Sigma_+ = \{\zeta_1,\dots,\zeta_k,\zeta_{k+2},\dots,\zeta_{k+l+1}\},
\\
&
\qquad \text{ where } \left\{
\begin{aligned}
& \zeta_0 = -\infty,\\ & \zeta_1,\dots,\zeta_k \in \bR \times \{0\} \;\text{in increasing order of real parts, } \\
& \zeta_{k+1} = +\infty, \\ & \zeta_{k+2},\dots,\zeta_{k+l+1} \in \bR \times \{1\} \;\text{in decreasing order of real parts.}
\end{aligned}
\right.
\end{aligned}
\end{equation}
As before, we also want to have interior marked points $\{\xi_1,\dots,\xi_m\}$. 
\begin{figure}
\begin{centering}
\includegraphics[scale=0.8]{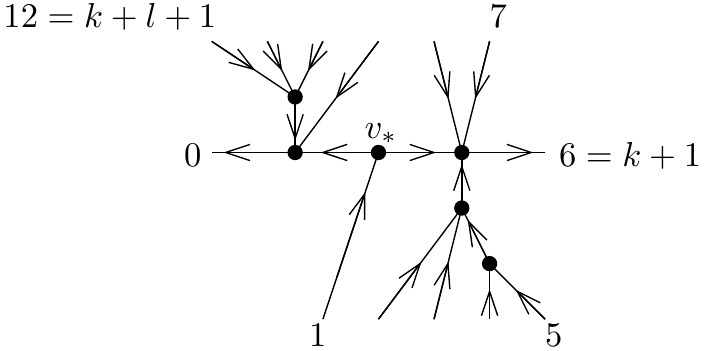}
\caption{\label{fig:orient-edge}A tree parametrizing a stratum of $\bar\scrR^{6,7;m}$.}
\end{centering}
\end{figure}%
The relevant moduli space $\scrR^{k+1,l+1;m}$ parametrizes the choices of $\zeta$'s and $\xi$'s, but without dividing by $\bR$-translation, which means that
\begin{equation}
\mathrm{dim}(\scrR^{k+1,l+1;m}) = k+l+2m.
\end{equation}
The compactification $\bar\scrR^{k+1,l+1;m}$ can be constructed inside the real part of the Fulton-MacPherson space of ordered configurations of $k+l+2+2m$ points on $\bC P^1$. Combinatorially, its strata are parametrized by planar trees $T$ with $(k+l+2)$ semi-infinite edges, which are numbered compatibly with their cyclic ordering. The $0$-th and $(k+1)$-st semi-infinite edge are declared to be roots (called the $0$-root and $(k+1)$-root), and the others leaves.  As before, this tree comes with a decomposition of $\{1,\dots,m\}$ into subsets $\Xi^v$. Additionally, we single out a vertex $v_*$, which must lie on the path from one root to the other. The choice of $v_*$ determines a preferred orientation of all edges in $T$, as indicated in Figure \ref{fig:orient-edge}. We can write the valence of $v_*$ as $|v_*| = k_* + l_* + 2$, where $k_*$ is the number of incoming edges which, in the cyclic ordering, come after the edge going to the $0$-root but before the edge going to the $(k+1)$-root. Finally, the condition \eqref{eq:valence-inequality} applies to all vertices $v \neq v_*$. Geometrically, to each $v \neq v_*$ we associate a disc with $|v|$ boundary punctures, possibly with trees of spheres attached to it, as in our previous description of $\bar\scrR^{d+1,m}$; and to $v_*$ a surface \eqref{eq:z-surface} with $(k_*,l_*)$ instead of $(k,l)$, again with additional spherical irreducible components. All those surfaces come with smooth interior points labeled by $\Xi^v$. Note that for $v_*$, we again do not identify surfaces which differ by an $\bR$-translation. Correspondingly, the usual no-automorphism stability condition applies to all irreducible components except for the principal component of $S^{v_*}$. Figure \ref{fig:double-r-example} shows the two-dimensional space $\bar\scrR^{1,1;1}$.
\begin{figure}
\begin{centering}
\includegraphics[scale=0.8]{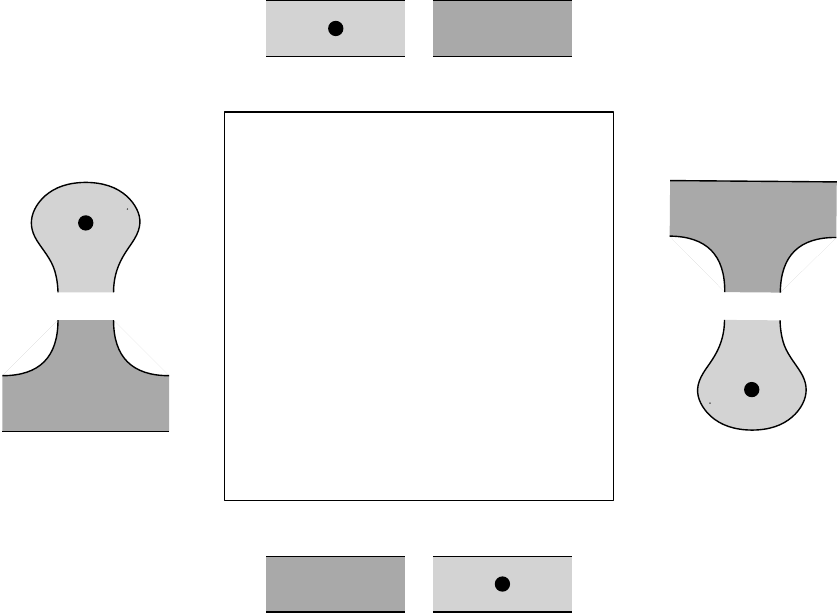}
\caption{\label{fig:double-r-example}The space $\bar\scrR^{1,1;1}$. The darker shaded parts are the surfaces corresponding to the distinguished vertex $v_*$.}
\end{centering}
\end{figure}%

\begin{remark}
These moduli spaces of Riemann surfaces (leaving interior points aside) are not new. They appear whenever one tries to construct a homomorphism of $A_\infty$-bimodules geometrically, for instance in the ``two-pointed closed-open string map'' from \cite[Section 5.6]{ganatra12}.
\end{remark}

Given a surface \eqref{eq:z-surface}, we label the boundary components with Lagrangians $L_0,\dots,L_{k+l+1}$. Fix appropriate data as in Setup \ref{th:ends-etc}. Given that and the interior marked points, we can consider solutions of associated Cauchy-Riemann equation \eqref{eq:cauchy-riemann} satisfying \eqref{eq:adjacency-conditions}. In the simplest case $k = l = m = 0$, the surface is an infinite strip, and what we have is equivalent to a continuation map equation relating $H_{L_0,L_1,t}$ and $-H_{L_0,L_1,1-t}$. Algebraically, the outcome is an element
\begin{equation} \label{eq:floer-delta}
\begin{aligned}
 \delta_{L_0,L_1} & \, \in \big( \mathit{CF}^*(L_0,L_1;H_{L_0,L_1})
\otimes \mathit{CF}^*(L_1,L_0;H_{L_1,L_0}) \big)^n \\
& = \mathit{Hom}\big(\mathit{CF}^*(L_1,L_0;H_{L_1,L_0})^\vee, \mathit{CF}^*(L_0,L_1;H_{L_0,L_1})\big)^n
\\ & = \mathit{Hom}\big( \mathit{CF}^*(L_0,L_1;-H_{L_0,L_1,1-t}), \mathit{CF}^*(L_0,L_1;H_{L_0,L_1,t}) \big),
\end{aligned}
\end{equation}
which is a cocycle for the differential induced by the Floer differentials. To clarify, the Floer differentials involved here, and the definition of \eqref{eq:floer-delta}, use only holomorphic curves in the complement of $\Omega_M$ (coming from surfaces with no interior marked points). To get to the last line of \eqref{eq:floer-delta}, we have used Poincar{\'e} duality in Floer theory. It is a standard fact about continuation maps that:

\begin{lemma} \label{th:delta-continuation}
If we think of $\delta_{L_0,L_1}$ as a map between Floer chain complexes, as in the second or third line of \eqref{eq:floer-delta}, then it is a quasi-isomorphism.
\end{lemma}

For the general construction, we need to make auxiliary choices consistently over all $\scrR^{k+1,l+1;m}$, in exactly the same way as before. This gives rise to moduli spaces $\scrR^{k+1,l+1;m}(x_0,\dots,x_{k+l+1})$. For generic choices, counting points in the zero-dimensional moduli spaces, as in \eqref{eq:mu-q}, yields
\begin{equation} \label{eq:psi-components}
\begin{aligned}
& \delta^{l,1,k}_q \in 
\mathit{Hom}\big( \mathit{CF}^*(L_{k+l},L_{k+l+1}) \otimes \cdots \otimes \mathit{CF}^*(L_{k+1},L_{k+2}) 
\otimes \mathit{CF}^*(L_{k+1},L_k)^\vee \\ & \qquad \qquad  \qquad \otimes 
\mathit{CF}^*(L_{k-1},L_k) \otimes \cdots \otimes \mathit{CF}^*(L_0,L_1), 
\mathit{CF}^*(L_0,L_{k+l+1}) \big)^{n-k-l}[[q]].
\end{aligned}
\end{equation}
Here, the choices of Hamiltonians have been suppressed for brevity. The $\delta^{l,1,k}_q$ satisfy equations governed by the codimension $1$ strata of the compactification $\bar\scrR^{k+1,l+1;m}(x_0,\dots,x_{k+l+1})$. These can be written in terms of \eqref{eq:quadratic-hochschild-differential} as $(d\delta_q)^{l,1,k} = 0$. The conclusion is that the collection of all \eqref{eq:psi-components} yields a cocycle
\begin{equation} \label{eq:psi-class}
\delta_q \in \mathit{CC}^n(\scrB_q, 2).
\end{equation}
By definition, $\delta_q$ constitutes a bimodule map $\scrB_q^\vee[-n] \rightarrow \scrB_q$, and Lemma \ref{th:delta-continuation} says that it is a filtered quasi-isomorphism. We have therefore proved the following (well-known) weak Calabi-Yau property of the relative Fukaya category:

\begin{cor}
The diagonal bimodule $\scrB_q$ and its shifted dual $\scrB_q^\vee[-n]$ are filtered quasi-isomorphic.
\end{cor}

\subsection{Sign issues for the diagonal class}
Let's outline how one gets from geometry to the signs in \eqref{eq:quadratic-hochschild-differential}. We use the same orientations of the Stasheff moduli spaces $\scrR^{d+1}$ as in \cite[Section 12g]{seidel04}. For the spaces $\scrR^{k+1,l+1}$, we choose the orientation given by (the real parts of) $(\zeta_1,\dots,\zeta_k,-\zeta_{k+2},\dots,-\zeta_{k+l+1})$, in this order.
\begin{figure}
\begin{centering}
\includegraphics[scale=0.8]{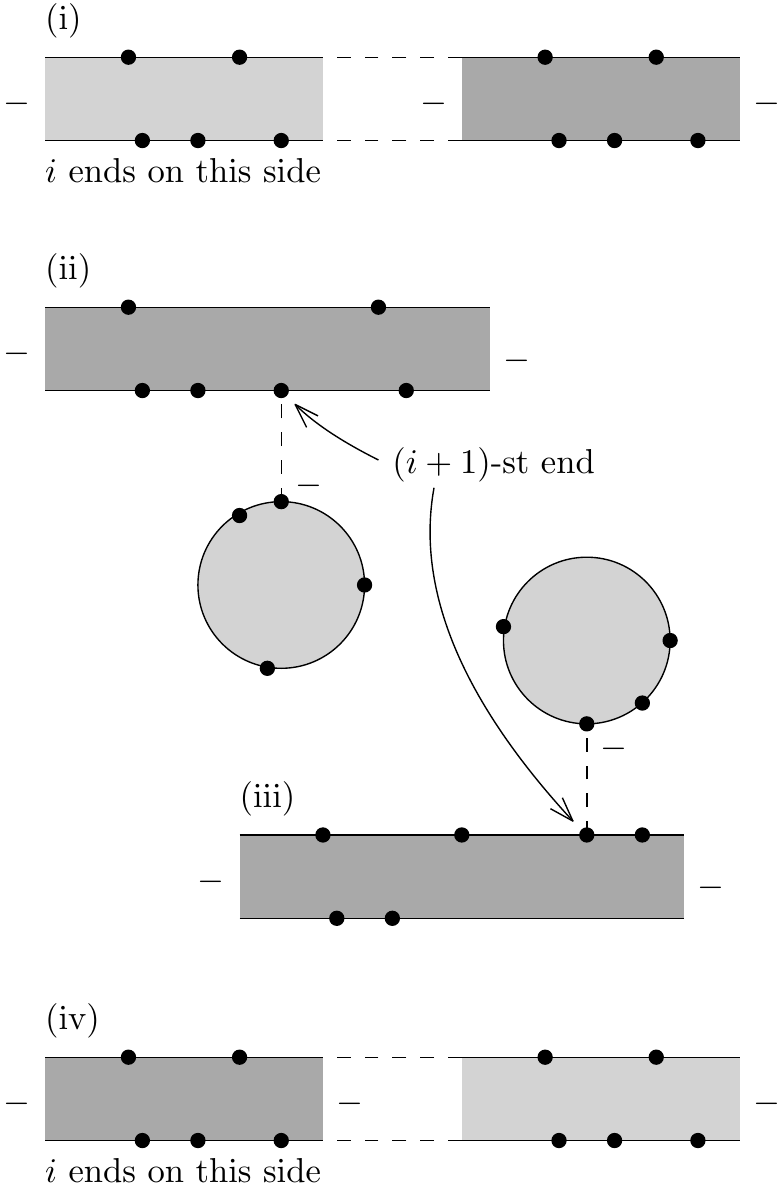}
\caption{\label{fig:4-bubbles}Codimension one boundary faces in $\bar\scrR^{k+1,l+1}$. 
As in Figure \ref{fig:double-r-example}, the darker shading singles out the components corresponding to $v_*$. Those are always drawn so that $\zeta_0$ is on the left and $\zeta_{k_*+1}$ on the right, corresponding to the coordinates in \eqref{eq:z-surface}. We have drawn some ends as boundary punctures, to simplify the picture. Negative ends are marked with a $-$.}
\end{centering}
\end{figure}

\begin{lemma} \label{th:codim-1}
The orientations of codimension one boundary strata in $\bar\scrR^{k+1,l+1}$ induced by their natural product structure differ from those coming from the orientation of the interior by a sign $(-1)^\ast$, where, following the notation from \eqref{eq:quadratic-hochschild-differential} and Figure \ref{fig:4-bubbles}: 

(i) the stratum is $\scrR^{k+l-j+i+3} \times \scrR^{k-i+1,j-k}$, and $\ast_{(i)} = 
(i+1)(j+k+l) + j(k+l) + 1$; \newline
(ii) the stratum is $\scrR^{k-j+2,l+1} \times \scrR^{j+1}$, and $\ast_{(ii)} = j(i+k+l) + i + j$; \newline
(iii) the stratum is $\scrR^{j+1} \times \scrR^{k+1,l-j+2}$, and 
$\ast_{(iii)} = j(i+k+l) + i + 1$;
\newline
(iv) the stratum is $\scrR^{i+1,k+l-i-j+2} \times \scrR^{j+1}$, and 
$\ast_{(iv)} = jl + j + k + 1$.
\end{lemma}

We will not reproduce the calculations here, because of their elementary nature. (Note that there is a symmetry $\bar\scrR^{k+1,l+1} \iso \bar\scrR^{l+1,k+1}$, which identifies orientations up to $(-1)^{kl}$, and swaps the four cases in pairs; one can use that to simplify the computation.) Interior marked points and the resulting constraints don't really affect sign issues, because they are of even real dimension, and hence we'll ignore them for the purpose of the present discussion, adopting simplified notation.

As before, we write $\bZ_x$ for the orientation space associated to a chord. More generally, given any map $u$ in the general class \eqref{eq:cauchy-riemann}, write $\bZ_u$ for the group isomorphic to $\bZ$, and such that a choice of isomorphism corresponds to an orientation of the space $\mathit{coker}(D_u) \oplus \mathit{ker}(D_u)$. If $u$ is defined on a $(d+1)$-punctured disc, with limits $(x_0,\dots,x_d)$, linear gluing yields a preferred isomorphism \cite[Equation (12.2)]{seidel04}
\begin{equation} \label{eq:orientation-isomorphism}
\bZ_{x_0} \iso \bZ_u \otimes \bZ_{x_d} \otimes \cdots \otimes \bZ_{x_1}.
\end{equation}
In similar notation, let $\bZ_{\scrR^{d+1}}$ be the group such that a choice of isomorphism $\bZ_{\scrR^{d+1}} \iso \bZ$ corresponds to an orientation of that space. For an isolated regular point, we have $\bZ_{\scrR^{d+1}} \otimes \bZ_u \iso \bZ$ canonically, and therefore \eqref{eq:orientation-isomorphism} turns into
\begin{equation}
\bZ_{x_0} \iso \bZ_{\scrR^{d+1}}^\vee \otimes \bZ_{x_d} \otimes \cdots \otimes\bZ_{x_1}.
\end{equation}
This, together with our chosen orientation of the spaces $\scrR^{d+1}$, determines the contribution of each isolated point in $\scrR^{d+1}(x_0,\dots,x_d)$ to the $x_0$-coefficient in $\mu^d(x_d,\dots,x_1)$; except that one inserts an additional artificial sign $(-1)^{\ddag_d}$ \cite[Equation (12.24)]{seidel04}, with
\begin{equation} \label{eq:ddag-1}
\ddag_d = |x_1| + 2|x_2| + \cdots + d|x_d|.
\end{equation}
For maps defined on surfaces \eqref{eq:z-surface}, the analogue of \eqref{eq:orientation-isomorphism} is
\begin{equation} \label{eq:orientation-isomorphism-2}
\bZ_{x_0} \iso \bZ_{\scrR^{k+1,l+1}}^\vee \otimes \bZ_{x_{k+l+1}} \otimes \cdots \otimes (\bZ[-n] \otimes \bZ_{x_{k+1}}^\vee) \otimes \cdots \otimes \bZ_{x_1}.
\end{equation}
Here, $\bZ[-n]$ is just a copy of $\bZ$ with a grading determined by the dimension of our Lagrangian submanifolds; we insert that into the formula since it affects the Koszul signs. The isomorphism \eqref{eq:orientation-isomorphism-2} enters into the definition of $\delta^{l,1,k}$ as before, where the additional sign is $(-1)^{\ddag_{k,l}}$,
\begin{equation} \label{eq:ddag-2}
\begin{aligned}
\ddag_{k,l} & = n|x_1| + (n+1)|x_2| + \cdots + (n+k-1)|x_k| + (n+k)|x_{k+1}| \\
& \qquad + (k+1) |x_{k+2}| + \cdots + (k+l) |x_{k+l+1}| + k.
\end{aligned} 
\end{equation}

We will not fully discuss the proof that this yields the sign conventions in \eqref{eq:quadratic-hochschild-differential}, but we will consider the contribution from (iv) in Lemma \ref{th:codim-1}, which is the most complicated one. For such points in $\bar\scrR^{k+1,l+1}(x_0,\dots,x_{k+l+1})$, the relevant instance of \eqref{eq:orientation-isomorphism} is
\begin{equation}
\bZ_{x_{k+1}} \iso \bZ_{\scrR^{j+1}}^\vee \otimes \bZ_{x_k} \otimes \cdots \otimes \bZ_{x_{i+1}} \otimes \bZ_x \otimes \bZ_{x_{i+j}} \otimes \cdots \otimes \bZ_{x_{k+2}},
\end{equation}
where $|x| = -|x_{i+1}| - \cdots - |x_k| + |x_{k+1}| - |x_{k+2}| - \cdots - |x_{i+j}| - 2 + j$. One can convert that into the form
\begin{equation} \label{eq:x-1}
\bZ_x^\vee \iso \bZ_{x_{i+j}} \otimes \cdots \otimes \bZ_{x_{k+2}} \otimes \bZ_{x_{k+1}}^\vee \otimes \bZ_{\scrR^{j+1}}^\vee \otimes \bZ_{x_k} \otimes \cdots \otimes \bZ_{x_{i+1}};
\end{equation}
this picks up a Koszul sign given by the parity of
\begin{equation} \label{eq:koszul-1}
j+|x_{i+1}|+ \cdots + |x_k|.
\end{equation}
Similarly, the relevant instance of \eqref{eq:orientation-isomorphism-2} is
\begin{equation} \label{eq:x-2}
\bZ_{x_0} \iso \bZ_{\scrR^{i+1,k+l-i-j+2}}^\vee \otimes \bZ_{x_{k+l+1}} \otimes \cdots \otimes \bZ_{x_{i+j+1}} \otimes
\bZ[-n] \otimes \bZ_x^\vee \otimes \bZ_{x_i} \otimes \cdots \otimes \bZ_{x_1}.
\end{equation}
One inserts the expression for $\bZ_x^\vee$ from \eqref{eq:x-1} into \eqref{eq:x-2}, and then transforms the outcome into the form
\begin{equation}
\bZ_{x_0} \iso (\bZ_{\scrR^{i+1,k+l-i-j+2}} \otimes \bZ_{\scrR^{j+1}})^\vee \otimes \bZ_{x_{r+s+1}} \otimes \cdots \bZ[-n] \otimes \bZ_{x_{k+1}}^\vee \otimes \cdots \otimes \bZ_{x_1},
\end{equation}
at the cost of another Koszul sign, determined by
\begin{equation} \label{eq:koszul-2}
n(|x_{k+2}| + \cdots + |x_{i+j}|) + 
j\big(|x_{k+1}| + \cdots + |x_{k+l+1}| + n+k+l\big).
\end{equation}
To that one adds the signs from \eqref{eq:ddag-1} and \eqref{eq:ddag-2},
\begin{align}
\label{eq:koszul-3}
&
\begin{aligned}
\ddag_j & = |x_{k+2}| + 2|x_{k+3}| + \cdots + (i+j-k-1) |x_{i+j}| \\
& \qquad \qquad + (i+j-k)|x| + (i+j-k+1) |x_{i+1}| + \cdots + j |x_k|, 
\end{aligned}
\\ \label{eq:koszul-4}
&
\begin{aligned}
& \ddag_{i,k+l-i-j+1} = n|x_1| + \cdots + (n+i-1)|x_i| + (n+i)|x| \\
& \qquad \qquad + (i+1)|x_{i+j+1}| + \cdots + (k+l-j+1) |x_{k+l+1}| + i. 
\end{aligned}
\end{align}
Adding up \eqref{eq:koszul-1}, \eqref{eq:koszul-2}--\eqref{eq:koszul-4} and $\ast_{(iv)}$ yields, after some computation mod $2$,
%
\begin{equation} \label{eq:koszul-5}
\dag_{(iv)} + (\ddag_{k,l} + |x_1| + \cdots + |x_{k+l+1}| + n).
\end{equation}
The first summand is the desired sign from \eqref{eq:quadratic-hochschild-differential}. Crucially, the rest of the formula, which includes the expression $\ddag_{k,l}$ from \eqref{eq:ddag-2}, is independent of exactly what splitting we are considering (it depends only on $(k,l)$ and the $x_i$). The same splitting-independent sign also arises when considering other codimension one boundary strata, and hence is ultimately irrelevant; in other words, what we get from geometry is the cocycle equation $d\delta = 0$ from \eqref{eq:quadratic-hochschild-differential} multiplied by that sign.

Finally, we should mention that, in our main argument later on, a version of the geometric diagonal class construction with an added parameter appears. Unsurprisingly, the resulting algebraic equations are governed by the differential \eqref{eq:quadratic-hochschild-differential} in degree $(n-1)$ (one less than the degree of the diagonal class). We won't comment on that sign computation, since the ingredients remain the same.

\section{Linear Cauchy-Riemann equations\label{sec:plane}}

This section brings together some of the ideas from Sections \ref{sec:gluing} and \ref{sec:fukaya}, in a geometric setup which is elementary but will, later on, underpin our main argument. Concretely, we consider certain linear Cauchy-Riemann equations on surfaces \eqref{eq:z-surface}, and solutions which have (lowest possible) exponential growth on one end, in the same sense as in Corollary \ref{th:r-moves}.

\subsection{Index theory preliminaries\label{subsec:weights}}
We will be looking at equations \eqref{eq:cauchy-riemann} in the toy model case where: $M = \bC$; all Lagrangians are $L = i\bR$; the standard complex structure is used; and the Hamiltonian functions involved are all multiples of $\half |w|^2$. In fact, for the Floer equation we want to make things even simpler, by choosing a $t$-independent
\begin{equation} \textstyle
H(w) = \frac{\alpha}{2} |w|^2, \;\;\;
\alpha\in (0,\pi). 
\end{equation}
The equation then becomes
\begin{equation} \label{eq:trivial-floer}
\left\{
\begin{aligned}
& u: \bR \times [0,1] \longrightarrow \bC, \\
& u(s,0), \; u(s,1) \in i\bR, \\
& \partial_s u + i\partial_t  u+ \alpha u = 0.
\end{aligned}
\right.
\end{equation}
This is linear, so the linearized operator $D$ is just its left hand side (replacing $u$ with $\Upsilon$, to stick with our general notation). Write $D = \partial_s + Q$, so that 
\begin{equation} \label{eq:model-q}
Q = i\partial_t +\alpha. 
\end{equation}
One can think of $Q$ as the selfadjoint operator associated to the unique chord $x = 0$. Its eigenvectors form the standard Fourier basis, labeled by $m \in \bZ$:
\begin{equation} \label{eq:lambda-m}
\left\{
\begin{aligned}
& \lambda_m = \pi m + \alpha, \\
& \Xi_m = i\exp(-\pi i m t).
\end{aligned}
\right.
\end{equation}

For more general surfaces \eqref{eq:s}, always assumed to come with a choice of ends, we take 
\begin{equation} \label{eq:elementary-k}
K = \beta \otimes \half|w|^2, \;\; \text{ where }
\left\{
\begin{aligned}
& \beta \in \Omega^1(S), \\
& \beta|\partial S = 0, \\
& \epsilon_\zeta^*\beta = \alpha\, \mathit{dt} \;\; \text{ for all $\zeta$.}
\end{aligned}
\right.
\end{equation}
The resulting Cauchy-Riemann equation
\begin{equation} \label{eq:plane-cr}
\left\{
\begin{aligned}
& u: S \longrightarrow \bC, \\
& u(\partial S) \subset i\bR, \\
& \bar\partial u - \beta^{0,1} iu = 0
\end{aligned}
\right.
\end{equation}
is again linear. As before, we denote the associated operator by $D$. In standard Sobolev spaces,
\begin{equation} \label{eq:c-index}
\mathrm{index}(D) = 1-|\Sigma_-|.
\end{equation}

\begin{lemma} \label{th:automatic-injectivity}
If $\mathrm{index}(D) \leq 0$, $D$ is injective.
\end{lemma}

\begin{proof}
We will only sketch the proof, since this is a familiar argument, compare e.g.\ \cite[Section 11d]{seidel04}. Suppose that we have a nonzero solution of $D\Upsilon = 0$. On a positive end, $\epsilon_\zeta^*\Upsilon = c_\zeta \exp(-\lambda_{m_\zeta} s) \Xi_{m_\zeta}(t) + \cdots$ for some $m_\zeta \geq 0$ and $c_\zeta \neq 0$. Hence, for $s \gg 0$, the loop 
\begin{equation}
t \longmapsto \Upsilon(\epsilon_\zeta(s,t))^2/|\Upsilon(\epsilon_\zeta(s,t))|^2
\end{equation}
has winding number $-m_\zeta \leq 0$. For a negative end we have $m_\zeta \leq -1$, hence the winding number is $-m_\zeta \geq 1$. Putting the two inequalities together yields
\begin{equation} \textstyle
\sum_{\zeta \in \Sigma_{\pm}} \mp m_\zeta \leq -|\Sigma_-| = \mathrm{index}(D)-1 < 0.
\end{equation}
That same number can be computed by adding up the multiplicities of zeros of $\Upsilon^2$ (with boundary points counting only half). Since this is nonnegative, we get a contradiction.
\end{proof}

We will also want to consider the same operators in weighted Sobolev spaces. Namely, take $\mu = (\mu_\zeta)_{\zeta \in \Sigma}$, where $\mu_\zeta \in \bZ$. As the domain, we take the space of those $\Upsilon$ such that $\exp(\pi \mu_\zeta s) \Upsilon(\epsilon_{\zeta}(s,t))$ lies in the standard $W^{1,2}$ space; the target space has the same modified version of $L^2$. (So, allowing exponential growth means taking $\mu_\zeta<0$ for $\zeta \in \Sigma_+$, respectively $\mu_\zeta > 0$ for $\zeta \in \Sigma_-$.) To mark the difference, we denote the weighted Sobolev completion by $D^\mu$. The counterpart of \eqref{eq:c-index} is
\begin{equation} \textstyle
\mathrm{index}(D^\mu) = 1-|\Sigma_-| + \sum_{\zeta \in \Sigma_{\pm}} \mp \mu_\zeta.
\end{equation}
By the same argument as before, one sees that:

\begin{lemma} \label{th:automatic-injectivity-2}
The statement of Lemma \ref{th:automatic-injectivity} also holds in weighted Sobolev spaces.
\end{lemma}

Given any $\zeta \in \Sigma_{\pm}$, one can look at the leading asymptotic coefficient:
\begin{equation} \label{eq:gamma-zeta-map}
\begin{aligned}
& \Lambda_{\zeta}: \mathit{ker}(D^\mu) \longrightarrow \bR, \\ &
\Lambda_{\zeta}(\Upsilon) = 
\frac{\lim_{s \rightarrow \pm\infty} \exp(\lambda_{m_\zeta} s) \Upsilon(\epsilon_\zeta(s,\cdot))\big)}{\Xi_{m_\zeta}}, \;\;
\text{where }
m_\zeta = 
\begin{cases}
\mu_\zeta & \zeta \in \Sigma_+, \\
\mu_\zeta - 1 & \zeta \in \Sigma_-.
\end{cases}
\end{aligned}
\end{equation}
The kernel of this map is $\mathit{ker}(D^{\tilde{\mu}})$, where $\tilde{\mu}$ is obtained from $\mu$ by increasing (for $\zeta \in \Sigma_+$) or decreasing (for $\zeta \in \Sigma_-$) $\mu_\zeta$ by one. By repeatedly applying this, until the index of $D^{\tilde{\mu}}$ becomes zero, and finally using Lemma \ref{th:automatic-injectivity-2}, one gets an upper bound $\mathrm{dim}(\mathit{ker}(D^\mu)) \leq \mathrm{index}(D^\mu)$. The conclusion is:

\begin{lemma} \label{th:surjective-operator}
If $\mathrm{index}(D^\mu) > 0$, $D^\mu$ is surjective, and each map \eqref{eq:gamma-zeta-map} is nonzero.
\end{lemma}

\subsection{Asymptotic behaviour and linear gluing}
Take surfaces $S^1$ and $S^2$, and pick $\zeta^1 \in \Sigma^1_+$, $\zeta^2 \in \Sigma^2_-$. Gluing these ends together, with gluing length $g \gg 0$, produces a surface $S$. Suppose that we have chosen weights $\mu^1$ and $\mu^2$ for our surfaces, such that
\begin{equation} \label{eq:weights-1}
\mu^1_{\zeta^1} = \mu^2_{\zeta^2} - 1.
\end{equation}
Forgetting those two, $S$ inherits weights $\mu$, such that
\begin{equation}
\mathrm{index}(D^\mu) = \mathrm{index}(D^{\mu^1}) + \mathrm{index}(D^{\mu^2}) - 1.
\end{equation}
Suppose now that $\mathrm{index}(D^{\mu^k}) > 0$ $(k = 1,2)$.

\begin{lemma} \label{th:linear-glue}
There is an operation (linear gluing) which produces a family of isomorphisms, smoothly depending on $g \gg 0$,
\begin{equation} \label{eq:linear-glue}
\mathit{ker}(D^{\mu^1}) \times_{\bR} \mathit{ker}(D^{\mu^2}) \stackrel{\iso}{\longrightarrow} \mathit{ker}(D^\mu),
\end{equation}
where the fibre product is taken with respect to \eqref{eq:gamma-zeta-map}, and the following properties hold:
\begin{align}
\label{eq:asymptotic-glue-1}
&
\mybox{
Suppose that $\zeta \neq \zeta^1$ is some other end of $S^1$, which then becomes an end of $S$ as well. The diagram
$$
\xymatrix{
\mathit{ker}(D^{\mu^1}) \times_{\bR} \mathit{ker}(D^{\mu^2}) 
\ar[d]_-{\text{projection}}
\ar[rr]^-{\eqref{eq:linear-glue}}
&&\mathit{ker}(D^\mu) \ar[d]^-{\Lambda_{\zeta}}
\\
\mathit{ker}(D^{\mu^1}) \ar[rr]^-{\Lambda_{\zeta}} &&
\bR
}
$$
is asymptotically commutative, which means that its failure to commute goes to zero as $g \rightarrow \infty$. 
}
\\ &
\label{eq:asymptotic-glue-2}
\mybox{
Similarly, suppose that $\zeta \neq \zeta^2$ is an end of $S^2$, and set $m = \mu^2_{\zeta^2}-1$; then, asymptotic commutativity holds for
$$
\xymatrix{
\mathit{ker}(D^{\mu^1}) \times_{\bR} \mathit{ker}(D^{\mu^2}) 
\ar[d]_-{\text{projection}}
\ar[rr]^-{\eqref{eq:linear-glue}} &&
\mathit{ker}(D^\mu) \ar[d]^-{\exp(\lambda_m g)\Lambda_{\zeta}}
\\
\mathit{ker}(D^{\mu^2}) \ar[rr]^-{\Lambda_{\zeta}} &&
\bR.
}
$$
}
\end{align}
\end{lemma}


\begin{proof}[Sketch of proof]
This is a thinly disguised form of a very familiar result. Instead of giving a full proof, we just describe the elementary considerations which reduce it to the standard form.

Given a surface $S$ and a choice of weights $\mu_\zeta$, choose a $\rho \in \smooth(S,\bC)$ such that
\begin{equation}
\rho(\epsilon_\zeta(s,t)) = \lambda_{m_\zeta} s + i\pi m_\zeta t = 
\alpha s + \begin{cases} \mu_\zeta \pi (s+it) & \zeta \in \Sigma_+, \\
& \\ (\mu_\zeta-1) \pi (s+it) & \zeta \in \Sigma_-. \end{cases}
\end{equation}
Then
\begin{equation}
\tilde{D}\tilde{\Upsilon} = \exp(\rho) D(\exp(-\rho)\tilde{\Upsilon})
\end{equation}
is a Cauchy-Riemann operator in a class slightly more general than the one considered before. In particular, the real boundary conditions vary along $\partial S$. However, on each end we have $\tilde{D} = \bar\partial = \partial_s + i\partial_t$ with imaginary boundary conditions. Finally, our choice of weights for $\Upsilon$ amounts to allowing $\tilde{\Upsilon} = \exp(\rho)\Upsilon$ to grow at a very small rate on each end. In fact, this means that the corresponding solutions of $\tilde{D}{\tilde\Upsilon} = 0$ are bounded, and extend holomorphically to the compactification obtained by adding a point $s = \pm\infty$ to each end. The value at the point at infinity is related to the previous asymptotics by
\begin{equation} \label{eq:relate}
\tilde{\Upsilon}(\epsilon_\zeta(\pm \infty)) = i \Lambda_\zeta(\Upsilon).
\end{equation}

The application to gluing goes as follows: take surfaces $S^1$ and $S^2$ with their corresponding functions $\rho^1,\rho^2$; by \eqref{eq:weights-1}, the functions have the same shape on the ends to be glued together. Given $(\Upsilon^1,\Upsilon^2)$ which lie in the fibre product from \eqref{eq:linear-glue}, the corresponding $(\tilde{\Upsilon}^1, \tilde{\Upsilon}^2)$ have the same value at the points $\epsilon_{\zeta^1}(+\infty)$, $\epsilon_{\zeta^2}(-\infty)$. Hence, one can use a standard construction to glue them together to a solution of a corresponding equation $\tilde{D}\tilde{\Upsilon} = 0$ on the surface $S$, for $g \gg 0$. That surface carries a function $\rho$ which equals $\rho^1$ on the part inherited from $S^1$, and $\exp(\lambda_m g) \rho^2$ on the part inherited from $S^2$, for $m$ as in \eqref{eq:asymptotic-glue-2}. This gives us the desired $\Upsilon = \exp(-\rho)\tilde{\Upsilon}$, which is the image of $(\Upsilon^1, \Upsilon^2)$ under \eqref{eq:linear-glue}. By construction, the values of $\tilde{\Upsilon}$ at the points $\epsilon_{\zeta}(\pm\infty)$, for $\zeta \neq \zeta^1,\zeta^2$, converge to those of $\tilde{\Upsilon}^1$, $\tilde{\Upsilon}^2$ as $g \rightarrow \infty$. In view of \eqref{eq:relate}, this immediately implies \eqref{eq:asymptotic-glue-1}. If $\zeta$ is an end of $S^2$, the analogue of \eqref{eq:relate} must be modified because of the different choice of $g$ over that end, and now says that
\begin{equation}
\tilde{\Upsilon}(\epsilon_\zeta(\pm \infty)) = i\exp(\lambda_m g) \Lambda_\zeta(\Upsilon),
\end{equation}
which leads to \eqref{eq:asymptotic-glue-2}. As for the gluing we've actually used in the argument (for maps on compact Riemann surfaces, with common boundary values), it's the linear version of a procedure which goes back to \cite{mrowka88} in the analogous gauge-theoretic context; in pseudo-holomorphic curve theory, it appears e.g.\ in \cite[Section 7.1.3]{fooo} or \cite{abouzaid12}.
\end{proof}

As a first application, consider more specifically surfaces $S$ which are $(d+1)$-punctured discs \eqref{eq:disc-surface}, numbering the $\zeta$s by $\{0,\dots,d\}$ as usual. We set $\mu = (1,0,\dots,0)$. By Lemma \ref{th:surjective-operator}, we have isomorphisms, for any $k \in \{1,\dots,d\}$,
\begin{equation} \label{eq:compare-ends}
\xymatrix{
\bR  && \mathit{ker}(D^\mu) \ar[ll]_-{\Lambda_{0}}^-{\iso} \ar[rr]^-{\Lambda_{k}}_-{\iso}
&& \bR.}
\end{equation}

\begin{lemma} \label{th:doesnt-matter-in-the-end}
The two maps \eqref{eq:compare-ends} always differ by multiplication with a positive number.
\end{lemma}

\begin{proof}
This property is clearly invariant under deformations, so it depends only on the topology of the situation (the numbers $d$ and $k$). For $d = 1$, it is true for elementary reasons, since one can equip $S = \bR \times [0,1]$ with $\beta = \alpha \mathit{dt}$; then $\mathit{ker}(D^\mu)$ is generated by $\Upsilon = \exp(-\lambda_0 s)\Xi_0$.

Now consider an $S^1$ with $d > 1$. Pick some $j \neq k$ in $\{1,\dots,d\}$. At the $j$-th end, one can glue in a once-punctured disc $S^2$ with weight $1$, so as to obtain an $S$ with one less puncture than $S^1$. By an application of \eqref{eq:asymptotic-glue-1} to this gluing process, one sees that the desired statement is true for $S^1$ iff it is true for $S$. This allows one to reduce the general case to the previously discussed one.
\end{proof}

\subsection{The parametrized version\label{subsec:parametrized-c}}
Let's extend \eqref{eq:plane-cr} by adding a parameter $r$, which changes the inhomogeneous term by a constant Hamiltonian vector field pointing in horizontal direction. Concretely, this means that we replace \eqref{eq:elementary-k} with
\begin{equation} \label{eq:gamma-form}
K_r = \beta \otimes \half |w|^2 - r(\gamma \otimes \mathrm{im}(w)),\;\; \text{ for some }
\left\{
\begin{aligned}
& \gamma \in \Omega^1_c(S), \\
& \gamma|\partial S = 0, \\
& \epsilon_\zeta^*\gamma = 0 \;\; \text{ for all $\zeta$};
\end{aligned}
\right.
\end{equation}
and \eqref{eq:plane-cr} with
\begin{equation} \label{eq:plane-cr-2}
\bar\partial u -  iu\beta^{0,1} - r \gamma^{0,1} = 0.
\end{equation}
Again, one takes the left hand side of this and replaces $(r,u)$ by $(R,\Upsilon)$, to get the parametrized version $D^{\mathit{para}}$ of the linearized operator. We will also adopt the notation $D^{\mathit{para},\mu}$ for the completion in weighted Sobolev spaces.

Suppose that we are gluing together two surfaces along ends as in \eqref{eq:weights-1}, but where $S^2$ carries some $\gamma^2$ which defines our parametrized equation (whereas $S^1$ does not). The glued surface $S$ carries an induced $\gamma$ term (which is zero on the part inherited from $S^1$). Assuming that $D^{\mu^1}$ and $D^{\mu^2,\mathit{para}}$ are onto, the analogue of \eqref{eq:linear-glue} is a map
\begin{equation} \label{eq:linear-glue-2}
\mathit{ker}(D^{\mu^1}) \times_{\bR} \mathit{ker}(D^{\mathit{para},\mu^2}) \stackrel{\iso}{\longrightarrow} \mathit{ker}(D^{\mathit{para},\mu}).
\end{equation}
This fits into analogues of \eqref{eq:asymptotic-glue-1} and \eqref{eq:asymptotic-glue-2}. Additionally, if we consider the forgetful map to the $R$-component of the domain, the following diagram is also asymptotically commutative:
\begin{equation} \label{eq:asymptotic-glue-3}
\xymatrix{
\mathit{ker}(D^{\mu^1}) \times_{\bR} \mathit{ker}(D^{\mathit{para},\mu^2}) 
\ar[dr]_-{\text{ forgetful}\;\;}
\ar[rr]^-{\eqref{eq:linear-glue-2}}
&&\mathit{ker}(D^{\mathit{para},\mu}) 
\ar[dl]^-{\;\;\;\;\;\;\text{ forgetful times }\exp(\lambda_{m_{\zeta^2}} g)}
\\
& \bR.
}
\end{equation}
Of course, the corresponding statements hold if we have a gluing situation where $S^1$ carries a parametrized equation and $S^2$ doesn't.

Now, let's consider more specifically surfaces \eqref{eq:z-surface}, with weights $\mu= (0,\dots,1,\dots,0)$, where the unique $1$ entry is associated to $\zeta_{k+1} \in \Sigma_-$. As an instance of Lemma \ref{th:automatic-injectivity-2}, $D^\mu$ is invertible. Therefore $D^{\mathit{para},\mu}$ is surjective with one-dimensional kernel, which maps isomorphically to the $R$-component in its domain. Let's combine this with the parametrized analogue of \eqref{eq:gamma-zeta-map} for the end carrying the nonzero weight, to get
\begin{equation} \label{eq:two-maps-to-r}
\xymatrix{
\bR && \ar[ll]_-{\text{forgetful}}^-{\iso} 
\mathit{ker}(D^{\mathit{para},\mu}) 
\ar[rr]^-{\Lambda_{k+1}}
&& \bR.
}
\end{equation}
Going from left to right in \eqref{eq:two-maps-to-r} is multiplication with some number, which we denote by
\begin{equation} \label{eq:lim-gamma}
\Gamma = \Gamma(S,\beta,\gamma) \in \bR.
\end{equation}
More concretely, this number is obtained by solving $D^\mu(\Upsilon) = \gamma^{0,1}$ and then looking at the coefficient of exponential growth near $\zeta_{k+1}$. As a consequence, it depends linearly on $\gamma$. 

\begin{lemma} \label{th:its-an-isomorphism}
$D^{\mathit{para}}$ (the parametrized operator on a surface \eqref{eq:z-surface}, with trivial weights) is an isomorphism if and only if \eqref{eq:lim-gamma} is nonzero.
\end{lemma}

This is clear from the definitions.

\begin{lemma} \label{th:it-can-be-nonzero}
For any surface \eqref{eq:z-surface} and any choice of $\beta$, one can find a $\gamma$ such that \eqref{eq:lim-gamma} is nonzero.
\end{lemma}

\begin{proof}
$D^\mu$ is surjective but $D$ isn't. Hence, we can find a smooth $\gamma$, supported in a compact subset of the interior of $S$ that's disjoint from the ends, such that $\gamma^{0,1}$ lies in the image of $D^{\mu}$ but not in that of $D$. 
\end{proof}


Next, we need  to look at the behaviour of \eqref{eq:lim-gamma} under gluing surfaces together.

\begin{lemma} \label{th:sign-remains}
Take two surfaces, one of type \eqref{eq:z-surface} (carrying $\beta$ and $\gamma$ one-forms), and an auxiliary one other of type \eqref{eq:disc-surface} (carrying a $\beta$). Gluing them together at one end yields another surface in the class \eqref{eq:z-surface}. For both the original and glued surface, consider \eqref{eq:lim-gamma}. If this is positive for the original surface, then the same holds for the glued surface, assuming the gluing length is sufficiently large (and correspondingly if it is negative).
\end{lemma}
\begin{figure}
\begin{centering}
\includegraphics[scale=0.8]{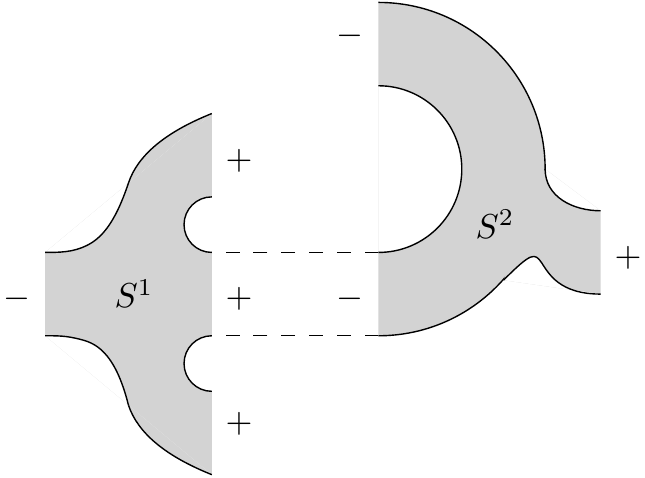}
\caption{\label{fig:special-glue}The gluing situation from the proof of Lemma \ref{th:sign-remains}. We have marked the ends as positive/negative. In the specific example drawn here, $d^1 = 3$, $(k^2,l^2) = (0,1)$, $k = 2$, and $(k,l) = (1,2)$.}
\end{centering}
\end{figure}

\begin{proof}
This is a consequence of the properties of \eqref{eq:linear-glue-2}. While most cases are straightforward, there is one which deserves discussion. Namely, suppose that $S^1$ is a $(d^1+1)$-punctured disc, with some choice of $\beta^1$; and $S^2$ of type \eqref{eq:z-surface} for some $(k^2,l^2)$, with its $(\beta^2,\gamma^2)$. We glue the $j$-th end of $S^1$ ($1 \leq k \leq d^1$, hence a positive end) to the $(k^2+1)$-st (negative) end of $S^2$, with the result being a surface $S$ of type $(k,l)$ (see Figure \ref{fig:special-glue}).

Consider a weighted operator $D^{\mu^1}$ on $S^1$, where the weights are $\mu^1 = (1,0,\dots,0)$. For $S^2$ we use $D^{\mathit{para},\mu^2}$ with the weights $\mu^2 = (0,\dots,1,\dots,0)$ that enter into \eqref{eq:two-maps-to-r}. Take the unique solution of $D^{\mathit{para},\mu^2}(R^2,\Upsilon^2) = 0$ with $R^2 = 1$. There is a unique solution of $D^{\mu^1}(\Upsilon^1) = 0$ with matching asymptotics on the ends undergoing gluing, meaning that $(\Upsilon^1,(R^2,\Upsilon^2))$ belongs to the domain of \eqref{eq:linear-glue-2}. Applying \eqref{eq:linear-glue-2} yields a solution of $D^{\mathit{para},\mu}(R,\Upsilon) = 0$, with the following properties.
\begin{align} \label{eq:s1}
&
\mybox{As an application of the asymptotic commutativity of \eqref{eq:asymptotic-glue-3}, we have that $\exp(\alpha g) R$ is close to $R^2 = 1$; hence $R$ is positive.}
\\ \label{eq:s2} &
\mybox{Using \eqref{eq:asymptotic-glue-1} in the same way, we see that the asymptotics of $\Upsilon$ at the $(k+1)$-st end are approximately inherited from those of $\Upsilon^1$ at the $0$-th end.}
\\ \label{eq:s3} &
\mybox{Lemma \ref{th:doesnt-matter-in-the-end} says that the asymptotics of $\Upsilon^1$ at the $0$-th end agree with those at the $j$-th end, up to multiplication with a positive constant. By construction, those asymptotics match those of $\Upsilon^2$ at the $(k^2+1)$-st end.}
\end{align}
Suppose that \eqref{eq:lim-gamma} is positive for $S^1$. This means that the leading coefficient of $\Upsilon^2$ at the $(k^2+1)$-st end is a positive multiple of $\Xi_0$. By applying \eqref{eq:s3} and then \eqref{eq:s2}, one sees that the same holds for $\Upsilon$ at the $(k+1)$-st end. Together with \eqref{eq:s1}, this implies that \eqref{eq:lim-gamma} is positive for $S$, as claimed. 
\end{proof}

\subsection{Modular dependence}
Let's start with the familiar spaces $\scrR^{d+1}$. On those, one can make consistent choices of (ends and) $\beta$ terms, in the sense of \cite[Section 9i]{seidel04}. Here, we can interpret consistency in a strict sense, which means the following. Take a point of $\bar\scrR^{d+1}$, lying in a stratum labeled by some tree $T$, and described by a collection of $|v|$-pointed discs $S^v$. Nearby points of $\scrR^{d+1}$ correspond to surfaces $S$ obtained by gluing the $S^v$ together, with some (large) gluing lengths. Then, the $\beta$ term on those surfaces is correspondingly glued together from those on the $S^v$.
%

Now consider spaces $\scrR^{k+1,l+1}$ as in Section \ref{subsec:diagonal-class}, but with no interior marked points (for simplicity). On these spaces, we can make consistent choices of $\beta$ and $\gamma$, in the same sense. For $\beta$, that condition involves the previous choices made for $\scrR^{d+1}$. For $\gamma$, what we mean is that if $S$ is obtained by gluing together some $S^v$ (with large gluing lengths), then its one-form $\gamma$ is concentrated in the part of $S$ coming from $S^{v_*}$, and agrees with what we had previously chosen for $S^{v_*}$ by itself. A slight generalization of Lemma \ref{th:sign-remains} yields the following:

\begin{lemma} \label{th:its-continuous}
Suppose that the $\beta$, $\gamma$ have been chosen consistently over the moduli space. For each point of $\bar\scrR^{k+1,l+1}$, consider the surface $S^{v_*}$ and the associated number \eqref{eq:lim-gamma}. The subset where this number is positive (or negative) is open in $\bar\scrR^{k+1,l+1}$.
\end{lemma}

From there on, using Lemma \ref{th:it-can-be-nonzero} and partitions of unity, one easily obtains the following:

\begin{prop} \label{th:sign-of-gamma}
Given any consistent choice of $\beta$, one can choose the $\gamma$'s consistently so that the function from Lemma \ref{th:its-continuous} is everywhere positive (or everywhere negative).
\end{prop}

\section{Connections on surfaces\label{sec:connections}}
Following \cite{seidel18}, one can base the definition of the Fukaya category of a Lefschetz fibration on Riemann surfaces equipped with flat connections, whose structure group is the group of affine automorphisms of the real line. We will need a generalization, in which flatness is replaced by a nonnegativity condition for the curvature. This is similar to \cite{seidel17} but more elementary, because the group involved is so simple.

\subsection{Definitions\label{subsec:affine}}
Let $G$ be the group of orientation-preserving affine transformations of the real line, and $\frakg$ its Lie algebra. $G$ is an extension
\begin{equation} \label{eq:affine-group}
0 \rightarrow \bR \longrightarrow G \longrightarrow \bR^{\gr} \rightarrow 1.
\end{equation}
The Lie algebra $\frakg$ consists of vector fields $(aw+b)\partial_w$. A connection on the trivial $G$-bundle over a manifold $S$ is an operation $d-A$, for $A \in \Omega^1(S,\frakg)$. Gauge transformations $\Phi \in \smooth(S,G)$ act on connections by
\begin{equation}
\Phi_*A = \Phi A \Phi^{-1} + (d\Phi) \Phi^{-1}.
\end{equation}
The curvature of a connection is
\begin{equation}
F_A = -dA + \half [A,A] \in \Omega^2(S,\frakg).
\end{equation}
The following statement is an instance of the general ``reduction of structure group'' for connections \cite[Chapter II, Theorem 7.1]{kobayashi-nomizu}:

\begin{lemma} \label{th:partial-trivialization}
Suppose that $S$ is simply-connected, and that $A$ is a connection whose curvature lies in $\Omega^2(S) \otimes \partial_w$. Choose a base point $\bullet \in S$. Then there is a gauge transformation $\Phi$ which is based, meaning that it is the identity at the base point, such that $\Phi_*A \in \Omega^1(S) \otimes \partial_w$ is a connection for the translation subgroup. Moreover, $\Phi$ is unique up to composition with based gauge transformations in the $\bR$ subgroup. 
\end{lemma}

Define the nonnegative part of $\frakg$ to be the subset of vector fields that are nonnegative on the entire real line, or equivalently
\begin{equation}
\frakg_{\ge} = \bR^{\ge} \partial_w.
\end{equation}
Let $S$ be an oriented surface. We say that a connection is nonnegatively curved if, after inserting any oriented basis of $TS$ into $F_A$, we get an element of $\frakg_{\ge}$. The nonnegative curvature condition is preserved under gauge transformations. If $S$ is simply-connected, one can apply Lemma \ref{th:partial-trivialization} to transform a nonnegatively curved connection into one of the form $\Phi_*A = \gamma \otimes \partial_w$, where $\gamma \in \Omega^1(S)$ satisfies
\begin{equation}
d\gamma \leq 0.
\end{equation}

\subsection{Spaces of connections\label{subsec:connections}}
Let $\scrA([0,1])$ be the space of those $A \in \Omega^1([0,1],\frakg)$ such that:
\begin{equation} \label{eq:move-past}
\mybox{
if $g \in G$ is the parallel transport for $A$ over $[0,1]$, acting on the real line, then $g(0) > 0$.
}
\end{equation}
There is a natural gauge group $\scrG([0,1])$ acting on $\scrA([0,1])$, namely those $\Phi \in \smooth([0,1],G)$ such that $\Phi_t(0) = 0$ at $t = 0,1$. In fact, that group acts simply transitively. On the other hand, the group is clearly weakly contractible, and therefore \cite[Section 2.1]{seidel18}:

\begin{lemma} \label{th:a-interval}
$\scrA([0,1])$ is weakly contractible.
\end{lemma}

The analogous two-dimensional concept is a little more cumbersome, because of technical constraints dictated by our intended application to Floer theory. Consider a surface \eqref{eq:s}, but this time in a topological sense (as an oriented smooth surface, with no complex structure). Suppose that our surface comes with ends \eqref{eq:ends} (again, in a topological sense, meaning that they are oriented smooth coordinates). Let $\scrA_{\ge}(S)$ be the space of those $A \in \Omega^1(S,\frakg)$ such that:
\begin{align}
& \label{eq:p-space-ends} \mybox{
$\epsilon_{\zeta}^*A$ is the pullback by projection of some $A_\zeta \in \scrA([0,1])$.}
\\ \label{eq:covariantly-constant-lambda} 
& \mybox{Parallel transport for $A|\partial S$ preserves $0 \in \bR$.} 
\\ \label{eq:p-space-curvature} & 
\mybox{$A$ is nonnegatively curved everywhere.}
\end{align}

\begin{lemma} \label{th:a-surface-preliminary}
The space $\scrA_{\ge}(S)$ is empty if $\Sigma_- = \emptyset$, and weakly contractible otherwise.
\end{lemma}

\begin{proof}
Fix a base point $\bullet \in S \setminus \partial S$, which is disjoint from the ends. Let $\scrG_\bullet(S) \subset \smooth(S,G)$ be the group of those based gauge transformations which, on each end, are independent of $s$, and whose restriction to $\partial S$ preserves the point $0 \in \bR$. This group acts on $\scrA_{\ge}(S)$. By a version of Lemma \ref{th:partial-trivialization}, one sees that there is a $\Phi \in \scrG_\bullet(S)$ such that $\Phi_*A = \gamma \otimes \tau$, where $\gamma$ satisfies the following analogues of \eqref{eq:p-space-ends}--\eqref{eq:p-space-curvature}:
\begin{align} \label{eq:abelian-1}
& \mybox{
$\epsilon_{\zeta}^*\gamma$ is the pullback of some $\gamma_\zeta \in \Omega^1([0,1])$, such that $\textstyle \int_0^1 \gamma_\zeta > 0$.
}
\\ & \mybox{$\gamma|\partial S = 0$.}
\label{eq:abelian-2}
\\ \label{eq:abelian-3} & \mybox{
$d\gamma$ is nonpositive everywhere.}
\end{align}

Applying Stokes' theorem shows that
\begin{equation} \label{eq:stokes-z}
0 \geq \int_S d\gamma  = \sum_{\zeta \in \Sigma_\pm} \pm \int_0^1 \gamma_\zeta.
\end{equation}
If all ends are positive ones, the right hand side of \eqref{eq:stokes-z} is positive, which is a contradiction. Suppose from now on that there is at least one negative end. Then, the space of $\gamma$ satisfying \eqref{eq:abelian-1}--\eqref{eq:abelian-3} is certainly nonempty (one first chooses the $\gamma_\zeta$ so that the right hand side of \eqref{eq:stokes-z} is nonpositive, and then arranges for an extension over the interior of $S$ so that \eqref{eq:abelian-3} holds). To investigate this further, write $\scrA_\ge^\dag(S)$ for the space of $\gamma$ satisfying \eqref{eq:abelian-1}--\eqref{eq:abelian-3}. Let $\scrG_\bullet^\dag(S)$ be the (additive) group of those $\phi \in \smooth(S,\bR)$ which vanish on $\partial S$ and at the base point, and which are $s$-independent on each end. This acts on $\scrA_\ge^\dag(S)$, simply by adding $d\phi$. Again by a version of Lemma \ref{th:partial-trivialization}, the relation between all these spaces and groups is that
\begin{equation}
\scrA_\ge^{\dag}(S) \times_{\scrG^\dag_\bullet(S)} \scrG_\bullet(S) \iso \scrA_{\ge}(S).
\end{equation}
Now, $\scrA_\ge^{\dag}(S)$ is weakly contractible, since \eqref{eq:abelian-1}--\eqref{eq:abelian-3} are convex conditions. The groups involved are also weakly contractible, so it follows that the same holds for $\scrA_{\ge}(S)$.
\end{proof}

Now, suppose that we work with fixed $A_\zeta$. This gives a space denoted by $\scrA_{\ge}(S,\Sigma)$, which fits into a weak fibration
\begin{equation}
\scrA_\ge(S,\Sigma) \longrightarrow \scrA_\ge(S) \longrightarrow \scrA([0,1])^{|\Sigma|}.
\end{equation}
Lemmas \ref{th:a-interval} and \ref{th:a-surface-preliminary} then imply:

\begin{lemma} \label{th:a-surface}
If $\Sigma_- \neq \emptyset$, $\scrA_\ge(S,\Sigma)$ is weakly contractible.
\end{lemma}

\section{Geometric and analytic setup\label{sec:setup}}
This section describes the geometry of the target space, and discusses some basic properties of the resulting Cauchy-Riemann equations. It combines ideas from \cite{seidel17, seidel18, seidel18b}, with the addition of a symplectic ample divisor. We will not be going through the entire foundational Floer theory package, as some aspects (notably transversality) entirely follow the standard process.

\subsection{Hyperbolic geometry\label{subsec:hyp}}
Take the open upper half-plane and its compactification,
\begin{equation}
\begin{aligned}
& W = \{\im(w) > 0\} \subset \bC, \\
& \bar{W} = \{\im(w) \geq 0\} \cup \{\infty\}, \\
& \partial_\infty W = \partial \bar{W} = \bar{W} \setminus W.
\end{aligned}
\end{equation}
We equip $W$ with the hyperbolic symplectic form,
\begin{equation}
\omega_W = \frac{d\re(w) \wedge d\im(w)}{\im(w)^2}.
\end{equation}
The group $G$ of affine transformations (see Section \ref{subsec:affine}) acts on $\bar{W}$, by holomorphically extending its action on the real line. When restricted to $W$, the action is Hamiltonian. Concretely, the vector fields and Hamiltonian functions are 
\begin{equation} \label{eq:g-action}
X_\gamma = (aw+b)\partial_w, \quad H_\gamma = \frac{a \re(w) + b}{\im(w)}.
\end{equation}

\subsection{Symplectic geometry\label{subsec:context}}
Let $E^{2n}$ be a symplectic manifold equipped with a symplectic ample divisor $\Omega_E$, and a one-form $\theta_E$ on its complement, as in Setup \ref{th:ample}. We assume that $E$ is Calabi-Yau, as in Setup \ref{th:branes}. Our manifold should come with a proper map
\begin{equation} \label{eq:pi-w}
p: E \longrightarrow W.
\end{equation}
If $x$ is a regular point of $p$, define $\mathit{TE}_x^h \subset \mathit{TE}_x$ to be the $\omega_E$-orthogonal complement of $\mathit{TE}_x^v = \mathit{ker}(Dp_x)$. Whenever $\mathit{TE}_x^v$ is a symplectic subspace, $\mathit{TE}_x^h$ is also symplectic and a complementary subspace, hence projects isomorphically to $TW_{p(x)} = \bC$. Let $E^{\mathit{triv}} \subset E$ be the set of regular points $x \in E$ such that $\mathit{TE}_x^v$ is a symplectic subspace, and $\omega_E|\mathit{TE}_x^h$ is the pullback of $\omega_W$. The significance of this condition is roughly as follows: on the subset of regular points for which $\mathit{TE}_x^v$ is symplectic, the restriction $\omega_E|\mathit{TE}^h_x$ is (a lift from Hamiltonian vector fields to functions of) the curvature of the symplectic connection $TE^h$. Hence, on the interior of $E^{\mathit{triv}}$, the connection is flat.

\begin{setup} \label{th:e-setup}
Of the map \eqref{eq:pi-w}, we require that:
\begin{align} &
\label{eq:e-triv}
\mybox{$E \setminus E^{\mathit{triv}}$ is relatively compact; and $\mathit{TE}^h$ is tangent to $\Omega_E$, on the complement of a compact subset in that submanifold.}
\\ &
\label{eq:e-triv-2}
\mybox{$p^{-1}(\{\im(w) \leq 4\}) \subset E^{\mathit{triv}}$; and on $\Omega_E \cap p^{-1}(\{\im(w) \leq 4\})$, $\mathit{TE}^h$ is tangent to $\Omega_E$.}
\end{align}
\end{setup}

Let $M = E_i$ be the fibre at $i \in W$. This is a symplectic manifold, and $\Omega_M = \Omega_E \cap M$ is a symplectic submanifold. Using parallel transport, and the flatness condition \eqref{eq:e-triv-2} for the connection, we get a symplectic isomorphism
\begin{equation} \label{eq:2-trivialization}
E \supset p^{-1}(\{\im(w) \leq 4\}) \iso \{\im(w) \leq 4\} \times M \subset W \times M
\end{equation}
which is fibered over $W$, and is the identity on the fibre over $i$. Moreover, this isomorphism takes (the relevant parts of) $\Omega_E$ to $W \times \Omega_M$. From the same perspective, let's consider \eqref{eq:e-triv}, which gives local trivializations for $p$ outside a compact subset. Using those, one can form a canonical compactification
\begin{equation} \label{eq:e-bar}
p: \bar{E} \longrightarrow \bar{W},
\end{equation}
with $\partial_\infty E = \bar{E} \setminus E = p^{-1}(\partial_\infty W)$. The closure $\bar{\Omega}_E \subset \bar{E}$ is again a codimension two submanifold. 

We consider Lagrangian submanifolds $L \subset E$ which are disjoint from $\Omega_E$ and exact for $\theta_E$, as in Setup \ref{th:ample}. Moreover, they should be graded and {\em Spin}, see Setup \ref{th:branes}. However, they are no longer required to be compact. Instead,

\begin{setup} \label{th:l}
All our Lagrangian submanifolds should satisfy:
\begin{align}
&
\mybox{$p(L) \cap \{\mathrm{im}(w) \geq 1\}$ is compact;} \\
& \mybox{$p(L) \cap \{\mathrm{im}(w) \leq 1+\epsilon\} \subset (0,1+\epsilon]i$, for some $\epsilon = \epsilon_L>0$.}
\end{align}
\end{setup}

As a consequence, under \eqref{eq:2-trivialization}, $L \cap p^{-1}(\{\mathrm{im}(w) \leq 1+\epsilon\})$ corresponds to $(0,1+\epsilon]i \times (L \cap M)$. Moreover, the closure $\bar{L} \subset \bar{E}$ is again a submanifold (with boundary $\partial \bar{L} = \bar{L} \cap\bar{E}_0$).

We use compatible almost complex structures $J$ on $E$ which make $\Omega_E$ into an almost complex submanifold (as in Setup \ref{th:ample}), but which are adapted to $p$, in the following sense.

\begin{setup} \label{th:j}
(i) All our almost complex structures should satisfy:
\begin{align}
& \mybox{Outside a compact subset, $Dp$ is $J$-holomorphic.} \\
& \mybox{$J$ extends smoothly to an almost complex structure on $\bar{E}$. 
}
\end{align}
(ii) For part of our argument, we will want to impose additional conditions. An almost complex structure of the kind described above is called ``suitable for restriction'' if it also satisfies:
\begin{align} \label{eq:j-splits}
& \mybox{In a neighbourhood of $M = E_i$, with respect to \eqref{eq:2-trivialization}, $J$ is the product of the standard complex structure on the base and some compatible almost complex structure on $M$.}
\end{align}
\end{setup}

Concerning Hamiltonian functions $H$ on $E$, we want the associated vector field $X$ to be tangent to $\Omega_E$, as in Setup \ref{th:ample}. There are requirements concerning the behaviour of $H$ at infinity, which involve the action of $G$ on $W$.

\begin{setup} \label{th:h}
(i) Given some $\gamma \in \frakg$, we want our Hamiltonian functions $H$ to satisfy:
\begin{equation} \label{eq:h-at-infinity}
\mybox{Outside a compact subset, $H$ is the pullback of the Hamiltonian $H_\gamma$ from \eqref{eq:g-action}.}
\end{equation}
(ii) As before, there is an additional requirement which makes $H$ ``suitable for restriction'', namely:
\begin{equation} \label{eq:h-splits}
\mybox{In a neighbourhood of $M = E_i$, under \eqref{eq:2-trivialization}, $H$ corresponds to the sum of a constant multiple of $p^*(\half|w-i|^2)$ and a function on $M$.}
\end{equation}
\end{setup}

The condition \eqref{eq:h-at-infinity} implies that the associated vector field is, outside a compact subset, the unique lift of $X_\gamma$ to $\mathit{TE}^h$. In particular, $X$ always extends smoothly to $\bar{E}$.

\subsection{Cauchy-Riemann equations\label{subsec:our-cr}}
We begin by setting up the data underlying Floer's equation. Take Lagrangian submanifolds $L_0,L_1$ as in Setup \ref{th:l}. 

\begin{setup} \label{th:setup-floer}
Choose some $A = a_t \mathit{dt} \in \scrA([0,1])$ (see Section \ref{subsec:connections}), and additionally, some $\alpha \in (0,\pi)$ (which, as we'll see below, needs to be appropriately small). For $t \in [0,1]$, we choose $H_t$ to lie in the class of functions associated to $a_t \in \frakg$, as in Setup \ref{th:h} (including its second part), where the constant multiple in \eqref{eq:h-splits} is $\frac{\alpha}{2} |w-i|^2$ for all $t$. We require nondegeneracy of chords (which, together with the other conditions, implies that there are only finitely many chords), as well as the following:
\begin{equation}
\label{eq:choose-floer-1}
\mybox{Each chord either lies in $p^{-1}(\{\mathrm{im}(w) > 1+\epsilon\})$, for some $\epsilon = \epsilon_{L_0,L_1,H} > 0$, or in $M = E_i$.}
\end{equation}
We also take $J_t$ as in Setup \ref{th:j} (again including its second part). Additionally, there are more specific technical conditions, which we need so that the techniques from Section \ref{sec:gluing} can be applied:
\begin{align}
&
\mybox{We require local linearity near the chords in the sense of \eqref{eq:locally-linear}. (This condition is compatible with the previous ones, thanks to the local product structure near chords that lie in $M$.)}
\\ & \label{eq:first-eigenvalue}
\mybox{Take any chord that is contained in $M = E_i$. We require that, for the selfadjoint operator associated to that chord, $\alpha$ should be the smallest positive eigenvalue, and that this eigenvalue is simple. (Because of the condition on the Hamiltonian, the selfadjoint operator associated to that chord splits as a direct sum of two pieces, one an operator \eqref{eq:model-q}, and the other an operator associated to the fibre direction. 
The first summand has an eigenvalue $\alpha$, by \eqref{eq:lambda-m}. Then, our condition is that all the positive eigenvalues of the second summand should be strictly bigger than $\alpha$; this can be achieved by first choosing the Hamiltonian in $M$, and then making $\alpha$ sufficiently small.)}
\end{align}
\end{setup}

More generally, take a surface \eqref{eq:s}, assuming that $\Sigma_- \neq \emptyset$, and equip it with strip-like ends. Our surface should come with a Lagrangian submanifold $L_C$ for each $C \subset \partial S$ (as in Setup \ref{th:l}). For each end, we choose $A_\zeta \in \scrA([0,1])$ and $\alpha_\zeta \in (0,\pi)$, and then corresponding $(H_\zeta, J_\zeta)$ as in Setup \ref{th:setup-floer}, where the Lagrangian submanifolds are the $(L_{\zeta,0}, L_{\zeta,1})$ associated to that end.

\begin{setup} \label{th:jk-adapted}
(i) Pick $A \in \scrA_{\ge}(S,\Sigma)$, meaning that $\epsilon_{\zeta}^*A = A_\zeta$ on the ends. We choose $K \in \Omega^1(S,\smooth(E,\bR))$, such that \eqref{eq:z-tangent} holds, and where for each $\xi \in TS$, $K(\xi)$ is a Hamiltonian function in the class determined by $A(\xi) \in \frakg$, as in Setup \ref{th:h}(i). Similarly, we pick a family $(J_z)_{z \in S}$, where each $J_z$ is as in Setup \ref{th:j}(i). On the ends, we want to have convergence conditions as in \eqref{eq:jk-converge}, together with a slightly stronger version, which generalizes our discussion of continuation maps in \eqref{eq:added-condition}:
\begin{equation}
\mybox{
If $x_\zeta(t)$ is a chord associated to some end $\zeta$, then in a neighbourhood of $\{(t,x_\zeta(t))\} \subset [0,1] \times E$, and on a suitable part $\pm s \gg 0$ of the end, we want to have $\epsilon_\zeta^*K = H_{\zeta,t}\mathit{dt}$ and $\epsilon_\zeta^*J = J_{\zeta,t}$.
}
\end{equation}

(ii) We say that the data chosen above are ``suitable for restriction'' if:
\begin{align}
\label{eq:j-restriction}
& \mybox{Each $J_z$ is as in Setup \ref{th:j}(ii).} \\
& \label{eq:k-restriction}
\mybox{There is a $\beta \in \Omega^1(S)$, satisfying $\epsilon_{\zeta}^*\beta = \alpha_\zeta \mathit{dt}$ on the ends, as well as $\beta|\partial S = 0$, such that the following holds. Near $M = E_i$, and with respect to \eqref{eq:2-trivialization}, each $K(\xi)$ corresponds to $\beta(\xi) p^*(\half |w-i|^2) + \text{\it (a function on $M$)}$.}
\end{align}
\end{setup}

For convenience, let's reproduce here the associated Cauchy-Riemann equation (the same as in \eqref{eq:cauchy-riemann}, but with target space $E$):
\begin{equation} \label{eq:cauchy-riemann-2}
\left\{
\begin{aligned} 
& u: S \longrightarrow E, \\
& u(C) \subset L_C, \\
& \textstyle \lim_{s \rightarrow \pm\infty} u(\epsilon_{\zeta}(s,t)) = x_\zeta(t), \\
& (du - Y)^{0,1} = 0.
\end{aligned}
\right.
\end{equation}
Consider $v = p(u): S \rightarrow W$. On the subset of $S$ where $v$ is sufficiently close to $\partial_\infty W$, it satisfies 
\begin{equation} \label{eq:cauchy-riemann-projected}
\left\{
\begin{aligned}
& v(\partial S) \subset i[0,\infty), \\
& (dv - Z)^{0,1} = 0,
\end{aligned}
\right.
\end{equation}
where $Z$ is the one-form with values in Hamiltonian vector fields on $W$ associated to $A$. Similarly, if Setup \ref{th:jk-adapted}(ii) applies, then on the subset where $u$ is sufficiently close to $M = E_i$, $v = p(u)-i$ satisfies \eqref{eq:plane-cr}, with $\beta$ as in \eqref{eq:k-restriction}. In fact, the Cauchy-Riemann equation then splits into two independent parts, where the base part is \eqref{eq:plane-cr}, and the fibre part is a Cauchy-Riemann equation with target $M$.

\begin{remark}
To be precise, there is a slight difference between the situation here and the original use of \eqref{eq:plane-cr}. In Section \ref{sec:plane}, we used a single $\alpha$ throughout, whereas in the present situation, the behaviour of $\beta$ on each end is determined by a different $\alpha_\zeta$. However, all the discussion from Section \ref{sec:plane} also goes through in this marginally more general context.
\end{remark}

\subsection{Approaching the fibre}
Suppose that we have a surface, with all the necessary auxiliary data, satisfying both parts of Setup \ref{th:jk-adapted}. Using our analysis of \eqref{eq:plane-cr}, we obtain the following.

\begin{lemma} \label{th:approach}
Let $u$ be a solution of \eqref{eq:cauchy-riemann-2} which is not contained in $M = E_i$. Then, 
\begin{align}
\label{eq:avoids-the-fibre}
& \mybox{$u^{-1}(M) = \emptyset$.} \\
\label{eq:cant-go-into}
& \mybox{If $\zeta$ is a negative end, $x_\zeta$ must lie outside $M$.} \\
\label{eq:ifits}
& \mybox{If $\zeta$ is a positive end, and $x_\zeta$ lies in $M$, $\lim_{s \rightarrow+\infty} \exp(\alpha_\zeta s) (p(u(\epsilon_\zeta(s,\cdot))) - i)$ is a positive multiple of $i$.}
\end{align}
\end{lemma}

\begin{proof}
The key is to consider the average intersection number
\begin{equation} \label{eq:average-int}
\half (u \cdot E_{i-\delta} + u \cdot E_{i+\delta}), \;\; \delta>0.
\end{equation}
Equivalently, this is the average degree of $v = p(u) -i$ over $\pm\delta$. Because the Lagrangian submanifolds and the chords are disjoint from $E_{i \pm \delta}$, this is defined, and the same, for all $\delta$. Since $u(S)$ is relatively compact in $E$, letting $\delta \rightarrow \infty$ shows that \eqref{eq:average-int} is zero. 

Let's compare this to what we can learn from $\delta \rightarrow 0$. In that limit, points of $u^{-1}(E_{i \pm \delta})$ either converge to points of $u^{-1}(M)$, or go to $\pm\infty$ on one of the ends, which is necessarily one with $x_\zeta$ contained in $M$. Hence, what we need to do is to determine the contribution of each such limit point, or end, to \eqref{eq:average-int}. For a point $z \in u^{-1}(M) \setminus \partial S$, or equivalently, an interior solution of $v(z) = 0$, the local degree over $0$ (multiplicity of vanishing) is necessarily positive. Then, the contribution of nearby points to the degree over $\pm \delta$, for small $\delta$, is the same, which means that we have a positive integer contribution to \eqref{eq:average-int}. One can similarly analyze the local behaviour at points $z \in v^{-1}(0) \cap \partial S$; for such a point, we have some $d>0$ such that the nearby degree over one of the two points $\pm\delta$ is $d$, and that over the other point is $d-1$. Again, that leads to a positive contribution $d-\half$ to \eqref{eq:average-int}. 

Suppose that $x_\zeta$ is a chord in $M$. For $\pm s \gg 0$, the equation \eqref{eq:plane-cr} for $v$ reduces to \eqref{eq:trivial-floer} for $v(\epsilon_\zeta(s,t))$, so in terms of \eqref{eq:lambda-m} we have an asymptotic behaviour
\begin{equation} \label{eq:v-asymp}
v(\epsilon_\zeta(s,t)) = c_m \exp(-\lambda_m s) \Xi_m + \cdots,\;\; \text{ where $c_m \neq 0$ and }
\begin{cases} m \geq 0 & \zeta \in \Sigma_+, \\
m < 0 & \zeta \in \Sigma_-.
\end{cases}
\end{equation}
Let's first discuss the case of a negative $\zeta$. For $s \ll 0$, consider the loop formed by taking $t \mapsto v(\epsilon(s,t))$ and then connecting its endpoints inside the imaginary axis. That loop will avoid $\pm \delta$ if $\delta$ is small enough, and the average of its winding number around those two points is $m/2 > 0$. Hence, the contribution of such an end to \eqref{eq:average-int} is again positive. For positive $\zeta$, one uses $t \mapsto v(\epsilon(s,-t))$ in a similar way, and gets a contribution of $-m/2 \geq 0$ to \eqref{eq:average-int}. Since the sum of all those contributions must be zero, we immediately get \eqref{eq:avoids-the-fibre} and \eqref{eq:cant-go-into}; and moreover, for positive ends we see that $m = 0$ in \eqref{eq:v-asymp}. 

Let's push the discussion of positive ends a little further. The relevant eigenvalue and eigenvector being $\lambda_0 = \alpha_\zeta$ and $\Xi_0 = i$, we must have
\begin{equation} \label{eq:which-side}
\text{for $s \gg 0$, } \mathrm{im}(v(\epsilon_\zeta(s,t))) 
\begin{cases}
> 0 & \text{if } c_0 > 0, \\
< 0 & \text{if } c_0 < 0.
\end{cases}
\end{equation}
Suppose now that $c_0 < 0$. Combining \eqref{eq:which-side} with the boundary condition, one finds that 
\begin{equation}
u(\epsilon_\zeta(s,k)) \in p^{-1}(i(0,1)) \;\; \text{ for $k = 0,1$ and $s \gg 0$.}
\end{equation}
Recall that by our previous argument, $p(u)$ avoids $i$. Hence, $u$ must map the components of $\partial S$ containing $\epsilon_\zeta(s,k)$ to $p^{-1}(i(0,1))$. This means that for any end $\zeta'$ adjacent to those boundary components, we must have $x_{\zeta'} \in M$, and moreover, $p(u)$ again approaches that limit from below in $W$. In other words, those $\zeta'$ have the same property as $\zeta$. One can go from one point of $\Sigma$ to its neighbours in this way, until one hits a point of $\Sigma_-$ (which always exists by assumption), yielding a contradiction. Hence $c_0 > 0$, which is \eqref{eq:ifits}.
\end{proof}

\subsection{Compactness}
The only concern here (beyond standard Fukaya-categorical issues) has to do with the noncompactness of $E$ and its Lagrangian submanifolds. One has to rule out the possibility of sequences of solutions (with bounded energy) escaping to infinity. We approach this using the compactification \eqref{eq:e-bar}, following \cite{seidel17, seidel18} (more specfifically, \cite[Lemma 4.10]{seidel18}); because the arguments are quite similar, only a limited amount of details will be given. The Cauchy-Riemann equation in \eqref{eq:cauchy-riemann-2} extends to maps $S \rightarrow \bar{E}$, because the almost complex structures, the Hamiltonian vector fields in the inhomogeneous term, and the Lagrangian submanifolds, have smooth extensions. Let's consider solutions of this equation, in a slightly more general context, where strip-like ends play no role, meaning that $S$ could be any Riemann surface with boundary, carrying a one-form which satisfies \eqref{eq:covariantly-constant-lambda}, \eqref{eq:p-space-curvature}.

\begin{lemma} \label{th:touch-boundary}
For a solution $u: S \rightarrow \bar{E}$ of (the extended version of) \eqref{eq:cauchy-riemann-2}, $u^{-1}(\partial_\infty E)$ is open and closed. Moreover, on that subset, $v = p(u)$ satisfies $dv = v^*Z$, where $Z$ is as in \eqref{eq:cauchy-riemann-projected}, but extended continuously to $\bar{W}$.
\end{lemma}

\begin{proof}
There is an open neighbourhood $U \subset S$ of $u^{-1}(\partial_\infty E)$, such that $v = p(u)|U$ satisfies the extended version of \eqref{eq:cauchy-riemann-projected}. For concreteness, let's write this down:
\begin{equation} \label{eq:cauchy-riemann-projected-2}
\left\{
\begin{aligned}
& v: U \longrightarrow \bar{W}, \\
& v(U \cap \partial S) \subset i[0,\infty), \\
& (dv - v^*Z)^{0,1} = 0.
\end{aligned}
\right.
\end{equation}
Note that the vector fields \eqref{eq:g-action}, which are those that appear in the inhomogeneous term $Z$, extend not just to the disc $\bar{W}$, but to the whole Riemann sphere $\bC \cup \{\infty\}$; and vanish at $\infty$.

Suppose that $v(z) = \infty \in \bar{W}$. Necessarily, $z$ is an interior point of $S$. Because the inhomogeneous term vanishes at $\infty$, a standard reduction to the case of pseudo-holomorphic curves (Gromov's trick) shows that either $v$ is locally constant near $z$, or else $v(z)$ is locally an open map, taking on all values on the Riemann sphere close to $\infty$. The second possibility is of course impossible, since $v$ remains inside $\bar{W}$. So, locally $v \equiv \infty$, which is what we wanted.

Now suppose that $v(z) \in \bR$. By a local gauge transformation, as used in the proof of Lemma \ref{th:partial-trivialization}, one can transform $v$ into a solution of a simpler equation 
\begin{equation} \label{eq:gauged-equation}
\left\{
\begin{aligned}
& v^\dag: U^\dag \longrightarrow \bar{W}, \\
& v^\dag(U^\dag \cap \partial S) \subset i[0,\infty), \\
& (v^\dag - \gamma)^{0,1} = 0,
\end{aligned}
\right.
\end{equation}
with $\gamma$ as in \eqref{eq:abelian-2}, \eqref{eq:abelian-3}. Then, $\psi = -\mathrm{im}(v^\dag)$ is a subharmonic function with Neumann boundary conditions, which is $\leq 0$ everywhere. If $\psi = 0$ at some point, then that must be a maximum, hence $\psi = 0$ nearby. This means $d \mathrm{im}(v^\dag) = 0$ locally, which by the Cauchy-Riemann equation \eqref{eq:gauged-equation} implies $d\mathrm{re}(v^\dag) = \gamma$. Undoing the gauge transformation shows that $dv = v^*Z$ near our point, as desired.
\end{proof}

\begin{lemma} \label{th:contradiction-on-the-end}
Take $S = (s_-,s_+) \times [0,1]$, with $A \in \Omega^1(S,\frakg)$ being the pullback of some connection in $ \scrA([0,1])$. Then, any solution $u: S \rightarrow \bar{E}$ of (the extended version of) \eqref{eq:cauchy-riemann-2} is necessarily contained in $E$.
\end{lemma}

\begin{proof}
Suppose that, on the contrary, $u^{-1}(\partial_\infty E) \neq \emptyset$. We can then apply Lemma \ref{th:touch-boundary}. Because $S$ has boundary points, it's impossible for $v = p(u)$ to be a constant map with value $\infty$. The alternative is that $v$ takes values in $\bR$, has boundary values $0$, and is covariantly constant for $A$. But that's also impossible, since parallel transport from $(s,0)$ to $(s,1)$ takes $0$ to a positive point on the real line, by definition \eqref{eq:move-past}.
\end{proof}

With that at hand, we return to the original case of surfaces $S$ with strip-like ends, and assuming that there is at least one such end.

\begin{proposition} \label{th:contained}
Suppose that $u_k: S \rightarrow E$ is a sequence of solution of \eqref{eq:cauchy-riemann}, with bounded energy. Then, there is a compact subset of $E$ which contains the images of all $u_k$.
\end{proposition}

\begin{proof}[Sketch of proof]
We can consider the $u_k$ as maps to $\bar{E}$. Because our metric on $E$ blows up at infinity, whereas a metric on $\bar{E}$ remains bounded, the given energy bound implies (compare \cite[Equation (4.19) and Lemma 4.5]{seidel18}) that
\begin{equation}
\int_S \|du_k - Y\|^2_{\bar{E}} \;\; \text{is bounded (independently of $k$).}
\end{equation}
Via Gromov's trick, we can convert the $u_k$ into a sequence of pseudo-holomorphic sections $S \rightarrow S \times \bar{E}$, with totally real boundary conditions, whose energies (with respect to a metric on $\bar{E}$) are bounded on any fixed compact subset of $S$. This is enough to apply a Gromov compactness argument: the outcome is convergence of a subsequence (on compact subsets) to a solution $u_\infty: S \rightarrow \bar{E}$ of \eqref{eq:cauchy-riemann}, plus possible (sphere and disc) bubble components, which are just pseudo-holomorphic maps.

Suppose that, possibly after passing to a subsequence, there are points $z_k \in S$, contained in a compact subset of $S$, such that $u_k(z_k) \rightarrow \partial_\infty E$. This means that the Gromov limit must intersect $\partial_\infty E$. If $u_\infty$ intersects $\partial_\infty E$, it must be entirely contained in it, by Lemma \ref{th:touch-boundary}; and the same holds for the bubble components. Connectedness of the limit therefore implies that all of it must be contained in $\partial_\infty E$. We can then look at a piece of a strip-like end, and derive a contradiction to Lemma \ref{th:contradiction-on-the-end}.

The other possible situation is that, again after passing to a subsequence, there are points $z_k = \epsilon_\zeta(s_k,t_k)$, such that $u_k(z_k) \rightarrow \partial_\infty E$ and $s_k \rightarrow \pm\infty$. In that case, one looks at the reparametrized maps $u_k(\epsilon_\zeta(s+s_k,t))$. On each compact subset of $\bR \times [0,1]$, these maps are defined for $k \gg 0$ and satisfy a sequence of Cauchy-Riemann equation which, in the limit, converge to Floer's equation. Gromov compactness again applies (in a slightly more complicated version) and yields a limit, which is a solution $u_\infty: \bR \times [0,1] \rightarrow \bar{E}$ of Floer's equation, with possible bubbles attached, and where one of the components must intersect $\partial_\infty E$. The rest is as before.
\end{proof}

In Fukaya-categorical constructions, we are dealing with a slightly more general situation where the domain Riemann surface varies, and can degenerate; but the same strategy applies.

\section{Acyclicity\label{sec:acyclic}}
This section provides the Floer-theoretic computation underlying the acyclicity statement in Theorem \ref{th:main}. This does not require the full construction of \eqref{eq:serre-trans}, only its most elementary piece. In particular, we will be working only with holomorphic curves that lie in the complement of the symplectic hypersurface $\Omega_E$, which effectively means in a situation where the symplectic form is exact.

\subsection{The first version of Floer cohomology}
To make our argument easier, we choose all the geometric data in very specific forms.

\begin{setup} \label{th:l0l1}
Fix two Lagrangian submanifolds $L_0,L_1 \subset E$. These lie in the general class from Setup \ref{th:l}, but are subject to an additional condition:
\begin{equation}
\mybox{
$p(L_k) \cap \{\mathrm{im}(w) \leq 4\} \subset \{\mathrm{re}(w) = \iota_k(\mathrm{im}(w))\}$,
with functions $\iota_k: (0,4] \rightarrow \bR$ such that
\[
\left\{ 
\begin{aligned} &
\iota_k(y) = 0 && \text{for $y \leq 2$,} \\
&
-1 < \iota_0(y) < \iota_1(y) < 1 &&\text{for $2 < y \leq 4$.} \\
 \end{aligned}
\right.
\]
} 
\end{equation}
\end{setup}

\begin{setup} \label{th:plus-floer}
For Floer's equation, we choose $A_- \in \scrA([0,1])$, $\alpha_- \in (0,\pi)$, and $(H_{-,t},J_{-,t})$ as in Setup \ref{th:setup-floer}, but with some more specific requirements.
\begin{align} 
& \label{eq:zero-box}
\mybox{$H_{-,t}$ is zero on $p^{-1}(B)$, where $B = \{ -1 \leq \mathrm{re}(w) \leq 1, \, 3 \leq \mathrm{im}(w) \leq 4\}$. On the same subset, $p$ is $J_{-,t}$-holomorphic.} \\
& \label{eq:sharpen}
\mybox{We sharpen \eqref{eq:choose-floer-1} as follows: each chord is either contained in $p^{-1}(\{\mathrm{im}(w) > 4\})$, or else in $M = E_i$ (see Figure \ref{fig:l0l1}).}
\end{align}
\end{setup}

As an immediate consequence of Lemma \ref{th:approach}, we have:

\begin{lemma} \label{th:geometric-minus}
In the situation of Setup \ref{th:plus-floer}, each solution of Floer's equation satisfies one of the following: either, its negative limit $x_- = \lim_{s \rightarrow -\infty} u(s,\cdot)$ lies in $p^{-1}(\{\mathrm{im}(w) > 4\})$; or else, the entire trajectory is contained in $M = E_i$.
\end{lemma}

Let's quickly look at the implications of this for transversality arguments. The first kind of solutions is regular for generic choices of auxiliary data, since part of it goes through a region where those data are essentially unconstrained. For the second kind of solutions, those contained in $M$, regularity in $E$ is equivalent to regularity as a Floer trajectory in $M$, by an argument involving the comparison of linearized operators. To  spell that out, let $u$ be such a solution, and $D_{M,u}$, $D_{E,u}$ the linearized operators in $M$ and $E$, respectively. There is a splitting
\begin{equation} \label{eq:les-operator}
\xymatrix{
\parbox{17em}{
$W^{1,2}(\bR \times [0,1], \bC, i\bR, i\bR)$ \newline \hspace*{6em} $\oplus$ \newline
$W^{1,2}(\bR \times [0,1], u^*TM, u^*T(L_0 \cap M),$ \newline \hspace*{8em} $u^*T(L_1 \cap M))$
} \ar[rr]^-{D \oplus D_{M,u}} \ar[d]_-{\iso} 
&&
\parbox{10em}{
$L^2(\bR \times [0,1], \bC)$ \newline \hspace*{3em} $\oplus$ \newline $L^2(\bR \times [0,1], u^*TM)$} \ar[d]^-{\iso}
\\
W^{1,2}(\bR \times [0,1], u^*TE, u^*TL_0, u^*TL_1) \ar[rr]^-{D_{E,u}} &&
L^2(\bR \times [0,1], u^*TE) 
}
\end{equation}
The operator $D$ is one of those considered in Section \ref{subsec:weights}, and clearly invertible; which implies the observation we have made. As a consequence, transversality can be achieved in spite of the constraints we have imposed. We use solutions of Floer's equation which have zero intersection number with $\Omega_E$, or equivalently (by a positivity of intersections argument, as in Lemma \ref{th:positivity}) which are disjoint from $\Omega_E$, to define the Floer cochain complex 
\begin{equation} \label{eq:plus-complex}
C^*_- = \mathit{CF}^*(L_0,L_1;H_-).
\end{equation}
We can then reformulate Lemma \ref{th:geometric-minus} in more algebraic terms, as follows:
\begin{figure}
\begin{centering}
\includegraphics{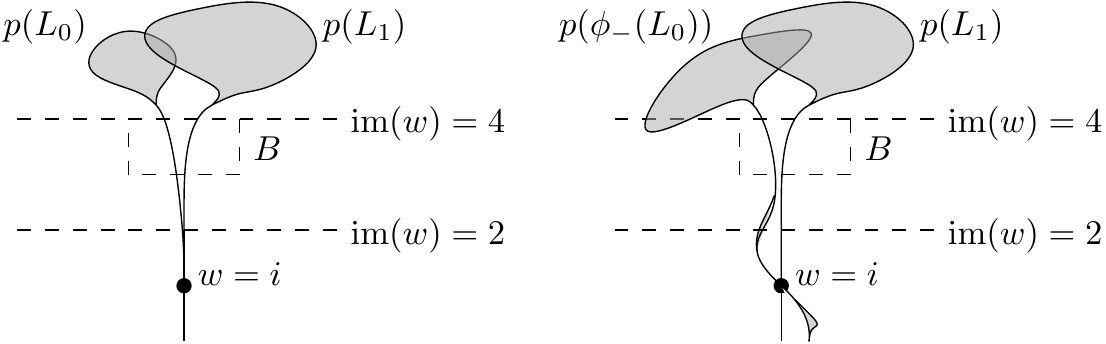}
\caption{\label{fig:l0l1}The Lagrangians from Setup \ref{th:l0l1} and their perturbed versions. Here, $\phi_-$ is the time-one map of $H_{-,t}$, in the notation from \eqref{eq:plus-complex}.}
\end{centering}
\end{figure}

\begin{lemma} \label{th:subcomplex-1}
Write
\begin{equation} \label{eq:decompose-1}
C_-^* = C^*_{-,\mathit{sub}} \oplus C^*_{-,\mathit{quot}},
\end{equation}
where the first summand is formed by the chords lying over $\{\mathrm{im}(w) > 4\}$, and the second one by those in $M$. Then (as the notation suggests) the first summand is a subcomplex. Moreover, the induced differential on the quotient computes the Floer cohomology of $(L_0 \cap M, L_1 \cap M)$ inside $M \setminus \Omega_M$.
\end{lemma}

For future applications, we want to explain a different approach to proving the first part of Lemma \ref{th:subcomplex-1} (namely, that $C^*_{-,\mathit{sub}}$ is a subcomplex), which relies on \eqref{eq:zero-box} and the following elementary topological argument:

\begin{lemma} \label{th:box-lemma}
Let $R = [-C,C] \times [0,1]$, for some $C>0$. Take any map 
\begin{equation} \label{eq:rectangle-map}
\left\{
\begin{aligned}
& u: R \longrightarrow E, \\
& u(s,0) \in L_0, \; u(s,1) \in L_1, \\
& \mathrm{im}(p \circ u)(-C,t) \leq 3, \\
& \mathrm{im}(p \circ u)(C,t) \geq 4.
\end{aligned}
\right.
\end{equation}
Then there is a point $w \in B$, the box from \eqref{eq:zero-box}, such that $u^{-1}(E_w) \cap \partial R = \emptyset$, and $u \cdot E_w < 0$.
\end{lemma}

\begin{proof}
Take the contractible open subset, contained in $B$, whose boundary is the union of $\{\mathrm{im}(w) = j, \, \iota_0(j) \leq \mathrm{re}(w) \leq \iota_1(j)\}$ for $j = 3,4$, and the graphs of $\iota_k|[3,4]$. 
By assumption, the boundary values of $p \circ u$ form a loop which winds $-1$ around that subset. Hence, taking any point in that subset yields the desired property.
\end{proof}

Returning to \eqref{eq:decompose-1}, suppose that $C^*_{-,\mathit{sub}}$ is not a subcomplex. Then there is a Floer trajectory with $p(\lim_{s \rightarrow -\infty} u(s,t)) = i$, $\mathrm{im}\, p(\lim_{s \rightarrow +\infty} u(s,t)) > 4$. By restricting to a suitably large rectangle $R \subset \bR \times [0,1]$, one gets a map as in \eqref{eq:rectangle-map}. But by \eqref{eq:zero-box}, $p(u)$ is a holomorphic function on the preimage of the interior of $B$, hence must have nonnegative degree over each point of that interior, which contradicts Lemma \ref{th:box-lemma}.

\subsection{The second version of Floer cohomology}
We use the same Lagrangian submanifolds as before.

\begin{setup} \label{th:dual-h}
For the pair $(L_1,L_0)$ (in this order!), choose $A = a_t \mathit{dt}$, $\alpha$, and $(H_t,J_t)$ as in Setup \ref{th:setup-floer}, and also subject to the following conditions:
\begin{align} 
& \label{eq:zero-box-2}
\mybox{On the preimage of the rectangle $B$, we require the same as in \eqref{eq:zero-box}.}
\\
& 
\mybox{We need a weaker version of \eqref{eq:sharpen}, saying that chords must be 
disjoint from $p^{-1}(\{3 \leq \mathrm{im}(w) \leq 4\})$.}
\end{align}
\end{setup}

We now revert to $(L_0,L_1)$, and adjust the data from Setup \ref{th:dual-h} to obtain the dual Floer cochain complex (see Figure \ref{fig:second-hamiltonian}):
\begin{equation} \label{eq:minus-complex}
\left\{
\begin{aligned}
& A_+ = -a_{1-t} \mathit{dt}, \;\; \alpha_+ = -\alpha, \\
& (H_{+,t}, J_{+,t}) = (-H_{1-t}, J_{1-t}), \\
& C^*_+ = \mathit{CF}^*(L_0,L_1;H_+) \iso \mathit{CF}^{n-*}(L_1,L_0;H)^\vee.
\end{aligned}
\right.
\end{equation}

\begin{lemma} \label{th:geometric-plus}
Write 
\begin{equation} \label{eq:second-complex}
C^*_+ = C^*_{+,\mathit{sub}}\oplus C^*_{+,\mathit{quot}},
\end{equation}
where the first piece is as in Lemma \ref{th:geometric-minus}, and the second one is generated by the remaining chords, meaning those in $p^{-1}(\{\mathrm{re}(w) < 3\})$. Then, $C^*_{+,\mathit{sub}}$ is again a subcomplex. Moreover, the quotient complex $C^*_{+,\mathit{quot}}$ is acyclic.
\end{lemma}

\begin{proof}
The subcomplex statement uses Lemma \ref{th:box-lemma} in the same way as before. The other statement follows from a chain of fairly straightforward observations. First of all, the ``suitable for restriction'' property of the Hamiltonian, Setup \ref{th:h}(ii), is not necessary here, and could be dropped. In that wider context, it is easy to choose a modified Hamiltonian $\tilde{H}_+$ so that the associated chain complex $\tilde{C}^*_+$ has $\tilde{C}^*_{+,\mathit{quot}} = 0$, see Figure \ref{fig:second-hamiltonian}. More precisely, that can be achieved by taking our original choice of Hamiltonian, and changing it only on a compact subset contained in $p^{-1}(\{\mathrm{im}(w) \leq 3\})$. This modified choice is related to the original choice by a pair of mutually homotopy inverse continuation maps. Moreover, Lemma \ref{th:box-lemma} continues to be applicable, meaning that both the continuation maps and the respective chain homotopies preserve the subcomplexes. As a consequence, the homotopy types of $C^*_{+,\mathit{quot}}$ and $\tilde{C}^*_{+,\mathit{quot}}$ agree, hence the desired acyclicity holds.
\end{proof}
\begin{figure}
\begin{centering}
\includegraphics{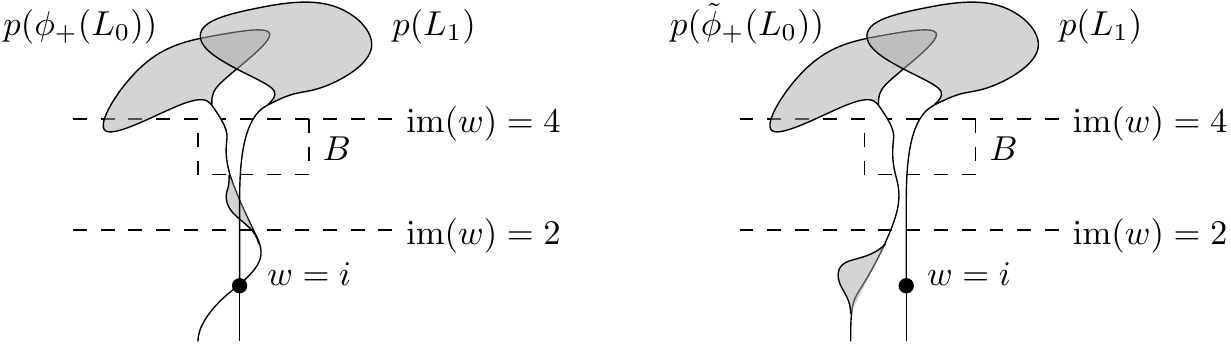}
\caption{\label{fig:second-hamiltonian}Continuing from Figure \ref{fig:l0l1}, we show the effect of applying the time-one map $\phi_+$ of the Hamiltonian from \eqref{eq:minus-complex} (on the left), as well as its modification from the proof of Lemma \ref{th:geometric-plus} (on the right).}
\end{centering}
\end{figure}

\subsection{The continuation map}
Our next step is to look at a Cauchy-Riemann equation on $\bR \times [0,1]$, with boundary conditions $(L_0,L_1)$ and auxiliary data $(H_{1,s,t},J_{1,s,t})$ (the meaning of the subscript $1$ will become clear later). If we wanted to adhere to the formalism from Section \ref{subsec:our-cr}, this would be formulated as follows.

\begin{setup} \label{th:setup-cont}
Consider $\bR \times [0,1]$ as a surface with two negative ends (which means that the pairs of Lagrangian submanifolds associated to the ends are $(L_0,L_1)$ at $s \rightarrow -\infty$ and $(L_1,L_0)$ as $s \rightarrow +\infty$). The Floer data associated to the ends will be as in Setup \ref{th:l0l1} and \ref{th:dual-h}, respectively. Given those, we follow Setup \ref{th:jk-adapted} in choosing $(K_1 = H_{1,s,t} \mathit{dt}, J_1)$, using part (i) only. We want to impose one more condition, as in \eqref{eq:zero-box}, \eqref{eq:zero-box-2}:
\begin{equation}
\label{eq:zero-box-3}
\mybox{On the preimage of the rectangle $B$, $H_{1,s,t}$ is zero and $p$ is $J_{1,s,t}$-holomorphic.}
\end{equation}
\end{setup}

Equivalently, this sets up a continuation map equation between the data from Setup \ref{th:plus-floer} and \eqref{eq:minus-complex}. Transversality is unproblematic. Again using only solutions which avoid $\Omega_E$, we get a chain map
\begin{equation} \label{eq:pm-continuation}
C^*_+ \longrightarrow C^*_-.
\end{equation}

\begin{lemma} \label{th:geometric-continuation}
The map \eqref{eq:pm-continuation} takes the subcomplex from Lemma \ref{th:geometric-plus} to that from Lemma \ref{th:geometric-minus}. Moreover, on those subcomplexes, the map is a quasi-isomorphism.
\end{lemma}

\begin{proof}
The first part (preservation of subcomplexes) is again a simple application of Lemma \ref{th:box-lemma}.
By using continuation maps relating different versions of the complexes, one also sees that, if the second part (quasi-isomorphism when restricted to the subcomplexes) is true for some choice of $C^*_\pm$ and of \eqref{eq:pm-continuation}, then it holds for all other choices as well. Moreover, when defining $C^*_+$, we can drop part (ii) of Setup \ref{th:h}, and the statements we've just made are still true.

With that in mind, we can restrict to a particularly simple situation, where one uses a modified Hamiltonian $\tilde{H}_+$ and complex $\tilde{C}_+^*$, as in the proof of Lemma \ref{th:geometric-plus}, meaning that all chords are contained in $p^{-1}(\{\mathrm{im}(w) > 4\})$. At the other end, one can assume that $H_{-,t} = \tilde{H}_{+,t}$ on $p^{-1}(\{\mathrm{im}(w) > 4\})$. For the equation underlying the map \eqref{eq:pm-continuation}, one can assume that
\begin{align}
& \label{eq:monotone-increasing}
\mybox{the Hamiltonian is nonincreasing, meaning that $\partial_s H_{1,s,t} \leq 0$;}
\\
& \label{eq:constant-solution}
\mybox{we have $H_{1,s,t} = H_{-,t} = \tilde{H}_{+,t}$ on $p^{-1}(\{\mathrm{im}(w) \geq 4\})$.}
\end{align}
(Figure \ref{fig:move-along} summarizes the situation.) The condition \eqref{eq:monotone-increasing}, and the fact that we are only considering solutions that avoid $\Omega_E$, imply that the continuation map is compatible with the action filtration; and by \eqref{eq:constant-solution}, there is a constant solution for each chord of $\tilde{H}_{+,t}$. Those two properties imply that the continuation map induces an isomorphism of chain complexes $\tilde{C}^*_+ = \tilde{C}^*_{+,\mathit{sub}} \longrightarrow \tilde{C}^*_{-,\mathit{sub}}$.
\end{proof}
\begin{figure}
\begin{centering}
\includegraphics{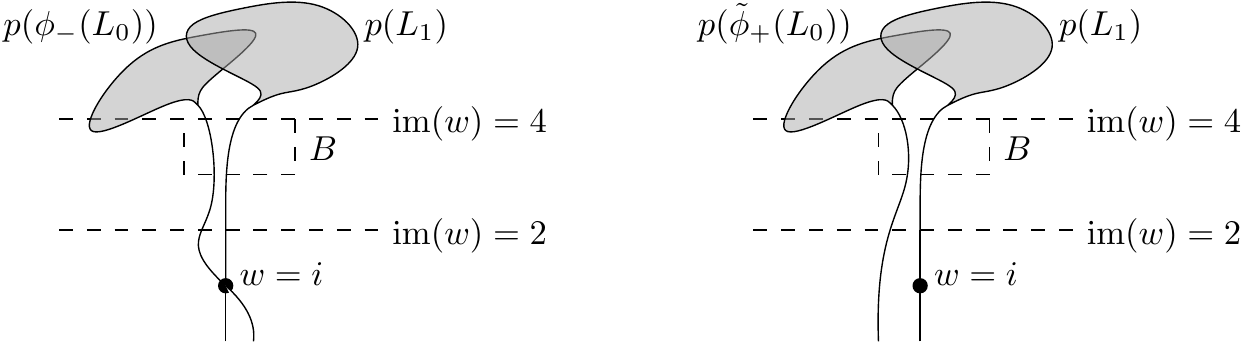}
\caption{\label{fig:move-along}A schematic idea of the choices in the proof of Lemma \ref{th:geometric-continuation}.}
\end{centering}
\end{figure}

\subsection{The nullhomotopy\label{subsec:simple-nullhomotopy}}
Still on the same surface, and with the same behaviour on the ends, we now introduce a family $(K_{r,s,t} = H_{r,s,t} \mathit{dt},J_{r,s,t})$ with an additional parameter $r \in [0,1]$.

\begin{setup} \label{th:parametrized-setup}
At the parameter value $r = 1$, the data reduce to the previously chosen ones. Moreover:
\begin{align}
& \mybox{At the parameter value $r =0$, all $(H_{0,s,t}, J_{0,s,t})$ are ``suitable for restriction'', in the sense of Setup \ref{th:j}(ii) and \ref{th:h}(ii). More specifically, near $M = E_i$, we write $H_{0,s,t} = b_{s,t} p^*(\half |w-i|^2) + \text{(a function on $M$ depending on $s,t$)}$, for some $b_{s,t}$ which is equal to $\alpha_{L_0,L_1,-}$ for $s \ll 0$, and $\alpha_{L_0,L_1,+}$ for $s \gg 0$.}
\\ \label{eq:dr}
& \mybox{We still look at $(H_{r,s,t},J_{r,s,t})$ near $M = E_i$, and for the parameter value $r = 0$. There, we require that $\partial_r H_{r,s,t} = -c_{s,t} p^*(\mathrm{im}(w)) + (\text{constant})$, where $c_{s,t} \in \bR$ becomes zero for $\pm s \gg 0$. Moreover, $\partial_r J_{r,s,t} = 0$, again for $r = 0$ and in a neighbourhood of $M$.
}
\\
& \label{eq:nonzero}
\mybox{Continuing from \eqref{eq:dr}, consider the operator $D^{\mathit{para}}(R,\Upsilon) = \partial_s \Upsilon +i\partial_t \Upsilon + b_{s,t} \Upsilon -i R c_{s,t}$, for complex-valued functions with purely imaginary boundary values, as in Section \ref{sec:plane}. Then, if we take the resulting number $\Gamma = \Gamma(\bR \times [0,1], \, b_{s,t} \mathit{dt}, c_{s,t} \mathit{dt})$ as in \eqref{eq:lim-gamma}, that satisfies $\Gamma < 0$.
}
\\ 
& \mybox{On $p^{-1}(B)$, we ask for the same as in \eqref{eq:zero-box-3}.} \label{eq:zero-box-4} \\
& \mybox{Finally, we ask that our family should have a smooth extension for small $r<0$ (this condition is of a technical nature).}
\end{align}
\end{setup}

Transversality for our parametrized moduli space is mostly straightforward, the only issue being the solutions for $r = 0$ which are contained in $M$. For those, we have the following:

\begin{lemma} \label{th:double-transversality}
For generic choices, the following holds. Take a point $(r = 0, u: \bR \times [0,1] \rightarrow M)$ in the parametrized moduli space. Then,
\begin{align}
& \label{eq:inside-m-solution}
\mybox{$u$ is regular as a solution of the continuation map equation in $M$, and $(0,u)$ is also regular as an element of the parametrized moduli space with target space $E$.} \\
& \label{eq:other-limit}
\mybox{Suppose now that our point is isolated. Take the linearized operator $D_{E,(0,u)}^{\mathit{para}}$, where the subscript $E$ clarifies the target space involved. Then there is a one-dimensional space of solutions of $D^{\mathit{para}}_{E,(0,u)}(R,\Upsilon) = 0$ which decay as $s \rightarrow -\infty$, and such that $\lim_{s \rightarrow +\infty} \exp(\alpha_+ s) \Upsilon$ is a constant multiple of $(\Xi_0 = i,0) \in \bC \oplus \mathit{TM} = \mathit{TE}$, the tangent space being taken along the relevant chord. Moreover, in that space, we have an element which has $R>0$, and such that the multiplicative constant is negative.}
\end{align}
\end{lemma}

\begin{proof}
Let's look at the analogue of \eqref{eq:les-operator}:
\begin{equation} \label{eq:les-operator-2}
\xymatrix{
\parbox{17em}{
$\bR \oplus W^{1,2}(\bR \times [0,1], \bC, i\bR, i\bR)$ \newline \hspace*{6em} $\oplus$ \newline
$W^{1,2}(\bR \times [0,1], u^*TM, u^*T(L_0 \cap M),$ \newline \hspace*{8em} $u^*T(L_1 \cap M))$
} \ar[rr]^-{D^{\mathit{para}} \oplus D_{M,u}} \ar[d]_-{\iso} 
&&
\parbox{10em}{
$L^2(\bR \times [0,1], \bC)$ \newline \hspace*{3em} $\oplus$ \newline $L^2(\bR \times [0,1], u^*TM)$} \ar[d]^-{\iso}
\\
W^{1,2}(\bR \times [0,1], u^*TE, u^*TL_0, u^*TL_1) \ar[rr]^-{D^{\mathit{para}}_{E,(0,u)}} &&
L^2(\bR \times [0,1], u^*TE) 
}
\end{equation}
Here, the second summand corresponds to deformations of $u$ in $M$-direction. The first summand consists of deformations of the parameter, and simultaneously of $u$ in $W$-direction. When one does that, the $M$-part of the Cauchy-Riemann equation remains the same to first order, thanks to the conditions imposed in \eqref{eq:dr}. Hence, the operator $D_{E,(0,u)}^{\mathit{para}}$ is indeed compatible with the splitting.
$D^{\mathit{para}}$ is exactly the operator from \eqref{eq:nonzero}, hence an isomorphism by Lemma \ref{th:its-an-isomorphism}. Transversality applied to $M$ implies that for generic choices, $D_{M,u}$ is onto. Therefore, $D^{\mathit{para}}_{E,(0,u)}$ is also generically onto.

As for \eqref{eq:other-limit}, because of the assumption that the point should be isolated, both $D_{M,u}$ and $D^{\mathit{para}}_{E,(0,u)}$ are invertible. Now let's look at all the same operators in weighted Sobolev spaces, where solutions still decay as $s \rightarrow -\infty$, but are allowed to grow at rate slightly larger than $-\alpha_+$ as $s \rightarrow +\infty$. In that context, $D_{M,u}$ is still invertible, because the allowed growth rate is less than the lowest positive eigenvalue of the selfadjoint operators that appear at the ends, by \eqref{eq:first-eigenvalue}. The other piece $D^{\mathit{para}}$ is surjective with one-dimensional kernel. It follows that $D^{\mathit{para}}_{E,(0,u)}$ is also surjective with one-dimensional kernel. The rest follows from the definition of $\Gamma$, see \eqref{eq:lim-gamma}.
\end{proof}

We will use the parametrized moduli space associated to our $(H_{r,s,t}, J_{r,s,t})$, but only those solutions with $r>0$, whose limit as $s \rightarrow -\infty$ lies in $M = E_i$, and which (as usual) avoid $\Omega_E$. Counting those solutions yields a map
\begin{equation} \label{eq:our-homotopy}
C^*_+ \longrightarrow C^{*-1}_{-,\mathit{quot}}.
\end{equation}
Excluding the $r = 0$ solutions is of course possible, if maybe jarring; it will be up to us to show that this is a meaningful concept. First of all, note that due to \eqref{eq:zero-box-4}, the (by now familiar) argument involving Lemma \ref{th:box-lemma} still applies, and therefore we have the following:

\begin{lemma} \label{th:homotopy-vanishes}
The map \eqref{eq:our-homotopy} is zero on $C^*_{+,\mathit{sub}}$.
\end{lemma}

The next step is the crucial one:

\begin{prop} \label{th:nullhomotopy-1}
The map \eqref{eq:our-homotopy} is a nullhomotopy for the composition of \eqref{eq:pm-continuation} and projection $C^*_- \rightarrow C^*_{-,\mathit{quot}}$.
\end{prop}

\begin{proof}
Following the usual Floer-theoretic principle, we have to look at the boundary points and ends of one-dimensional parametrized moduli spaces. More precisely, we consider only the parts of the one-dimensional parametrized moduli spaces which consist of solutions with $r>0$. There are boundary points at $r = 1$, and those by construction are just the points that enter into \eqref{eq:pm-continuation} projected to $C^*_{-,\mathit{quot}}$, so they are necessary for our desired formula to hold. There are also ends that occur by splitting off a solution of Floer's equation on the left or right, with the principal component retaining the property that $r>0$, and those give rise to the chain homotopy equation. Hence, the desired property will hold if there are no other kinds of ends.

Suppose that we had an end of the space under consideration, which can be compactified by adding a solution $u$ with $r = 0$. Because the $r = 0$ data are ``suitable for restriction'', Lemma \ref{th:approach} applies here, and as a consequence, $u$ must be contained in $M = E_i$. The relevant indices, for the same reason as in \eqref{eq:les-operator-2}, are $\mathrm{index}(D_{E,(0,u)}^{\mathit{para}}) = \mathrm{index}(D_{M,u}) = 1$. Both operators are onto, see Lemma \ref{th:double-transversality}. Hence, our $u$ belongs to a one-parameter family of solutions, for $r = 0$, which lie inside $M$. For dimension reasons, the same family describes the surroundings of $(0,u)$ in the parametrized moduli space. Hence $(0,u)$ could not, after all, arise as the limit of solutions with $r>0$.

The remaining worry are broken solutions with $r = 0$. Such a solution consists of a principal component, which satisfies the continuation map equation; and other components which are Floer trajectories. In our case, because of dimension constraints and the given regularity properties, it follows that the principal component must be an isolated point of the parametrized moduli space, and that there is only one Floer trajectory involved, which is also an isolated one. 

Suppose that the broken solution is of the form $(u^2,(0,u^1))$, where $u^2$ is Floer trajectory for $(H_-,J_-)$, and $(0,u^1)$ a continuation map solution. We are then necessarily in the second case of Lemma \ref{th:geometric-minus}, which means that $u^2$ is contained in $M$. By the analogous argument, $u^1$ is also contained in $M$. Then, we encounter the same phenomenon as before: one can glue together the two components to a family of continuation map solutions that have $r = 0$ and remain inside $M$, so this can't after all happen as a limit of an $r>0$ end of the parametrized moduli space.

Now suppose that the broken solution is of the form $((0,u^2), u^1)$, where $u^2$ is a continuation map solution and $u^1$ a Floer trajectory for $(H_+,J_+)$. Then, $u^2$ is necessarily contained in $M$. If the same holds for $u^1$, then the previous argument applies here as well. Hence, the remaining case is a broken solution $((0,u^2), u^1)$, where $u^2$ is contained in $M$ but $u^1$ is not. Addressing this case requires a much more substantial effort. Let $x$ be the chord (in $M$) that is the common limit of $u^2$ and $u^1$.
The first ingredient is
\begin{equation} \label{eq:full-positive-limit}
\mybox{$\lim_{s \rightarrow -\infty} \exp(\alpha_+ s) \partial_s u^1$ is a positive multiple of $(\Xi_0 = i, 0) \in \bC \oplus \mathit{TM}_{x(t)} = \mathit{TE}_{x(t)}$.}
\end{equation}
To see that, remember that one can think of the reversed trajectory $u^1(-s,1-t)$ as a Floer trajectory for the Hamiltonian from Setup \ref{th:dual-h}. Lemma \ref{th:approach} applies to the reversed trajectory; the outcome being that
\begin{equation} \textstyle
\lim_{s \rightarrow-\infty} \exp(\alpha_+ s)( p(u^1(s,\cdot))-i) \text{ is a positive multiple of $i$.}
\end{equation}
In terms of the standard expansion (for $s \ll 0$) $p(u^1(s,\cdot)) = i + \sum_{m \geq 0} c_m \exp(\lambda_m s) \Xi_m(1-t)$, see \eqref{eq:lambda-m} for the notation, this means that $c_0 > 0$. Differentiating yields
\begin{equation} \textstyle \label{eq:partial-limit}
\lim_{s \rightarrow -\infty} \exp(\alpha_+ s) Dp( \partial_s u^1) \text{ is a positive multiple of $i$.}
\end{equation}
We know that there must be some eigenvalue $\lambda_+ < 0$ of the selfadjoint operator associated to the chord $x$, such that $\lim_{s \rightarrow -\infty} \exp(\lambda_+ s) \partial_s u_+$ is a nonzero multiple of the corresponding eigenvector. By \eqref{eq:partial-limit}, that eigenvalue can't be more negative than $\alpha_+$. But $\alpha_+$ is the highest negative eigenvalue, as one can see from \eqref{eq:first-eigenvalue}, bearing in mind the sign reversal from \eqref{eq:minus-complex}; hence $\lambda_+ = \alpha_+$. Moreover, that eigenvalue has multiplicity $1$, and the eigenvector is precisely $(\Xi_0,0)$. Hence, from \eqref{eq:partial-limit} one derives \eqref{eq:full-positive-limit}.

Given \eqref{eq:full-positive-limit}, and \eqref{eq:other-limit} applied to $u = u^2$, we are now in a position to bring to bear Corollary \ref{th:r-moves}. The outcome is that we can glue together the two pieces of our broken solution, which yields a family of solutions that describes a neighbourhood in the compactified parametrized moduli space; and all those glued solutions have $r \leq 0$. Hence, our broken solution could not, after all, have appeared as the limit of an end where $r>0$.
\end{proof}

\begin{cor} \label{th:simple-acyclic}
The following total complex, formed out of the previously constructed map and homotopy, is acyclic:
\begin{equation}
\label{eq:contractible-2}
\big\{\!
\xymatrix{ 
\ar@/_1pc/[rrrr]_-{\eqref{eq:our-homotopy}}
C^*_+ \ar[rr]^-{\eqref{eq:pm-continuation}} && C^*_- \ar[rr]^-{\text{projection}} && C^*_{-,\mathit{quot}}
}\!\big\}.
\end{equation}
\end{cor}

\begin{proof}
This is straightforward. Let's write down the pieces of the total complex:
\begin{equation} \label{eq:all-complex}
\xymatrix{
C^{*}_{+,\mathit{sub}} \ar[rr]^-{\eqref{eq:pm-continuation}}_-{\htp} &&
C^{*}_{-,\mathit{sub}} 
\\
\ar@/_2pc/[rrrr]_-{\eqref{eq:our-homotopy}} \ar[urr]_-{\eqref{eq:pm-continuation}}
C^{*}_{+,\mathit{quot}} \ar[u]^-{d_+} \ar[rr]_-{\eqref{eq:pm-continuation}} && 
\ar[u]_-{d_-}
C^*_{-,\mathit{quot}} \ar[rr]_-{\mathit{id}} && 
C^*_{-,\mathit{quot}}
\
}
\end{equation}
Here, we have represented all the maps involved, except for the parts of the Floer differential which take a given piece to itself; and we have used Lemmas \ref{th:geometric-continuation}, \ref{th:homotopy-vanishes}. Note that the top row of \eqref{eq:all-complex} is an acyclic subcomplex. After quotienting that out, the middle and right groups in the remaining line form an acyclic subcomplex as well; and quotienting that out leaves $C^*_{+,\mathit{quot}}$, which is itself acyclic by Lemma \ref{th:geometric-plus}.
\end{proof}

\subsection{Comparison with Morse theory}
The following discussion is not part of the necessary logical flow of the paper. It is designed as an additional explanation, by way of analogy, of what happens in our construction of the nullhomotopy. To do that, we return to the Morse-theoretic toy model mentioned in the Introduction. In particular, the notation here is different from the rest of the paper.

Let $E^n$ be a compact manifold with boundary $M = \partial E$. We fix a collar neighbourhood 
\begin{equation}
[0,2] \times M \hookrightarrow E, 
\end{equation}
where $\{0\} \times M = \partial E$. We will be exclusively using functions $f$ which are in split form on the collar, meaning that $f|[0,2] \times M$ is the sum of a function on $[0,2]$ and a function on $M$. For simplicity, we will work with a fixed metric throughout, which on the collar is also the product of the standard metric on $[0,2]$ and a metric on $M$.

Let's start with a Morse function $f_-$, such that $f_-(w,x) = (w-1)^2/2 + (\text{\it some function on $M$})$ on the collar. As a consequence of these particular choices, the Morse complex $C^*(f_-)$ comes with a projection map $C^*(f_-) \rightarrow C^*(f_-|\{1\} \times M)$. Now, take another function $f_+$, this time with $f_+(w,x) = -(w-1)^2/2 + (\text{\it some function on $M$})$ on the collar. This choice means that $C^*(f_+)$ computes the cohomology rel boundary. The two choices are joined by a continuation map 
\begin{equation}
C^*(f_+) \longrightarrow C^*(f_-), 
\end{equation}
defined using a family $f_{1,s}$ which equals $f_\pm$ for $\pm s \gg 0$, and such that
\begin{equation} \label{eq:partial-partial}
(\partial_s \partial_w f_{1,s}) \geq 0 \;\; \text{ at points $(w = 0, x) \in \partial E$, and for any $s$.} 
\end{equation}
The condition \eqref{eq:partial-partial} prevents solutions of the continuation map equation from reaching $\partial E$. We want to show that the composition
\begin{equation} \label{eq:desired-null}
C^*(f_+) \xrightarrow{\text{continuation}} C^*(f_-) \xrightarrow{\text{projection}} C^*(f_-|\{1\} \times M)
\end{equation}
is nullhomotopic. To do that, we extend the previous choice to a family $(f_{r,s})$ with an additional parameter $r \in [0,1]$, still having the same $\pm s \gg 0$ behaviour, and satisfying the analogue of \eqref{eq:partial-partial} for each value of $r$. We require:
\begin{align}
&
\mybox{At the parameter value $r = 0$, all the functions must be of the form $f_{0,s}(w,x) = 
(db_s/ds) (w-1)^2/2 + (\text{\it a function on $M$})$ near $\{1\} \times M$, with $b_s \in \bR$ such that $db_s/ds = \mp 1$ for $\pm s \gg 0$.} \\ &
\mybox{For small $r > 0$, we have $f_{r,s}(w,x) = (db_s/ds) (w-1)^2/2 - r c_s w + \text{\it (a function on $M$})$, again near $\{1\} \times M$. Here, the $b_s$ are as before; the $c_s$ are nonpositive for all $s$, zero for $|s| \gg 0$, and negative for some values of $s$.} \label{eq:small-r}
\end{align}
The point of \eqref{eq:small-r} is that for small $r>0$, we have $-\partial_w f_{r,s} = r c_s \leq 0$ along $\{1\} \times M$. Consider a solution of the associated continuation map equation,
\begin{equation} \label{eq:r-morse}
\left\{
\begin{aligned}
& u: \bR \longrightarrow E, \\
& du/ds = -\nabla f_{r,s}, \\
& \textstyle \lim_{s \rightarrow -\infty} u(s) \text{ is a critical point of $f_-$ in $\{1\} \times M$,} \\
& \textstyle \lim_{s \rightarrow +\infty} u(s) \text{ is a critical point of $f_+$.}
\end{aligned}
\right.
\end{equation}
Because of the nature of $f_-$, such a solution, for small $r>0$, necessarily has $u(s) \in \{1\} \times M$ for $s \ll 0$. But because of \eqref{eq:small-r}, when it eventually escapes $\{1\} \times M$ (which is must, since the $r c_s$ aren't identically zero), it has to do so into $[0,1) \times M$, and then it must remain in that subset for all greater values of $s$. But then, it's impossible for $\lim_{s \rightarrow +\infty} u(s)$ to be a critical point (for critical points in $\{1\} \times M$, that's because the structure of $f_+$ would force $u(s)$ to lie in $\{1\} \times M$ for $s \gg 0$). It follows that there are no solutions of \eqref{eq:r-morse} for small values of $r$, which means that we can use solutions for all $r \in (0,1]$ to build the desired nullhomotopy \eqref{eq:desired-null}.

The condition on $c_s$ in \eqref{eq:small-r} works, but is somewhat overkill. By an analysis of gluing in the parametrized moduli space, analogous to Section \ref{sec:gluing}, one could weaken it to
\begin{equation} \label{eq:define-constant}
C \stackrel{\text{def}}{=}\textstyle \int_{\bR} \exp(b_s) c_s \, \mathit{ds} < 0.
\end{equation}
While we don't want to go through the entire argument, one can explain the appearance of this condition as follows. Consider the linearized operator associated to solutions of the parametrized continuation map. More specifically, we consider a solution for $r=0$, and which is entirely contained in $\{1\} \times M$. Then, we can project to the $w$-component; the corresponding operator, for $\Upsilon(s) \in \bR$, is
\begin{equation}
D^{\mathit{para}}(R,\Upsilon) = d\Upsilon/ds + (db_s/ds) \Upsilon - c_s R.
\end{equation}
Explicit solutions $D^{\mathit{para}}(1,\Upsilon) = 0$ are 
\begin{equation}
\Upsilon(s) = \textstyle \exp(-b_s) \int_{-\infty}^s \exp(b_\sigma) c_\sigma \, \mathit{d\sigma}.
\end{equation}
These satisfy
\begin{equation}
\Upsilon(s) = \begin{cases}
0 & s \ll 0, \\
C \exp(s) (\text{\it times a positive constant}) & s \gg 0,
\end{cases}
\end{equation}
with $C$ as in \eqref{eq:define-constant}. Hence, the sign of $C$ precisely determines the sign of the leading coefficient of $\Upsilon$ as $s \rightarrow +\infty$. With that in mind, the requirement that $C<0$ is precisely analogous to the $\Gamma<0$ condition in \eqref{eq:nonzero}. As a final remark, note that our simpler Morse-theoretic argument, based on \eqref{eq:small-r}, relied strongly on the ODE nature of the gradient flow equation; hence (as far as this author knows) it has no Floer-theoretic counterpart.

\section{The main construction\label{sec:final}}
This section brings together all the pieces: first, combining material from Sections \ref{sec:fukaya} and \ref{sec:setup}, we define the (relative) Fukaya category $\scrA_q$. For elementary reasons, this comes with a restriction functor $\scrQ_q$ as in \eqref{eq:restriction-trans}. The definition of the natural transformation $\delta_q$ from \eqref{eq:serre-trans} is along the same lines, again leaning on the material from Section \ref{sec:fukaya}. After that, we tackle the construction of the nullhomotopy from Theorem \ref{th:main}. While based on the same idea as in Section \ref{sec:acyclic}, this requires more material from Section \ref{sec:plane}, as well as a generalization of the analysis of glued solutions from Section \ref{sec:gluing}; that generalization is the only new technical ingredient introduced in this section. Finally, filtered acyclicity follows from what we've already shown in Section \ref{sec:acyclic}.

\subsection{The relative Fukaya category and associated structures}
We continue to work in the geometric situation from Section \ref{subsec:context}, in particular with $(E, \Omega_E)$ and $p$ satisfying Setup \ref{th:e-setup}. Objects of the category $\scrA_q$ are Lagrangian submanifolds $L$ as in Setup \ref{th:l}, with their grading and {\em Spin} structure. We now mimic the construction of the relative Fukaya category from Section \ref{subsec:define-fukaya}, with the appropriate modifications which ensure that the restriction functor is well-defined.

For each pair of objects $(L_0,L_1)$, choose $(H,J) = (H_{L_0,L_1},J_{L_0,L_1})$ as in Setup \ref{th:setup-floer}. For generic choices, solutions of Floer's equation will be regular. We are only interested in solutions that avoid $\Omega_E$, and use those to define the differential on $\mathit{hom}_{\scrA}(L_0,L_1) = \mathit{CF}^*(L_0,L_1;H)$. Lemma \ref{th:approach} applies and therefore, as in \eqref{eq:decompose-1}, the Floer complex has a subcomplex, generated by chords lying in $p^{-1}(\{\mathrm{im}(w) > 1\})$; moreover, the induced differential on the quotient only counts Floer trajectories which remain inside $M$. We write $\mathit{hom}_{\scrB}(L_0,L_1)$ for the quotient, and $\scrQ: \mathit{hom}_{\scrA}(L_0,L_1) \rightarrow \mathit{hom}_{\scrB}(L_0,L_1)$ for the projection. 

Take a $(d+1)$-punctured disc $S$ with $m$ interior marked points, which represents a point of $\scrR^{d+1;m}$, together with a choice of objects $(L_0,\dots,L_d)$. On each such disc, we want to choose data as in Setup \ref{th:jk-adapted}, using the previous Floer data to determine asymptotics on the ends, and always requiring the ``suitable for restriction'' property from Setup \ref{th:jk-adapted}. The choices are subject to $\mathit{Sym}_m$-equivariance and consistency constraints, as in Section \ref{subsec:one-negative}. Note that the choice involves a one-form $\beta$ on $S$, which governs part of the behaviour of the inhomogeneous term near $M$. We ask that these one-forms be strictly compatible with gluing together surfaces for sufficiently long gluing lengths (this means that the choice of $\beta$ near the boundary of each $\bar\scrR^{d+1;m}$ is governed by the choices previously made for lower-dimensional parameter spaces). We also ask that gluing be compatible with the local linearity constraints, in the same sense as in \eqref{eq:local-ex}--\eqref{eq:local-ex-3}. We consider solutions of the associated equation \eqref{eq:cauchy-riemann-2}, imposing intersection conditions with $\Omega_E$ which follow the model of \eqref{eq:adjacency-conditions}. 

Transversality is an issue, specifically for solutions which remain inside $M$. The analogue of \eqref{eq:les-operator} for such a solution $u$ is
\begin{equation} \label{eq:les-operator-3}
\xymatrix{
\parbox{15em}{
$W^{2,2}(S,\bC, i\bR, \dots, i\bR)$ \newline \hspace*{6em} $\oplus$ \newline
$T\scrR^{d+1;m} \oplus W^{2,2}(S, u^*TM, $ \newline $u^*T(L_0 \cap M), \dots, u^*T(L_d \cap M); 
\newline \xi_1,\dots,\xi_m)$}
\ar[rr]^-{D \oplus \scrD_{M,u}} \ar[d]_-{\iso} 
&&
\parbox{12em}{
$W^{1,2}(S, \mathit{Hom}^{0,1}(TS,\bC))$ \newline \hspace*{3em} $\oplus$ \newline $W^{1,2}(S, \mathit{Hom}^{0,1}(TS,u^*TM))$} \ar[d]^-{\iso}
\\
\parbox{15em}{
$T\scrR^{d+1;m} \oplus W^{2,2}(S, u^*TE, u^*TL_0,$ \newline
\hspace*{3em} $\dots, u^*TL_d; \xi_1,\dots,\xi_m)$}
\ar[rr]^-{\scrD_{E,u}} &&
W^{1,2}(S, \mathit{Hom}^{0,1}(TS,u^*TE))
}
\end{equation}
The notation $W^{2,2}(\dots; \xi_1,\dots,\xi_m)$ means that we consider a finite-codimension subspace of the usual Sobolev space, which includes the linearization of the adjacency conditions \eqref{eq:adjacency-conditions}. These are pointwise conditions, and in order to accommodate them, we have increased the order of differentiability by one. The notation $\scrD$ stands for the linearized operators extended to include variations of the domain Riemann surface. The main reason why we still have the block diagonal form $\scrD_{E,u} = D \oplus \scrD_{M,u}$ is that the $W$-component of the Cauchy-Riemann equation is satisfied even if we change the Riemann surface structure (in other words, for the model \eqref{eq:plane-cr}, $u = 0$ is a solution for any $S$ and $\beta$), which implies that infinitesimal variations in $T\scrR^{d+1;m}$ only affect the $M$-component of the linearized operator.
%

\begin{remark} \label{th:choice-of-splitting}
One should really say that the domains of the $\scrD$ operators are extensions of $T\scrR^{d+1;m}$ (the tangent space of the moduli space at the point determined by our surface) by the relevant Sobolev spaces. Of course, there's no harm in choosing a splitting, and writing it as a direct sum as we have done. One convenient choice of splitting is to pick a finite-dimensional linear space of compactly supported infinitesimal deformations of the complex structure on $S$, and to leave $\xi_1,\dots,\xi_m$ unchanged (of course, a subspace of such deformations will then necessarily be equivalent, through suitable infinitesimal diffeomorphisms, to keeping the complex structure constant but moving the $\xi_i$).
\end{remark}

The operator $D$ is one of the standard ones from Section \ref{sec:plane} and has index $0$, hence is invertible (this statement remains true even though we have changed Sobolev spaces). Hence, once more, invertibility of $\scrD_{M,u}$ and $\scrD_{E,u}$ are equivalent; and we can use transversality inside $M$ to ensure that the first of those properties holds for our moduli spaces.

\begin{remark}
At this point, the conditions on almost complex structures say that each of the strata $\Omega_E$, $M$, and $\Omega_M = \Omega_E \cap M$ must be an almost complex submanifold. This means, for instance, that transversality for holomorphic spheres has to deal separately with the spheres contained in the various strata. This does not pose any difficulties, since the normal bundles of $M \subset E$ and $\Omega_M \subset \Omega_E$ are trivial (the normal bundles of $\Omega_M \subset M$ and $\Omega_E \subset E$ are not, but that has already been dealt with in Section \ref{sec:fukaya}).
\end{remark}

As in \eqref{eq:mu-q}, one uses (isolated points in) the resulting moduli spaces, together with the previous differential (as a constant $q^0$ term), to define a curved $A_\infty$-category structure $\mu_{\scrA_q}$, with morphism spaces $\mathit{hom}_{\scrA_q}(L_0,L_1) = \mathit{hom}_{\scrA}(L_0,L_1)[[q]]$. Lemma \ref{th:approach} still applies. Hence, setting $\hom_{\scrB_q}(L_0,L_1) = \hom_{\scrB}(L_0,L_1)[[q]]$, the $q$-linear extension of the previously mentioned projection $\scrQ$ gives  a curved $A_\infty$-functor $\scrQ_q: \scrA_q \rightarrow \scrB_q$. (We should reiterate that this is the very simplest kind of $A_\infty$-functor: it's a $q$-independent linear map on morphism spaces, with no other terms.) Moreover, the curved $A_\infty$-structure on $\scrB_q$ involves only solutions that lie entirely inside $M$, and is in fact a full subcategory of the relative Fukaya category of $(M,\Omega_M)$.

Our next task is to carry over the construction from Section \ref{subsec:diagonal-class} to the present context. We use the parameter spaces $\scrR^{k+1,l+1;m}$. This time, when choosing auxiliary data on the surfaces, we only require part (i) of Setup \ref{th:jk-adapted}. In particular, for generic choices, there won't be solutions contained in $M$. Given that, transversality issues are straightforward, and one ends up with a bimodule map $\delta_q$ as in \eqref{eq:serre-trans}.

\subsection{Constructing the nullhomotopy\label{subsec:gosh}}
Take a surface $S$ representing a point of $\scrR^{l+1,k+1;m}$, together with objects $(L_0,\dots,L_{k+1,l+1})$ of $\scrA_q$. For it, we choose data $(J_r,K_r)$ as in Setup \ref{th:jk-adapted}(i), depending on an additional parameter $r \in [0,1]$. For $r = 1$, these should reduce to those previously used to define $\delta_q$. They also need to satisfy equivariance and consistency constraints, of the usual kind. Moreover:

\begin{setup} \label{th:setup-gosh}
For each such surface $S$:
\begin{align}
& \mybox{At the parameter value $r = 0$, $(J_0,K_0)$ is``suitable for restriction'', in the sense of Setup \ref{th:jk-adapted}(ii). Moreover, as in the definition of $\scrA_q$, the underlying $\beta$ one-forms should be compatible with gluing together of surfaces.}
\\ \label{eq:dr-new}
& \mybox{Look at $(K_r,J_r)$ near $M = E_i$. At $r = 0$ we require that $\partial_r J_r = 0$, and $\partial_r K_r = -\gamma p^*(\mathrm{im}(w)) + (\text{constant})$, where $\gamma$ is as in \eqref{eq:gamma-form}. Moreover, the $\gamma$ should be compatible with gluing (meaning that if we glue together $S$ and some disc with one negative end, then for large values of the gluing parameter, the one-form on the glued surface is the obvious extension of $\gamma$).} 
\\
& 
\mybox{Our choices of $\beta$ and $\gamma$ give rise to one of the operators $D^{\mathit{para},\mu}$ discussed in Section \ref{subsec:parametrized-c}, and hence to a number $\Gamma = \Gamma(S,\beta,\gamma)$. We require that all these numbers should be negative. This is possible thanks to Proposition \ref{th:sign-of-gamma} (strictly speaking, in that Proposition we had no interior marked points, but the argument carries over without any problems).} \label{eq:negative-choice}
\\ 
& \mybox{Finally, we ask that our family should have a smooth extension for small $r<0$ (as before, this is a technical requirement).}
\end{align}
\end{setup}

As usual, transversality is straightforward except for solutions with $r = 0$ and which are contained in $M$. The analogue of Lemma \ref{th:double-transversality} is:

\begin{lemma} \label{th:double-tranverslity-2}
For generic choices, the following holds. Take a point in the parametrized moduli space, of the form $(r = 0, u: S \rightarrow M)$, for some $[S] \in \scrR^{k+1,l+1;m}$. Then
\begin{align} & \label{eq:double-2}
\mybox{$u$ is regular as a map to $M$ (with varying domain in $\scrR^{k+1,l+1;m}$), and $(0,u)$ is regular as a point in the parametrized moduli space of maps to $E$ (also with varying domains).} 
\\ & \label{eq:isolated-point-op}
\mybox{Suppose now that our point is isolated. Take the operator $\scrD_{E,(0,u)}^{\mathit{para}}$, which is the parametrized version of that in \eqref{eq:les-operator-3}. Let's think of its domain as in Remark \ref{th:choice-of-splitting}, so that it consists of triples $(R,I,\Upsilon)$, with $R \in \bR$ the variation of the parameter $r$; $I$ the infinitesimal change of the Riemann surface structure, which is always compactly supported; and $\Upsilon$ the infinitesimal deformation of the map $u$, respecting the adjacency conditions at the marked points. Then there is a one-dimensional space of solutions of $\scrD_{E,(0,u)}^{\mathit{para}}(R,I,\Upsilon) = 0$ such that $\Upsilon$ decays at all ends except for the $(k+1)$-st one, and with $\lim_{s \rightarrow -\infty} \exp(\alpha s)\Upsilon(\epsilon_{k+1}(s,t))$ a constant multiple of $(\Xi_0 = i,0)$, where $\alpha \in (0,\pi)$ is the angle associated to the relevant pair of Lagrangians. That space contains an element with $R>0$ and such that the multiplicative constant is negative.
}
\end{align}
\end{lemma}

\begin{proof}
The analogue of \eqref{eq:les-operator-3} is
\begin{equation} \label{eq:les-operator-4}
\xymatrix{
\parbox{17em}{
$\bR \oplus W^{2,2}(S,\bC, i\bR, \dots, i\bR)$ \newline \hspace*{6em} $\oplus$ \newline
$T\scrR^{k+1,l+1;m} \oplus W^{2,2}(S, u^*TM, $ \newline $u^*T(L_0 \cap M), \dots, u^*T(L_{k+l+1} \cap M); 
\newline \xi_1,\dots,\xi_m)$}
\ar[rr]^-{D^{\mathit{para}} \oplus \scrD_{M,u}} \ar[d]_-{\iso} 
&&
\parbox{12em}{
$W^{1,2}(S, \mathit{Hom}^{0,1}(TS,\bC))$ \newline \hspace*{3em} $\oplus$ \newline $W^{1,2}(S, \mathit{Hom}^{0,1}(TS,u^*TM))$} \ar[d]^-{\iso}
\\
\parbox{17em}{
$\bR \oplus T\scrR^{k+1,l+1;m} \oplus W^{2,2}(S, u^*TE,$ \newline
\hspace*{1em} $u^*TL_0, \dots, u^*TL_{k+1+1}; \xi_1,\dots,\xi_m)$}
\ar[rr]^-{\scrD_{E,(0,u)}^{\mathit{para}}} &&
W^{1,2}(S, \mathit{Hom}^{0,1}(TS,u^*TE))
}
\end{equation}
The block diagonal form of $\scrD_{E,(0,u)}^{\mathit{para}}$ relies crucially on \eqref{eq:dr-new}. As a consequence of that assumption, the $W$-component of the Cauchy-Riemann equation looks like \eqref{eq:plane-cr-2} up to terms of order $r^2$ or higher. When we linearize at $r = 0$, we get a term $R \gamma^{0,1}$, but the variation of the Riemann surface structure still does not appear in the linearization of the $W$-component.

As before, generic choices ensure that $\scrD_{M,u}$ is onto; and $D^{\mathit{para}}$ is invertible by Lemma \ref{th:its-an-isomorphism} and \eqref{eq:negative-choice}, ensuring that $\scrD_{E,(0,u)}^{\mathit{para}}$ will be onto. The proof of \eqref{eq:isolated-point-op} is similar to that of \eqref{eq:other-limit}, again using \eqref{eq:negative-choice}.
\end{proof}

We consider isolated points in our parametrized moduli space which have $r>0$, and whose limit over the $(k+1)$-st (negative) end lies in $M$. Counting such points in the usual way, as in \eqref{eq:mu-q}, yields operations
\begin{equation} \label{eq:h-operations}
\begin{aligned}
& h_q^{l,1,k}: \hom_{\scrA_q}(L_{k+l},L_{k+l+1}) \otimes \cdots \otimes \hom_{\scrA_q}(L_{k+1},L_{k+2}) 
\otimes \hom_{\scrA_q}(L_{k+1},L_k)^\vee[-n] \\ & \qquad \otimes 
\hom_{\scrA_q}(L_{k-1},L_k) \otimes \cdots \otimes \hom_{\scrA_q}(L_0,L_1) \longrightarrow
\hom_{\scrB_q}(L_0,L_{k+l+1})[-1-k-l].
\end{aligned}
\end{equation}

\begin{prop} \label{th:nullhomotopy-2}
The operations \eqref{eq:h-operations} define a nullhomotopy $h_q$ for the composition of $\delta_q$ and the projection, seen as an $\scrA_q$-bimodule map $\scrA_q^\vee[-n] \rightarrow \scrQ_q^*\scrB_q$.
\end{prop}

The overall strategy of proof follows that of Proposition \ref{th:nullhomotopy-1}, but the details are a little more complicated, and will involve re-working some of our earlier analytical observations in greater generality. The nullhomotopy relations are established through one-dimensional parametrized moduli spaces. In those spaces and their compactifications, the boundary points which give rise to the nullhomotopy relations are readily identified: they consist of the $r = 1$ solutions, plus the same kind of broken solutions (with added parameter) as when one proves that $\delta_q$ is a bimodule homomorphism. A priori, there can also be other undesirable contributions, and the main thrust of the proof consists of ruling those out.

First of all, it could happen that we have an end of a one-dimensional parametrized moduli space which can be compactified by a solution with $r = 0$. Note that, as an application of Lemma \ref{th:approach}, such a solution must be contained within $M$. In the situation of interest, its index in the parametrized moduli space is $1$, and the same holds for the unparametrized index when considered as a map to $M$. Because of this and the regularity statement \eqref{eq:double-2}, the one-parameter family which deforms such a solution lies in the $r = 0$ space and takes values in $M$, which means that it can't occur as an end of the moduli space we are interested in.

The remaining undesirables consist of broken solutions happening at $r = 0$. For the ends of one-dimensional moduli spaces, any limiting broken solution can have only two components, and must be one of the possibilities shown in Figure \ref{fig:4-bubbles}. In all cases, the limit at the leftmost negative $(-)$ end must lie in $M$. Whenever all components lie in $M$, the same argument as before applies, the conclusion being that gluing together the two pieces yields a family with $r = 0$. In the situations from Figure \ref{fig:4-bubbles}(i)--(iii), a repeated application of Lemma \ref{th:approach} shows that it is always the case that all components lie in $M$; (iv) requires more effort, and is the subject of the subsequent discussion.

\subsection{The gluing argument}
It is useful to set up some notation. 
\begin{equation} \label{eq:s1-surface}
\mybox{
We have a surface $S^1$ representing a point in $\scrR^{j+1;m^1}$, together with a choice of one of its $j$ positive ends; in notation consistent with \eqref{eq:quadratic-hochschild-differential}, that would be the $(i+j-k)$-th end. For brevity, we denote that end by $\epsilon^1$. Given some $\sigma \geq 0$, we write 
\[
S^1_{\sigma} = S^1 \setminus \epsilon_1( (\sigma,+\infty) \times [0,1]). 
\]
Our surface comes with a map $u^1: S^1 \rightarrow E$, whose image is not contained in $M$, but where $x = \lim_{s \rightarrow +\infty} u^1(\zeta^1(s,\cdot)) \in M$. The associated operator $\scrD_{u^1}$ is invertible  (from now on, we will no longer add the target manifold to the notation for linearized operators, since $E$ is always intended).
}
\end{equation}
Lemma \ref{th:approach} restricts the asymptotic behaviour of $u^1$ on $\epsilon^1$, with the outcome being the same as in \eqref{eq:full-positive-limit} (up to a sign difference due to the different conventions for the ends). Namely, let $\alpha \in (0,\pi)$ be the angle associated to the pair of Lagrangians that appear at that end. Then,
\begin{equation} \label{eq:full-positive-limit-2}
\mybox{
$\lim_{s \rightarrow +\infty} \exp(\alpha s) \partial_s (u^1(\epsilon^1(s,t)))$ 
 is a negative multiple of $(\Xi_0 = i,0) \in \bC \oplus \mathit{TM}_{x(t)} = \mathit{TE}_{x(t)}$.}
\end{equation}
Next,
\begin{equation} \label{eq:s2-surface}
\mybox{
We have a surface $S^2$ representing a point in $\scrR^{i+1,k+l-i-j+1;m^2}$. Denote the $(i+1)$-st (negative) end of that surface by $\epsilon^2$. Given some $\sigma \geq 0$, write 
\[
S^2_{\sigma} = S^2 \setminus \epsilon^2( (-\infty,-\sigma) \times [0,1]). 
\]
Our surface comes with a map $u^2: S^2 \rightarrow M$, which has the same limit $x$ on $\epsilon^2$. Moreover, $(0,u^2)$ is a regular isolated point in the parametrized moduli space, meaning that we are in the situation of \eqref{eq:isolated-point-op}.
}
\end{equation}
Gluing our two components together gives a family $(r_g, u_g: S_g \rightarrow E)$, depending on a gluing length $g \gg 0$. The surfaces $S_g$ describe a path in an appropriate space $\scrR^{k+1,l+1;m}$ ($m = m^1+m^2$) converging, as $g \rightarrow \infty$, to the point $([S^1], [S^2])$ in a codimension $1$ boundary stratum of $\bar\scrR^{k+1,l+1;m}$. More concretely, these surfaces can be obtained as follows. Take $S^1$ and $S^2$ as defined above, with the same interior marked points, but with modified complex structures $j^1_g$, $j^2_g$ which depend on $g$. Here, the change of complex structure is supported in a compact subset which is disjoint from the ends; and moreover, as $g \rightarrow \infty$, $j^\nu_g$ converges to the original complex structure exponentially fast. This convergence rate comes from the fact that $\exp(-\pi g)$ is a transverse coordinate to the boundary point in $\bar\scrR^{k+1,l+1;m}$. 
%
We chop off pieces of the ends to form $S^1_g$ and $S^2_g$ as in \eqref{eq:s1-surface}, \eqref{eq:s2-surface}, still carrying their $g$-dependent complex structures; and finally use the obviously identification $\epsilon^1(s,t) \sim \epsilon^2(s-g,t)$ to define $S_g$. For the glued maps $u_g$, we have an analogue of Lemma \ref{th:convergence-to-derivative-2} (which holds for the same reasons as before, plus the exponential convergence of $j_g^\nu$):

\begin{lemma} \label{th:convergence-to-derivative-3}
For any fixed $\sigma$, consider $S^1_\sigma \subset S_g$, $g \geq \sigma$. As $g \rightarrow \infty$, the derivative $\partial_g u_g|S^1_\sigma$ converges to zero. The same holds for $S^2_\sigma$.
\end{lemma}

If we apply that to $S^1_\sigma$ and translate the results back to the other piece of the surface, we get the following analogue of \eqref{eq:plus-derivative} (for the same reason as in \eqref{eq:full-positive-limit-2}, with reversed sign):
\begin{equation} \label{eq:plus-derivative-2}
\mybox{
Consider $(\partial_g u_g)(\epsilon^2(s-g,t))$, which is defined on $[0,g] \times [0,1]$. As $g \rightarrow \infty$, that converges on any compact subset to $(\partial_s u^1)(\epsilon^1(s,t))$.
}
\end{equation}
Take
\begin{equation} \label{eq:tangent-vector-2}
(R_g, I_g, \Upsilon_g) = \partial_g (r_g, j_g^2, u_g|S_g^2),
\end{equation}
which is an element in the kernel of the operator $\scrD^{\mathit{para}}_{(r_g,[S_g],u_g)}$ restricted to $S_g^2 \subset S_g$. Differentiation with respect to $g$ makes sense here, since all $S_g^2$ can also be considered as part of the fixed surface $S^2$. The analogue of Proposition \ref{th:rescaled-limit} is:

\begin{proposition} \label{th:rescaled-limit-3}
Take \eqref{eq:tangent-vector-2} for some sequence $g_k$ going to $\infty$, and assume that $(dr/dg)_{g_k} \neq 0$ for all $k$. Then, after rescaling by suitable positive constants $c_k$, a subsequence of the rescaled versions will converge on compact subsets to a nonzero solution 
\begin{equation}
\scrD_{(0,u^2)}^{\mathit{para}}(R,I,\Upsilon) = 0,
\end{equation}
and which has the following properties:
\begin{align}
& \label{eq:left-decay-3}
\mybox{$\Upsilon$ decays over all ends except for $\epsilon^2$.}
\\
& \label{eq:matching-3}
\mybox{$\lim_{s \rightarrow -\infty} \exp(\alpha s) \Upsilon(\epsilon^2(s,t))$ is a nonpositive multiple of $\Xi_0 = (i,0)$.}
\end{align}
\end{proposition}

To prove this, one chooses the $c_k$ so that, replacing subscripts $g_k$ everywhere with $k$,
\begin{equation}
\mybox{
$c_k^2 \big(R_k^2 + \|I_k\|^2_{W^{1,2}(S^2_k)} + \|\Upsilon_k\|^2_{W^{1,2}(S^2_k)}\big)$
 is bounded, and also bounded away from zero.}
\end{equation}
The argument is as in Proposition \ref{th:rescaled-limit}, with \eqref{eq:matching-3} coming from the combination of \eqref{eq:full-positive-limit-2} and \eqref{eq:plus-derivative-2} (again, the situation from Section \ref{subsec:modified} is actually the closer model). Having established that, we get the counterpart of Corollary \ref{th:r-moves}:

\begin{corollary} \label{th:r-moves-2}
For the glued family $(r_g,u_g)$, we must have $r_g \leq 0$ for all sufficiently large $g$.
\end{corollary}

Namely, if there are arbitrarily large $g$ with $r_g > 0$, then we can find a sequence $g_k$ such that $(\partial r_g/\partial g)_{g_k} < 0$. Then, the limit from Proposition \ref{th:rescaled-limit-3} would have $R \leq 0$, and the behaviour of $(R,I,\Upsilon)$ would contradict \eqref{eq:isolated-point-op}. 

Corollary \ref{th:r-moves-2} nearly completes the proof of Proposition \ref{th:nullhomotopy-2}. There is one remaining case which notationally doesn't quite fit in with the discussion above, namely where our broken limit consists of a surface $S^2$ representing a point in $\scrR^{k+1,l+1;m}$, and the other component, appearing at the $(k+1)$-st end of $S^2$, is a Floer trajectory. However, that case is in fact a straightforward generalization of the material in Section \ref{sec:gluing}, and hence does not deserve a separate discussion.

\subsection{Filtered acyclicity}
Attentive readers will have noticed that, in Section \ref{sec:acyclic}, we have imposed some quite detailed additional conditions on the pair of Lagrangian submanifolds $(L_0,L_1)$ (Setup \ref{th:l0l1}), as well as on the Floer data for both $(L_0,L_1)$ and $(L_1,L_0)$ (Setup \ref{th:plus-floer}, \ref{th:dual-h}). These conditions were used in our proof of Corollary \ref{th:simple-acyclic}. Our definition of the relative Fukaya category and its associated structures lacks those conditions. However, one can remedy that as follows.

Fix a pair of objects $(L_0,L_1)$ in $\scrA_q$. Let's temporarily enlarge our category by adding two more objects $(\bar{L}_0, \bar{L}_1)$ (always different, even if $L_0 = L_1)$. We denote the enlarged category by $\bar\scrA_q$ (and adopt similar notations, such as $\bar\scrB_q$, $\bar\scrQ_q$, \dots). Here, each $\bar{L}_k$ is chosen so as to be isotopic to $L_k$, within the class of Lagrangian submanifolds under consideration; and the new pair of objects should satisfy Setup \ref{th:l0l1}. When extending the curved $A_\infty$-structure to include those objects, we can choose the Floer data so that Setup \ref{th:plus-floer} is satisfied for $(\bar{L}_0,\bar{L}_1)$, and Setup \ref{th:dual-h} for $(\bar{L}_1,\bar{L}_0)$. Similarly, when defining $\bar\delta_q$, we can arrange that the linear $q = 0$ part $\bar\delta^{0,1,0}: \mathit{hom}_{\bar\scrA}(\bar{L}_1,\bar{L}_0)^\vee \rightarrow \mathit{hom}_{\bar\scrA}(\bar{L}_0,\bar{L}_1)$ is constructed as in Setup \ref{th:setup-cont}. Finally, we can arrange that the linear $q = 0$ piece of the nullhomotopy from Section \ref{subsec:gosh}, for $(\bar{L}_0,\bar{L}_1)$, reproduces the construction from Section \ref{subsec:simple-nullhomotopy}. Here, it is crucial that the requirements in Setup \ref{th:parametrized-setup} and \ref{th:setup-gosh} do not contradict each other.

As a consequence of this and Corollary \ref{th:simple-acyclic}, if we consider the bimodule \eqref{eq:contractible} for our enlarged category, and set $q = 0$, then that bimodule is acyclic when we specialize it to the pair of objects $(\bar{L}_0,\bar{L}_1)$. As a consequence of standard invariance properties of Floer cohomology, $\bar{L}_k$ is quasi-isomorphic to $L_k$ in $\bar\scrA$. By general algebraic properties of bimodules, the acyclicity property is then inherited by $(L_0,L_1)$. In other words, Corollary \ref{th:simple-acyclic} says that the total complex
\begin{equation}
\big\{\!
\xymatrix{ 
\ar@/_1pc/[rrrr]_-{\bar{h}^{0,1,0}}
\mathit{hom}_{\bar\scrA}(\bar{L}_0,\bar{L}_1)^\vee[-n]
\ar[rr]^-{\bar\delta^{0,1,0}} && \mathit{hom}_{\bar\scrA}(\bar{L}_0,\bar{L}_1) \ar[rr]^-{\text{projection}} && \mathit{hom}_{\bar\scrB}(\bar{L}_0,\bar{L}_1)
}\!\big\}
\end{equation}
is acyclic. Then, for elementary algebraic reasons, the same holds for
\begin{equation} \label{eq:original-complex}
\big\{\!
\xymatrix{ 
\ar@/_1pc/[rrrr]_-{h^{0,1,0}}
\mathit{hom}_{\scrA}(L_0,L_1)^\vee[-n]
\ar[rr]^-{\delta^{0,1,0}} && \mathit{hom}_{\scrA}(L_0,L_1) \ar[rr]^-{\text{projection}} && \mathit{hom}_{\scrB}(L_0,L_1)
}\!\big\},
\end{equation}
which now takes place inside the original (not enlarged) categorical framework.
Since we can carry out this construction (separately) for any pair $(L_0,L_1)$, all the complexes \eqref{eq:original-complex} are acyclic, which is precisely the statement of Theorem \ref{th:main}.


\begin{thebibliography}{10}

\bibitem{abouzaid12}
M.~Abouzaid.
\newblock Framed bordism and {L}agrangian embeddings of exotic spheres.
\newblock {\em Ann. of Math.}, 175:71--185, 2012.

\bibitem{abouzaid-seidel07}
M.~Abouzaid and P.~Seidel.
\newblock An open string analogue of {V}iterbo functoriality.
\newblock {\em Geom. Topol.}, 14:627--718, 2010.

\bibitem{charest-woodward15}
F.~Charest and C.~Woodward.
\newblock Floer theory and flips.
\newblock Preprint arXiv:1508.01573.

\bibitem{charest-woodward17}
F.~Charest and C.~Woodward.
\newblock Floer trajectories and stabilizing divisors.
\newblock {\em J. Fixed Point Theory Appl.}, 19:1165--1236, 2017.

\bibitem{cieliebak-mohnke07}
K.~Cieliebak and K.~Mohnke.
\newblock Symplectic hypersurfaces and transversality in {G}romov-{W}itten
  theory.
\newblock {\em J. Symplectic Geom.}, 5:281--356, 2007.

\bibitem{fooo}
K.~Fukaya, Y.-G. Oh, H.~Ohta, and K.~Ono.
\newblock {\em Lagrangian intersection {F}loer theory - anomaly and
  obstruction}.
\newblock Amer. Math. Soc., 2010.

\bibitem{ganatra12}
S.~Ganatra.
\newblock {\em Symplectic cohomology and duality for the wrapped {F}ukaya
  category}.
\newblock PhD thesis, MIT, 2012.

\bibitem{kobayashi-nomizu}
S.~Kobayashi and K.~Nomizu.
\newblock {\em Foundations of differential geometry}.
\newblock Wiley, 1963.

\bibitem{kontsevich-vlassopoulos13}
M.~Kontsevich and Y.~Vlassopoulos.
\newblock Pre-{C}alabi-{Y}au algebras and topological quantum field theories.
\newblock Unpublished manuscript, 2013.

\bibitem{mrowka88}
T.~Mrowka.
\newblock {\em A local {M}ayer-{V}ietoris principle for {Y}ang-{M}ills moduli
  spaces}.
\newblock PhD thesis, UC Berkeley, 1988.

\bibitem{perutz-sheridan20}
T.~Perutz and N.~Sheridan.
\newblock Constructing the relative {F}ukaya category. {I}.
\newblock Unpublished manuscript.

\bibitem{robbin-salamon02}
J.~Robbin and D.~Salamon.
\newblock Asymptotic behaviour of holomorphic strips.
\newblock {\em Ann. Inst. H. Poincar{\'e} Anal. Non Lin{\'e}aire}, 18:573--612,
  2001.

\bibitem{seidel04}
P.~Seidel.
\newblock {\em {F}ukaya categories and {P}icard-{L}efschetz theory}.
\newblock European Math. Soc., 2008.

\bibitem{seidel06}
P.~Seidel.
\newblock {S}ymplectic homology as {H}ochschild homology.
\newblock In {\em {A}lgebraic {G}eometry: {S}eattle 2005}, volume~1, pages
  415--434. Amer.\ Math.\ Soc., 2008.

\bibitem{seidel12b}
P.~Seidel.
\newblock Fukaya {$A_\infty$}-structures associated to {L}efschetz fibrations.
  {I}.
\newblock {\em J. Symplectic Geom.}, 10:325--388, 2012.

\bibitem{seidel14b}
P.~Seidel.
\newblock {F}ukaya {$A_\infty$}-structures associated to {L}efschetz
  fibrations. {II}.
\newblock In D.~Auroux, L.~Katzarkov, T.~Pantev, Y.~Soibelman, and
  Y.~Tschinkel, editors, {\em Algebra, {G}eometry and {P}hysics in the 21st
  {C}entury ({K}ontsevich {F}estschrift)}, volume 324 of {\em Progress in
  Math.}, pages 295--364. Birkh{\"a}user, 2017.

\bibitem{seidel15}
P.~Seidel.
\newblock Fukaya {$A_\infty$}-structures associated to {L}efschetz fibrations.
  {II} 1/2.
\newblock {\em Adv. Theor. Math. Phys.}, 20:883--944, 2016.

\bibitem{seidel16}
P.~Seidel.
\newblock Fukaya {$A_\infty$}-structures associated to {L}efschetz fibrations.
  {III}.
\newblock {\em J. Differential Geom.}, 117:485--589, 2021.

\bibitem{seidel17}
P.~Seidel.
\newblock Fukaya {$A_\infty$}-structures associated to {L}efschetz fibrations.
  {IV}.
\newblock In {\em Breadth in contemporary topology}, pages 195--276. Amer.
  Math. Soc., 2019.

\bibitem{seidel18}
P.~Seidel.
\newblock Fukaya {$A_\infty$}-structures associated to {L}efschetz fibrations.
  {IV} 1/2.
\newblock {\em J. Symplectic Geom.}, 18:291--332, 2020.

\bibitem{seidel18b}
P.~Seidel.
\newblock {F}ukaya {$A_\infty$}-structures associated to {L}efschetz fibrations. {VI}.
\newblock Preprint arXiv:1810.07119.

\bibitem{sheridan11b}
N.~Sheridan.
\newblock Homological mirror symmetry for {C}alabi-{Y}au hypersurfaces in
  projective space.
\newblock {\em Invent. Math.}, 199:1--186, 2015.

\bibitem{tradler01}
T.~Tradler.
\newblock Infinity-inner-products on {A}-infinity-algebras.
\newblock {\em J. Homotopy Related Struct.}, 3:245--271, 2008.

\bibitem{tradler-zeinalian07}
T.~Tradler and M.~Zeinalian.
\newblock Algebraic string operations.
\newblock {\em {$K$}-theory}, 38:59--82, 2007.

\end{thebibliography}

\end{document}